\documentclass[parskip=never,abstract=true]{scrartcl}

\usepackage[utf8]{inputenc}
\usepackage{csquotes}

\usepackage[T1]{fontenc}

\usepackage{lmodern}

\usepackage{yfonts}
\usepackage{textcomp}
\usepackage{setspace}
\usepackage{fullpage}

\usepackage{extarrows}
\usepackage{etoolbox}
\usepackage{xparse}

\usepackage{mathtools}
\usepackage{amsthm,thmtools}
\usepackage{dsfont}
\let\mathbb\mathds
\usepackage{amssymb}
\usepackage[scr]{rsfso} 
\usepackage{bm}
\usepackage{euscript}
\usepackage{amsbsy}
\DeclareMathAlphabet\mathbfcal{OMS}{cmsy}{b}{n}

\usepackage{microtype}
\usepackage{authblk}
\usepackage{tikz}
\usetikzlibrary{cd,quotes}
\usetikzlibrary{decorations.markings}
\usetikzlibrary{decorations.pathreplacing,calligraphy}
\usepackage{pgfplots}
\pgfplotsset{compat=1.13}
\usepackage{booktabs}
\usepackage{float}
\usepackage{environ}
\usepackage{xcolor,soul}
\usepackage{hyperref} 
\hypersetup{%
  colorlinks,%
  linkcolor={red!60!black},%
  citecolor={blue!50!black},%
  urlcolor={blue!80!black}%
}
\usepackage{enumitem,cleveref}
\usepackage{eso-pic}

\tikzset{curve/.style={settings={#1},to path={(\tikztostart)
    .. controls ($(\tikztostart)!\pv{pos}!(\tikztotarget)!\pv{height}!270:(\tikztotarget)$)
    and ($(\tikztostart)!1-\pv{pos}!(\tikztotarget)!\pv{height}!270:(\tikztotarget)$)
    .. (\tikztotarget)\tikztonodes}},
    settings/.code={\tikzset{quiver/.cd,#1}
        \def\pv##1{\pgfkeysvalueof{/tikz/quiver/##1}}},
    quiver/.cd,pos/.initial=0.35,height/.initial=0}

\tikzset{tail reversed/.code={\pgfsetarrowsstart{tikzcd to}}}
\tikzset{2tail/.code={\pgfsetarrowsstart{Implies[reversed]}}}
\tikzset{2tail reversed/.code={\pgfsetarrowsstart{Implies}}}
\tikzset{no body/.style={/tikz/dash pattern=on 0 off 1mm}}

\usepackage[normalem]{ulem}

\def\on{\operatorname}

\def\C{\EuScript{C}}

\def\DD{\bcat{D}}
\def\CC{\bcat{C}}
\def\XX{\bcat{X}}
\def\bS{\textbf{MB}}
\def\MS{\textbf{MS}}
\def\sS{\textbf{S}}

\def\OO{\mathbb{O}}

\def\Cat{\on{Cat}}

\def\Hom{\on{Hom}}

\def\Nsc{\on{N}^{\mathbf{sc}}}

\def\Set{\on{Set}}

\def\scsSet{{\on{Set}_{\Delta}^{\mathbf{sc}}}}

\def\mbsSet{{\Set_\Delta^{\mathbf{mb}}}}

\def\St{\on{St}}

\def\UN{\mathbb{U}\!\on{n}}
\def\sc{{\on{sc}}}

\DeclareMathOperator\SSt{\mathbb{S}t}

\declaretheoremstyle[bodyfont=\itshape,notefont=\bfseries]{abellanA}
\declaretheoremstyle[notefont=\bfseries]{abellanB}

\declaretheorem[style=abellanA,numberwithin=section,name={Theorem}]{theorem}

\declaretheorem[style=abellanA,numberlike=theorem,name={Lemma}]{lemma}
\declaretheorem[style=abellanB,numberlike=theorem,name={Definition}]{definition}
\declaretheorem[style=abellanB,numberlike=theorem,name={Definition-Lemma}]{definition-lemma}
\declaretheorem[style=abellanB,numberlike=theorem,name={Remark}]{remark}

\declaretheorem[style=abellanA,numberlike=theorem,name={Proposition}]{proposition}

\declaretheorem[style=abellanB,numbered=no,name={Notation}]{notation}
\declaretheorem[style=abellanA,numberlike=theorem,name={Corollary}]{corollary}

\newtheorem{thm*}{Theorem}

\newtheorem*{prop*}{Proposition}
\newtheorem{cor*}{Corollary}

\docsvlist{theorem,lemma,definition,remark,proposition,example,corollary}

\let\leq\leqslant
\let\geq\geqslant
\let\epsilon\varepsilon
\let\isom\simeq

\newcommand*\tensor{\otimes}

\newcommand*{\SB}{\bcat{S}}

\newcommand*\mathblank{\mathord{-}}

\def\Cat{\on{Cat}}
\def\scr{\EuScript}

\usepackage[bbgreekl]{mathbbol}

\let\emptyset\varnothing

\makeatletter
\newcommand{\fixed@sra}{$\vrule height 2\fontdimen22\textfont2 width 0pt\rightarrow$}
\newcommand{\shortarrowup}[1]{%
  \mathrel{\text{\rotatebox[origin=c]{65}{\fixed@sra}}}
}
\newcommand{\shortarrowdown}[1]{%
  \mathrel{\text{\rotatebox[origin=c]{250}{\fixed@sra}}}
}
\makeatother

\newcommand{\upslash}{\!\shortarrowup{1}}

\def\lra{\longrightarrow}
\def\lla{\longleftarrow}

\def\llra{\def\arraystretch{.1}\begin{array}{c} \lra \\ \lla \end{array}}

\def\op{{\on{op}}}

\makeatletter
\tikzcdset{
  open/.code     = {\tikzcdset{hook, circled};},
  open'/.code    = {\tikzcdset{hook', circled};},
  circled/.code  = {\tikzcdset{markwith = {\draw (0,0) circle (.375ex);}};},
  markwith/.code ={
    \pgfutil@ifundefined%
    {tikz@library@decorations.markings@loaded}%
    {\pgfutil@packageerror{tikz-cd}{You need to say %
        \string\usetikzlibrary{decorations.markings} to use arrows with markings}{}}{}%
    \pgfkeysalso{/tikz/postaction = {
        /tikz/decorate,
        /tikz/decoration={markings, mark = at position 0.5 with {#1}}}
    }
  },
}
\makeatother

\makeatletter
\DeclareSymbolFont{lettersA}{U}{txmia}{m}{it}

\DeclareRobustCommand*{\varmathbb}[1]{\gdef\F@ntPrefix{m@thbbch@r}%
	\@EachCharacter #1\@EndEachCharacter}

\long\def\DoLongFutureLet #1#2#3#4{%
	\def\@FutureLetDecide{#1#2\@FutureLetToken
		\def\@FutureLetNext{#3}\else
		\def\@FutureLetNext{#4}\fi\@FutureLetNext}
	\futurelet\@FutureLetToken\@FutureLetDecide}
\def\DoFutureLet #1#2#3#4{\DoLongFutureLet{#1}{#2}{#3}{#4}}
\def\@EachCharacter{\DoFutureLet{\ifx}{\@EndEachCharacter}%
	{\@EachCharacterDone}{\@PickUpTheCharacter}}
\def\m@keCharacter#1{\csname\F@ntPrefix#1\endcsname}
\def\@PickUpTheCharacter#1{\m@keCharacter{#1}\@EachCharacter}
\def\@EachCharacterDone \@EndEachCharacter{}

\DeclareMathSymbol{\m@thbbch@rA}{\mathord}{lettersA}{129}
\DeclareMathSymbol{\m@thbbch@rB}{\mathord}{lettersA}{130}
\DeclareMathSymbol{\m@thbbch@rC}{\mathord}{lettersA}{131}
\DeclareMathSymbol{\m@thbbch@rD}{\mathord}{lettersA}{132}
\DeclareMathSymbol{\m@thbbch@rE}{\mathord}{lettersA}{133}
\DeclareMathSymbol{\m@thbbch@rF}{\mathord}{lettersA}{134}
\DeclareMathSymbol{\m@thbbch@rG}{\mathord}{lettersA}{135}
\DeclareMathSymbol{\m@thbbch@rH}{\mathord}{lettersA}{136}
\DeclareMathSymbol{\m@thbbch@rI}{\mathord}{lettersA}{137}
\DeclareMathSymbol{\m@thbbch@rJ}{\mathord}{lettersA}{138}
\DeclareMathSymbol{\m@thbbch@rK}{\mathord}{lettersA}{139}
\DeclareMathSymbol{\m@thbbch@rL}{\mathord}{lettersA}{140}
\DeclareMathSymbol{\m@thbbch@rM}{\mathord}{lettersA}{141}
\DeclareMathSymbol{\m@thbbch@rN}{\mathord}{lettersA}{142}
\DeclareMathSymbol{\m@thbbch@rO}{\mathord}{lettersA}{143}
\DeclareMathSymbol{\m@thbbch@rP}{\mathord}{lettersA}{144}
\DeclareMathSymbol{\m@thbbch@rQ}{\mathord}{lettersA}{145}
\DeclareMathSymbol{\m@thbbch@rR}{\mathord}{lettersA}{146}
\DeclareMathSymbol{\m@thbbch@rS}{\mathord}{lettersA}{147}
\DeclareMathSymbol{\m@thbbch@rT}{\mathord}{lettersA}{148}
\DeclareMathSymbol{\m@thbbch@rU}{\mathord}{lettersA}{149}
\DeclareMathSymbol{\m@thbbch@rV}{\mathord}{lettersA}{150}
\DeclareMathSymbol{\m@thbbch@rW}{\mathord}{lettersA}{151}
\DeclareMathSymbol{\m@thbbch@rX}{\mathord}{lettersA}{152}
\DeclareMathSymbol{\m@thbbch@rY}{\mathord}{lettersA}{153}
\DeclareMathSymbol{\m@thbbch@rZ}{\mathord}{lettersA}{154}

\makeatother

\makeatletter
\newcommand{\myitem}[1]{%
  \item[#1]\protected@edef\@currentlabel{#1}%
}
\makeatother

\def\bcat{\varmathbb}

\title{On local fibrations of $(\infty,2)$-categories}
\author{Fernando Abellán}
\date{}
\begin{document}
  \maketitle
\begin{abstract}
  In this work we provide a model-independent notion of local fibrations of $(\infty,2)$-categories which generalises the well-known theory of locally coCartesian fibrations of $(\infty,1)$-categories. Based on previous work, we construct a model category which serves as a specific combinatorial model for this type of fibrations. Our main result is a generalisation of the locally coCartesian straightening and unstraightening construction of Lurie, which yields for any scaled simplicial set $S$ an equivalence of $(\infty,2)$-categories between the $(\infty,2)$-category of $(0,1)$-fibrations over $S$ (also known as inner coCartesian fibrations) and the $(\infty,2)$-category of functors $S \to \bcat{C}\!\on{at}_{(\infty,2)}$ with values in $(\infty,2)$-categories. Given an $(\infty,2)$-category $\bcat{B}$, our Grothendieck construction can be specialised to produce an equivalence between the $(\infty,2)$-category of local fibrations over $\bcat{B}$ and the $(\infty,2)$-category of oplax unital functors with values in $\bcat{C}\!\on{at}_{(\infty,2)}$. Finally, as an application of our results we provide a version of the Yoneda lemma for $(\infty,2)$-categories.
\end{abstract}
\tableofcontents
\newpage
  \section{Introduction}
  The theory of $(\infty,2)$-categories is enjoying in recent years a rapid and extensive development. Several fundamental constructions such as Gray tensor products (\cite{GHL_Gray}), partially lax colimits (\cite{Berman},\cite{GHL_LaxLim},\cite{AG_cof}) and a 2-dimensional theory of fibrations (\cite{GHL_Cart},\cite{AGS_CartI},\cite{AGS_CartII}) are already available and ready to be used in the study of homotopy coherent structures. Even more remarkably, we are starting to see specific examples of how these techniques can be used to categorify current areas of study and how they can put into perspective known constructions. 

  \begin{itemize}
    \item  In \cite{Global}, the authors give a new description of the $\infty$-category of global spectra as a certain partially lax colimit thus characterizing this $\infty$-category by means of a 2-dimensional universal property.
    \item In \cite{CDW}, the authors study complexes of stable $\infty$-categories using 2-categorical techniques and explore the first steps in the construction of a categorified theory of homological algebra.
  \end{itemize}

  However, not all foundational questions have been addressed yet and this paper aims at providing an additional piece of technology which will be relevant for future applications: A theory of local fibrations. Before delving into the 2-dimensional theory let us recall the notion of locally coCartesian fibration of $(\infty,1)$-categories.

  Let $p:X \to \scr{C}$ be a functor\footnote{We are giving a model-indepent definition here but it should be noted that by functor we mean a ``suficiently fibrant" map.} of $(\infty,1)$-categories. We say that $p$ is a $\emph{locally coCartesian fibration}$ (resp locally Cartesian) if for every morphism $e:\Delta^1 \to \scr{C}$ the restriction of $p$ along $e$,
  \[
    \begin{tikzcd}
      X \times_{\scr{C}}\{e\} \arrow[r] \arrow[d] & X \arrow[d,"p"] \\
      \Delta^{1} \arrow[r,"e"] & \scr{C}
    \end{tikzcd}
  \]
  is a coCartesian fibration (resp. Cartesian) over $\Delta^1$. In more informal terms, we can think of a locally coCartesian fibration as a way of enconding the following data:
  \begin{itemize}
    \item[i)] For every $c \in \scr{C}$ an $\infty$-category $X_c$.
    \item[ii)] For every morphism $u: c_0 \to c_1$ in $\scr{C}$ a functor $u_*:X_{c_0} \to X_{c_1}$.
    \item[iii)] For every commutative triangle in $\scr{C}$
    \[
       \begin{tikzcd}
         & c_1  \arrow[dr,"v"]& \\
         c_0 \arrow[ur,"u"] \arrow[rr,"w"] & & c_2
       \end{tikzcd}
     \] 
     a (not necessarily invertible) natural transformation $\eta: X_{c_0} \times \Delta^1 \to X_{c_2}$ between $w_* \Rightarrow{}v_* \circ u_*$.
  \end{itemize}
  This description can succesfully be formalised using the locally coCartesian straightening-unstraightening equivalence of Lurie (see \cite[Theorem 3.8.1]{LurieGoodwillie}) which shows that a locally coCartesian fibration over of $\scr{C}$ corresponds precisely to the data of an oplax unital functor $\St_{\scr{C}}(X): \scr{C} \to \bcat{C}\!\on{at}_\infty$ with values in $\infty$-categories. The goal of this paper is to generalise the previous discussion to the setting of $(\infty,2)$-categories. In order words, we will give a definition of local fibration of $(\infty,2)$-categories and show that such fibrations can be equivalently be described as oplax unital functors with values in the $(\infty,2)$-category of $(\infty,2)$-categories.

  In order to introduce our main definition we recall the notion of an $(i,j)$-fibration of $(\infty,2)$-categories.

  Let $p:\bcat{X} \to \bcat{C}$ be a functorof $(\infty,2)$-categories. We say that $p$ is an $(i,j)$-fibration where $i,j \in \{0,1\}$ if:
  \begin{itemize}
    \item[F1)]  For every $a, b \in \bcat{X}$ the induced map $\bcat{X}(a,b) \to \bcat{S}(p(a),p(b))$ on mapping $\infty$-categories is a coCartesian fibration if $j=0$ or a Cartesian fibration if $j=1$.
    \item[F2)] For every $a,b,c \in \bcat{X}$ the composition functors 
    \[
      \bcat{X}(a,b)\times \bcat{X}(b,c) \xlongrightarrow{} \bcat{X}(a,c)
    \]
    preserve coCartesian edges if $j=0$ (resp. Cartesian edges if $j=1$).
    \item[C1)] Let $i=0$. Given an object $a \in \bcat{X}$ and a morphism $e:p(a) \to y$ in $\bcat{S}$, then there exists an edge $\hat{e}:a \to \hat{y}$ over $e$ with the following property: For every $z \in \bcat{X}$ precomposition with $\hat{e}$ induces a pullback of $\infty$-categories
    \[
      \begin{tikzcd}
        \bcat{X}(\hat{y},z) \arrow[r] \arrow[d] & \bcat{X}(a,z) \arrow[d] \\
        \bcat{S}(y,p(z)) \arrow[r] & \bcat{S}(p(a),p(z)).
      \end{tikzcd}
    \]
    We say that $\hat{e}$ is a $(0,j)$-Cartesian lift of $e$. If $i=1$ one defines a dual condition (which generalises the $(\infty,1)$-notion of Cartesian edge) and obtains the definition of a $(1,j)$-Cartesian edge.
  \end{itemize}
  We would like to point out that $(0,j)$-fibrations appear in the literature (\cite{GHL_Cart},\cite{AGS_CartI}) under the name of outer coCartesian $j=0$ (resp. inner coCartesian if $j=1$) and similarly $(1,j)$-fibrations are called outer Cartesian fibrations if $j=0$ and inner Cartesian fibrations if $j=1$.

  In \cite{AGS_CartI},\cite{AGS_CartII} we gave a systematic analysis of the theory of $(1,0)$-fibrations and provided the corresponding Grothendieck construction which identifies this kind of fibrations with contravariant functors $\scr{F}:\bcat{S}^\op \to \bcat{C}\!\on{at}_{(\infty,2)}$. However, the models employed to study these fibrations in the aforementioned works do not generalise the theory of locally coCartesian fibrations developed in \cite{HTT} and \cite{LurieGoodwillie}. In order to realise the local theory of fibrations in the scaled simplicial model  (\cite{LurieGoodwillie}), we will be working in this paper with the $(0,1)$-variance. Needless to say, in model-independent terms these problems do not arise and we can give the following general definition.

   Let $p:\bcat{X} \to \bcat{S}$ be a functor. We say that $p$ is a local $(i,j)$-fibration if the following holds:
   \begin{itemize}
     \item Conditions F1) and F2) are satisfied. In this case, we say that $p$ is \emph{coCartesian}-enriched if $j=0$ and \emph{Cartesian}-enriched if $j=1$.
     \item For every morphism $e:\Delta^1 \to \bcat{S}$ the pullback $\bcat{X}\times_{\bcat{S}}\{e\} \xlongrightarrow{} \Delta^1$ is an $(i,j)$-fibration. We say that an edge $\hat{e}:\Delta^1 \to \bcat{X}$ is a local $(0,1)$-Cartesian edge if it is $(0,1)$-Cartesian after taking the corresponding pullback.
   \end{itemize}

Our first result in this paper uses the machinery of marked-biscaled $(\bS )$ simplicial sets constructed in \cite{AGS_CartI} to produce for each (not necessarily fibrant) scaled simplicial set $(S,T_S)$ a model structure whose fibrant objects are $(0,1)$-fibrations (see \autoref{thm:model} and \autoref{thm:comparisonscaled}) with respect to $(S,T_S)$. We would like to remark that for an arbitrary scaled simplicial set $(S,T_S)$ a $(0,1)$-fibration is described as a certain map of marked biscaled simplicial sets satisfying the right lifting property against the class of $\bS$-anodyne morphisms (see \autoref{def:mbsanodyne}). In particular, in the case where our base is not fibrant we cannot always give a model-indepent description of a $(0,1)$-fibration without having to consider the fibrant replacement of $(S,T_S)$ in the model structure of scaled simplicial sets.

\begin{thm*}\label{thmintro:model}
  Let $(S,T_S)$ be a scaled simplicial set. Then there exists a left proper combinatorial simplicial model structure on $\left(\on{Set}^{\mathbf{mb}}_\Delta\right)_{/(S,\sharp,T_S\subset \sharp)}$, which is characterized uniquely by the following properties:
  \begin{itemize}
    \item[C)] A morphism $f:X \to Y$ is a cofibration if and only if $f$ induces a monomorphism on the underlying simplicial sets.
    \item[F)] An object $p:X \to S$ is fibrant if and only if it is a $(0,1)$-fibration.
  \end{itemize}
  Moreover, if $S=\Delta^0$ this model structure is Quillen equivalent to Lurie's model structure (see Theorem 4.27 in \cite{LurieGoodwillie}) on $\on{Set}_\Delta^{\mathbf{sc}}$ , the  category of scaled simplicial sets 
\end{thm*}

In the case where $\bcat{S}=(S,T_S)$ models an $(\infty,2)$-category (where $T_S$ consists of those 2-simplices that represent commuting triangles in $\bcat{S}$) we can consider a subcollection of triangles $M_S \subseteq T_S$ depicted visually
\[\begin{tikzcd}
  & b \\
  a && c
  \arrow["u", from=2-1, to=1-2]
  \arrow["v", from=1-2, to=2-3]
  \arrow[""{name=0, anchor=center, inner sep=0}, from=2-1, to=2-3]
  \arrow["\simeq"', shorten <=3pt, Rightarrow, from=0, to=1-2]
\end{tikzcd}\] 
where either $u$ or $v$ are equivalences. In \autoref{thm:comparison} we specialize the previous result to $(S,M_S)$ to obtain a model-independent interpretation of the theory of $(0,1)$-fibrations over $(S,M_S)$.

\begin{thm*}\label{thmintro:modelindepnt}
  Let $(S,T_S)$ be a fibrant scaled simplicial set and let $p:X \to S$ be an object of $\left(\on{Set}^{\mathbf{mb}}_\Delta\right)_{/(S,\sharp,M_S\subset \sharp)}$. Then $p$ defines a fibrant object if and only if its associated functor  $p:\bcat{X} \to \bcat{S}$ of $(\infty,2)$-categories is a local $(0,1)$-fibration.
\end{thm*}

It is well known (see Proposition 2.4.2.8 in \cite{HTT}) that a locally coCartesian fibration whose locally coCartesian edge compose must be a coCartesian fibration. We extend this analysis to the $(\infty,2)$-categorical case by considering a fibrant scaled simplicial set $(S,T_S)$ and a collection of triangles $M_S \subseteq \mathcal{U} \subseteq T_S$ which allows us to make the following definition:
\begin{itemize}
  \item A local $(0,1)$-fibration $p:\bcat{X} \to \bcat{S}$ is said to be $\mathcal{U}$-local if local $(0,1)$-Cartesian edges compose along triangles lying over $\mathcal{U}$.
\end{itemize}
It then follows from \autoref{thm:ulocal} that $\mathcal{U}$-local fibrations can also be characterised as fibrant objects in our model structure.

\begin{thm*}
  Let $(S,T_S)$ be a fibrant scaled simplicial  and let $U$ be a collection of triangles such that $M_S \subseteq \mathcal{U} \subseteq T_S$. Then an object $p:X \to S$ is fibrant in $\left(\on{Set}^{\mathbf{mb}}_\Delta\right)_{/(S,\sharp,\mathcal{U}\subset \sharp)}$ if and only if its associated functor of $p:\bcat{X} \to \bcat{S}$ of $(\infty,2)$-categories is a local $(0,1)$-fibration which is in addition $\mathcal{U}$-local.
\end{thm*}

The results above can be interpreted as follows: Given an $(\infty,2)$-category modelled by a fibrant scaled simplicial set $(S,T_S)$ a $(0,1)$-fibration over the scaled simplicial set $(S,\mathcal{U})$ for $M_S \subseteq \mathcal{U} \subseteq T_S$ is a local $(0,1)$-fibration whose locally $(0,1)$-Cartesian edges have further composability properties dictated by the scaling $\mathcal{U}$.

Once the basics of the local theory of fibrations of $(\infty,2)$-categories are established we focus our attention into providing the expected Grothendieck construction which will allow us to interpret our fibrations as functors with values in $\bcat{C}\!\on{at}_{(\infty,2)}$ the $(\infty,2)$-category of $(\infty,2)$-categories. Since we are exclusively working with scaled simplicial sets in this paper we will chose an specific model for  $\bcat{C}\!\on{at}_{(\infty,2)}$ which we denote $\bcat{B}\!\on{icat}_{\infty}$: the $\infty$-bicategory of $\infty$-bicategories (i.e. fibrant scaled simplicial sets).

Let $S$ be a scaled simplicial and  denote by $\bcat{F}\!\on{ib}^{01}(S)$ the $\infty$-bicategory of $(0,1)$-fibrations over $S$\footnote{This $\infty$-bicategory is obtained from our model structure and thus depends on the scaling on $S$.} and let $\on{Fun}(S,\bcat{B}\!\on{icat}_\infty)$ the functor $\infty$-bicategory. Our main construction is a generalisation of the straightening-unstraightening equivalence of Lurie given in \cite{LurieGoodwillie} (where only fibrations with $\infty$-categorical fibres are considered) to the setting of $(0,1)$-fibrations whose fibres are $\infty$-bicategories. Combining \autoref{thm:mainun} and \autoref{rem:rehash} we obtain:

\begin{thm*}
  Let $S$ be a scaled simplicial set. Then the straightening-unstraightening adjunction 
  \[
    \SSt_{S}: \bcat{F}\!\on{ib}^{01}(S) \llra \on{Fun}(S,\bcat{B}\!\on{icat}_\infty): \UN_S
  \]
  yields an equivalence of $\infty$-bicategories between the $\infty$-bicategory of $(0,1)$-fibrations over $S$ and the $\infty$-bicategory of covariant functors with values in $\infty$-bicategories.
\end{thm*}

The general nature of the theorem above allows us to consider special cases which are of special interest. Namely, let $\bcat{S}=(S,T_S)$ be an $\infty$-bicategory and let $S_{\on{oplax}}=(S,M_S)$. We can now consider $\bcat{L}\!\on{Fib}(\bcat{S}):=\bcat{F}\!\on{ib}^{01}(S_{\on{oplax}})$ and apply our Grothendieck construction to obtain an equivalence between the $\infty$-bicategory of local $(0,1)$-fibrations over $\bcat{S}$ and the category $\on{Fun}(S_{\on{oplax}}, \bcat{B}\!\on{icat}_\infty)$. The later can be interpreted used the work of  Gagna-Harpaz-Lanari in \cite{GHL_Gray} as a model for the $\infty$-bicategory of oplax unital functors with values in $\infty$-bicategories. Indeed, given a map of scaled simplicial sets $\scr{F}:S_{\on{oplax}} \to \bcat{B}\!\on{icat}_\infty$  unraveling the definitions we obtain:
\begin{enumerate}
  \item For every objecet $s \in S$ an $\infty$-bicategory $\scr{F}(s)=X_s$.
  \item For every morphism $u:s_0 \to s_1$ in $S$ a functor $u_*: X_{s_0} \to X_{s_1}$.
  \item For every 2-simplex $\sigma:\Delta^2 \to S_{\on{oplax}}$,
  \[\begin{tikzcd}
  & s_1 \\
  s_0 \arrow[rr,"w",swap] && s_2
  \arrow["u", from=2-1, to=1-2]
  \arrow["v", from=1-2, to=2-3]
  \arrow[""{name=0, anchor=center, inner sep=0}, from=2-1, to=2-3]
  \arrow[ shorten <=3pt, Rightarrow, from=0, to=1-2]
\end{tikzcd}\] 
     a  natural transformation $\eta_\sigma: X_{c_0} \times \Delta^1 \to X_{c_2}$ between $w_* \Rightarrow{}v_* \circ u_*$ which is invertible whenever $\sigma \in M_S$.
\end{enumerate}
The fact that $\eta_\sigma$ is degenerate whenever $\sigma$ is degenerate tells us that $\scr{F}$ is an unital functor. Moreover, since in general $M_S \subseteq T_S$ it follows that $\scr{F}$ preserves composition only up to non-invertible 2-morphism. We further note, that since every invertible 2-morphism in $\bcat{S}$ belongs to $M_S$ it follows that for every pair of objects of $\bcat{S}$, $\scr{F}$ defines a functor between the associated mapping $\infty$-categories
\[
  \scr{F}_{(s,s')}: \bcat{S}(s,s') \xlongrightarrow{} \on{Fun}(X_s,X_{s'}).
\]

 Our main theorem then specialises (\autoref{cor:laxunitl}) to

\begin{cor*}
  Let $\bcat{S}$ be an $(\infty,2)$-category presented by a fibrant scaled simplicial set $(S,T_S)$. Then the straightening-unstraightening adjuction associated to the scaled simplicial set $(S,M_S)$
   \[
    \SSt^{\on{oplax}}_{\bcat{S}}: \bcat{L}\!\on{Fib}(\bcat{S}) \llra \on{Fun}^{\on{oplax}}(\bcat{S},\bcat{B}\!\on{icat}_\infty): \UN^{\on{oplax}}_{\bcat{S}}
  \]
  yields an equivalence of $\infty$-bicategories between the $\infty$-bicategory of local $(0,1)$-fibrations over $\bcat{S}$ and the $\infty$-bicategory of oplax unital functors with values in $\infty$-bicategories.
\end{cor*}

We expect that our fibrational approach will be a valuable tool to define oplax functors in situations when handling the coherence data might be too unwiedly. Additionally, the this variant of the Grothendieck construction will play a fundamental role in upcoming work \cite{AGH} where we will study a general calculus of mates for $(\infty,2)$-categories.

We conclude this section by extending the previous result to the $\mathcal{U}$-local case. Given a collection of triangles $M_S \subseteq \mathcal{U} \subseteq T_S$ we define $(S,\mathcal{U})=S_{\mathcal{U}}$ together with $\bcat{L}\!\on{Fib}(\bcat{S})^{\mathcal{U}}=\bcat{F}\!\on{ib}^{01}(S_{\mathcal{U}})$. Again, our main theorem specialises to:

\begin{cor*}
  Let $\bcat{S}$ be an $(\infty,2)$-category presented by a fibrant scaled simplicial set $(S,T_S)$. Given a collection of thin triangles $M_S \subseteq \mathcal{U} \subseteq T_S$, then the straightening-unstraightening adjuction associated to the scaled simplicial set $(S,\mathcal{U})$
   \[
    \SSt^{\mathcal{U}}_{\bcat{S}}: \bcat{L}\!\on{Fib}(\bcat{S})^{\mathcal{U}} \llra \on{Fun}^{\mathcal{U}}(\bcat{S},\bcat{B}\!\on{icat}_\infty): \UN^{\mathcal{U}}_{\bcat{S}}
  \]
  yields an equivalence of $\infty$-bicategories between the $\infty$-bicategory of local $(0,1)$-fibrations over $\bcat{S}$ which are in addition $\mathcal{U}$-local and the $\infty$-bicategory of oplax unital functors with values in $\infty$-bicategories which preserve composition along triangles in $\mathcal{U}$.
\end{cor*}

\subsection*{The \texorpdfstring{$\infty$-}-bicategorical Yoneda lemma}
As an application of our results we give a fibrational proof of the Yoneda lemma for $\infty$-bicategories. We would like to stress that such a result is already present in the literature in the work of Hinich \cite{Hin} where a general Yoneda lemma for enriched $\infty$-categories is established.

In this work we provide a proof of the Yoneda lemma as a direct application of the Grothendieck construction. Given an $\infty$-bicategory $\bcat{C}$ we consider the $(0,1)$-fibration $\on{ev}_1:\on{Fun}^{\on{gr}}(\Delta^1,\bcat{C}) \to \bcat{C}$ (see Theorem 2.2.6 in \cite{GHL_Cart} for more details) where $\on{Fun}^{\on{gr}}(\Delta^1,\bcat{C})$ is the category of functors and lax natural transformations also known as the lax arrow category of $\bcat{C}$.

Carefully unwinding this construction reveals that the fibres of this functor come equipped with maps 
\[
  \bcat{C}_{\upslash c}=\on{Fun}^{\on{gr}}(\Delta^1,\bcat{C}) \times_{\bcat{C}}\{c\} \xlongrightarrow{}\bcat{C}
\]
which are in turn $(1,0)$-fibrations with $\infty$-categorical fibres. We can describe the $\infty$-bicategory $ \bcat{C}_{\upslash c}$ as follows
\begin{itemize}
  \item The objects of $\bcat{C}_{\upslash c}$ are given by morphisms $f:x \to c$ in $\bcat{C}$ whose target is $c$.
  \item A morphism from $f:x \to c$ to $g:y \to c$ is given by a laxly commutative diagram
  \[\begin{tikzcd}
  x && y \\
  & c
  \arrow[""{name=0, anchor=center, inner sep=0}, "f"', from=1-1, to=2-2]
  \arrow["\alpha", from=1-1, to=1-3]
  \arrow["g", from=1-3, to=2-2]
  \arrow[shorten <=12pt, shorten >=12pt, Rightarrow, from=0, to=1-3]
\end{tikzcd}\]
\item A 2-morphism from $\alpha: x \to y$ to $\beta:x \to y$ is given by a 2-morphism $\theta: \alpha \Rightarrow \beta$ making the obvious diagram commute.
\end{itemize}

Applying the Grothendieck construction in two steps we obtain a functor $\mathcal{Y}: \bcat{C} \to \on{Fun}(\bcat{C}^\op,\bcat{C}\!\on{at}_\infty)$. Moreover, it follows from previous work (Theorem 3.17 in \cite{AScof}) that $\bcat{C}_{\upslash c}$ corresponds under the Grothendieck construction to the representable functors $\bcat{C}(\mathblank,c)$. We prove in \autoref{thm:yoneda} the final result of this paper.

\begin{thm*}
   For every $\infty$-bicategory $\bcat{C}$ the Yoneda embedding
    \[
      \mathcal{Y}:\bcat{C} \xlongrightarrow{} \on{Fun}(\bcat{C}^{\op},\bcat{C}\!\on{at}_\infty), \enspace c \mapsto \bcat{C}(\mathblank,c)
    \]
   is fully-faithful. Moreveover, given a functor $\scr{F}: \bcat{C}^\op \to \bcat{C}\!\on{at}_\infty$ there is a equivalence of $\bcat{C}\!\on{at}_\infty$-valued functors
    \[
      \on{Nat}_{\bcat{C}^{\op}}(\mathcal{Y}(\mathblank),\scr{F}) \xRightarrow{\simeq} \scr{F}
    \]
    (where $\on{Nat}_{\bcat{C}^\op}(\mathblank,\mathblank)$ is the mapping $\infty$-category in  $\on{Fun}(\bcat{C}^{\op},\bcat{C}\!\on{at}_\infty)$), which is natural in $\scr{F}$.
\end{thm*}

\subsection*{Acknowledgments}
The author would like to thank Rune Haugseng for the helpful discussions regarding the model independent definition of local fibrations.
\section{Preliminaries}
In thisi section we will mainly gather the main definitions of the theory of scaled simplicial sets as presented by Lurie in \cite{LurieGoodwillie}. 
\begin{definition}\label{def:scaledsimplicialset}
  A scaled simplicial set $(X,T_X)$ consists of a simplicial set $X$ together with a collection of $2$-simplices (also called triangles) $T_X$, which contains \emph{every degenerate} triangle. We call the elements of $T_X$ the \emph{thin triangles} of $X$. A morphism of scaled simplicial sets $f:(X,T_X) \to (Y,T_Y)$ is a map of simplicial sets $f: X \to Y$ such that $f(T_X)\subseteq T_Y$. We denote the corresponding category of scaled simplicial sets by $\on{Set}_\Delta^{\mathbf{sc}}$.
\end{definition}

\begin{remark}
  When no confusion shall arise we will omit the pair notationand simply denote the scaled simplicial  $(X,T_X)$ set as  $X$.
\end{remark}

\begin{notation}
  Given a simplicial set $A$ we have two canonical ways of viewing it as a scaled simplicial set:
  \begin{itemize}
    \item We define $A_\flat=(A,\flat)$ where $\flat$ is the collection consisting only in the degenerate triangles of $A$.
    \item We define $A_\sharp=(A,\sharp)$ where $\sharp$ is the collection consisting in every triangle of $A$.
  \end{itemize}
\end{notation}

  \begin{definition}\label{def:scanodyne}
  The set of \emph{generating scaled anodyne maps} \(\sS\) is the set of maps of scaled simplicial sets consisting of:
  \begin{enumerate}
    \item[(i)]\label{item:anodyne-inner} the inner horns inclusions
    \[
    \bigl(\Lambda^n_i,\{\Delta^{\{i-1,i,i+1\}}\}\bigr)\rightarrow \bigl(\Delta^n,\{\Delta^{\{i-1,i,i+1\}}\}\bigr)
    \quad , \quad n \geq 2 \quad , \quad 0 < i < n ;
    \]
    \item[(ii)]\label{i:saturation} the map 
    \[
    (\Delta^4,T)\rightarrow (\Delta^4,T\cup \{\Delta^{\{0,3,4\}}, \ \Delta^{\{0,1,4\}}\}),
    \]
    where we define
    \[
    T\overset{\text{def}}{=}\{\Delta^{\{0,2,4\}}, \ \Delta^{\{ 1,2,3\}}, \ \Delta^{\{0,1,3\}}, \ \Delta^{\{1,3,4\}}, \ \Delta^{\{0,1,2\}}\};
    \]
    \item[(iii)]\label{item:anodyne_outer} the set of maps
    \[
    \Bigl(\Lambda^n_0\coprod_{\Delta^{\{0,1\}}}\Delta^0,\{\Delta^{\{0,1,n\}}\}\Bigr)\rightarrow \Bigl(\Delta^n\coprod_{\Delta^{\{0,1\}}}\Delta^0,\{\Delta^{\{0,1,n\}}\}\Bigr)
    \quad , \quad n\geq 3.
    \]
  \end{enumerate}
  A general map of scaled simplicial set is said to be \emph{scaled anodyne} if it belongs to the weakly saturated closure of \(\sS\).
\end{definition}

\begin{definition}
  A scaled simplicial set $X$ is said to be an $\infty$-bicategory if it has the right lifting property again the scaled of scaled anodyne maps in \autoref{def:scanodyne}. In this case, we view the 2-simplices of $T_X$ as the collection of commuting triangles.
\end{definition}

\begin{definition}
  We denote by $\Cat_\Delta^+$ the category of $\on{Set}_\Delta^+$-enriched categories  (i.e. categories enriched in marked simplicial sets). We note that we can view the category of (strict) 2-categories $2\Cat$ as a full subcategory of $\Cat_\Delta^+$ by applying the nerve functor Hom-wise and marking the equivalences in each mapping category.
\end{definition}

\begin{definition}\label{def:OI}
  Let $I$ be a linearly ordered finite set. We define a $2$-category $\OO^{I}$ as
  follows
  \begin{itemize}
    \item the objects of $\OO^I$ are the elements of $I$,
    \item the category $\mathbb{O}^{I}(i,j)$ of morphisms between objects $i,j \in I$ is defined
    as the poset of finite sets $S \subseteq I$ such that $\min(S)=i$ and $\max(S)=j$
    ordered by inclusion,
    \item the composition functors are given, for $i,j,l\in I$, by
    \[
    \mathbb{O}^{I}(i,j) \times \mathbb{O}^{I}(j,l) \to \mathbb{O}^{I}(i,l), \quad (S,T) \mapsto S \cup T.
    \]
  \end{itemize}
  When $I=[n]$, we denote $\OO^I$ by $\OO^n$. Note that the $\OO^n$ form a cosimplicial object in $2\!\Cat$, which we denote by $\OO^\bullet$. 
\end{definition}

\begin{definition}\label{def:rigidification}
  The map
  \[
  \begin{tikzcd}
    \Delta \arrow[r,"\mathfrak{C}"] & {\Cat_\Delta^+} 
  \end{tikzcd}
  \]
  which sends $[n]$ to $\OO^n$ gives us a cosimplicial object in $\Cat_\Delta^+$. We can moreover send the thin 2-simplex $\Delta^2_\sharp$ to $\mathfrak{C}[\Delta^2]$ equipped with maximally-marked mapping spaces. The usual machinery of nerve and realization then gives us adjoint functors  
  \[
  \begin{tikzcd}
    \mathfrak{C}^{\sc}: &[-3em] \scsSet \arrow[r,shift left] & \Cat_\Delta^+ \arrow[l,shift left] &[-3em] :\Nsc 
  \end{tikzcd}
  \]
  which we will call the \emph{scaled nerve} and \emph{scaled rigidification}. 
\end{definition}

\begin{theorem}\label{thm:scaledmodel}
  There is a left proper, combinatorial model structure on $\scsSet$ with
  \begin{itemize}
    \item[W)] The weak equivalences are the morphisms $f:A\to B$ such that $\mathfrak{C}^{\sc}[f]:\mathfrak{C}^{\sc}[A]\to \mathfrak{C}^{\sc}[B]$ is an equivalence in $\Cat_\Delta^+$. 
    \item[C)] The cofibrations are the monomorphisms. 
  \end{itemize}
  Moreover, the fibrant objects in this model structure are the $\infty$-bicategories, and the adjunction 
  \[
  \begin{tikzcd}
    \mathfrak{C}^{\sc}: &[-3em] \scsSet \arrow[r,shift left] & \Cat_\Delta^+ \arrow[l,shift left] &[-3em] :\Nsc 
  \end{tikzcd}
  \]
  is a Quillen equivalence.
\end{theorem}
\begin{proof}
  This is \cite[Thm A.3.2.4]{LurieGoodwillie}. The characterization of fibrant objects is \cite[Thm 5.1]{GHL_Equivalence}.
\end{proof}

\begin{definition}
  We say that a map of scaled simplicial sets is a \emph{bicategorical equivalence} if it is a weak equivalence in the model structure given in \autoref{thm:scaledmodel}. Similarly, call the fibrations in the model structure of scaled simplicial sets \emph{bicategorical fibrations}.
\end{definition}

\begin{definition}
  Given a pair of scaled simplicial sets $X,Y$ we denote by $\on{Fun}(X,Y)$ the scaled simplicial set determined by the universal property
  \[
    \Hom_{\on{Set}_\Delta^{\mathbf{sc}}}(K,\on{Fun}(X,Y))\isom  \Hom_{\on{Set}_\Delta^{\mathbf{sc}}}(K \times X, Y)
  \]
  where $K\times X$ denotes the Cartesian product of scaled simplicial sets.
\end{definition}

\begin{definition}\label{def:mappingcat}
  Let $\bcat{C}$ be an $\infty$-bicategory. Given an object $y \in \bcat{C}$ we define a scaled simplicialset $\bcat{C}_{\upslash y}$  as follows. The data of an $n$-simplices $\Delta^n \to \bcat{C}_{\upslash y}$ is given by a map $\sigma:\Delta^{n+1} \to \bcat{C}$ such that $\sigma(n+1)=y$. The inclusion $d_{n+1}: \Delta^n \to \Delta^{n+1}$ induces a map 
  \[
    \pi:\bcat{C}_{\upslash y} \xlongrightarrow{} \bcat{C}
  \]
  which we use to declare a triangle in $\bcat{C}_{\upslash y}$ to be thin if and only if its image unde $\pi$ is thin in $\bcat{C}$. It follows from \cite[Prop. 2.33]{GHL_Equivalence} that the fibre of $\pi$ at an object $x \in \bcat{C}$ is a model for $\bcat{C}(x,y)$, the mapping $\infty$-category.
\end{definition}

  \section{The model structure}

  \begin{definition}
  A \emph{marked biscaled} simplicial set ($\bS$ simplicial set) is given by the following data
  \begin{itemize}
    \item A simplicial set $X$.
    \item A collection of edges  $E_X \in X_1$ containing all degenerate edges.
    \item A collection of triangles $T_X \in X_2$ containing all degenerate triangles. We will refer to the elements of this collection as \emph{thin triangles}.
    \item A collection of triangles $C_X \in X_2$ such that $T_X \subseteq C_X$. We will refer to the elements of this collection as \emph{lean triangles}.
  \end{itemize}
  We will denote such objects as triples $(X,E_X, T_X \subseteq C_X)$. A map $(X,E_X, T_X \subseteq C_X) \to (Y,E_Y,T_Y \subseteq C_Y)$ is given by a map of simplicial sets $f:X \to Y$ compatible with the collections of edges and triangles above. We denote by $\mbsSet$ the category of mb simplicial sets.
\end{definition}

\begin{notation}
  Let $(X,E_X, T_X \subseteq C_X)$ be a mb simplicial set. Suppose that the collection $E_X$ consist only of degenerate edges. Then we fix the notation $(X,E_X, T_X \subseteq C_X)=(X,\flat,T_X \subseteq E_X)$ and do similarly for the collection $T_X$. If $C_X$ consists only of degenerate triangles we fix the notation $(X,E_X, T_X \subseteq C_X)=(X,E_X, \flat)$. In an analogous fashion we wil use the symbol “$\sharp$“ to denote a collection containing all edges (resp. all triangles). Finally suppose that $T_X=C_X$ then we will employ the notation $(X,E_X,T_X)$.
\end{notation}

\begin{remark}
  We will often abuse notation when defining the collections $E_X$ (resp. $T_X$, resp. $C_X$) and just specified its non-degenerate edges (resp. triangles).
\end{remark}

  \begin{definition}\label{def:mbsanodyne}
  The set of \emph{generating mb anodyne maps} \(\bS\) is the set of maps of mb simplicial sets consisting of:
  \begin{enumerate}
    \myitem{(A1)}\label{mb:innerhorn} The inner horn inclusions 
    \[
    \bigl(\Lambda^n_i,\flat,\flat \subset \{\Delta^{\{i-1,i,i+1\}}\}\bigr)\rightarrow \bigl(\Delta^n,\flat,\flat \subset \{\Delta^{\{i-1,i,i+1\}}\}\bigr)
    \quad , \quad n \geq 2 \quad , \quad 0 < i < n ;
    \]
    These maps force left-degenerate lean-scaled triangles to represent Cartesian edges of the mapping category.
    \myitem{(A2)}\label{mb:wonky4} The map 
    \[
    (\Delta^4,\flat,\flat \subset T)\rightarrow (\Delta^4,\flat,\flat \subset T\cup \{\Delta^{\{0,3,4\}}, \ \Delta^{\{0,1,4\}}\}),
    \]
    where we define
    \[
    T\overset{\text{def}}{=}\{\Delta^{\{0,2,4\}}, \ \Delta^{\{ 1,2,3\}}, \ \Delta^{\{0,1,3\}}, \ \Delta^{\{1,3,4\}}, \ \Delta^{\{0,1,2\}}\};
    \]
    \myitem{(A3)}\label{mb:2coCartesianmorphs} The set of maps
    \[
    \Bigl(\Lambda^n_0,\{\Delta^{\{0,1\}}\}, \{ \Delta^{\{0,1,n\}} \}\Bigr) \to \Bigl(\Delta^n,\{\Delta^{\{0,1\}}\}, \{ \Delta^{\{0,1,n\}} \}\Bigr) \quad , \quad n \geq 2.
    \]
    This forces the marked morphisms to be $p$-coCartesian with respect to the given thin triangles. 
    \myitem{(A4)}\label{mb:2CartliftsExist} The inclusion of the initial vertex
    \[
    \Bigl(\Delta^{0},\sharp,\sharp \Bigr) \rightarrow \Bigl(\Delta^1,\sharp,\sharp \Bigr).
    \]
    This requires $p$-coCartesian lifts of morphisms in the base to exist.
    \myitem{(S1)}\label{mb:composeacrossthin} The map
    \[
    \Bigl(\Delta^2,\{\Delta^{\{0,1\}}, \Delta^{\{1,2\}}\},\sharp \Bigr) \rightarrow \Bigl(\Delta^2,\sharp,\sharp \Bigr),
    \]
    requiring that $p$-coCartesian morphisms compose across thin triangles.
    \myitem{(S2)}\label{mb:coCartoverThin} The map
    \[
    \Bigl(\Delta^2,\flat,\flat \subset \sharp \Bigr) \rightarrow \Bigl( \Delta^2,\flat,\sharp\Bigr),
    \]
    which requires that lean triangles over thin triangles are, themselves, thin.
    \myitem{(E)}\label{mb:equivalences} For every Kan complex $K$, the map
    \[
    \Bigl( K,\flat,\sharp  \Bigr) \rightarrow \Bigl(K,\sharp, \sharp\Bigr).
    \]
    Which requires that every equivalence is a marked morphism.
  \end{enumerate}
  A map of mb simplicial sets is said to be \bS-anodyne if it belongs to the weakly saturated closure of \bS.
\end{definition}

\begin{remark}
  In previous work \cite{AGS_CartI} we also introduced the notion of a \bS-anodyne map when dealing with the theory of $(1,0)$-fibrations. We would like to point out that both classes of maps are different but we are using the same name to avoid the overly cumbersome notation $\bS^{01}$-anodyne.
\end{remark}

\begin{definition}
  A map of $\bS$ simplicial sets is said to be an \bS-\emph{fibration} if it has the right lifting property against the class of \bS-anodyne maps.
\end{definition}

\begin{lemma}
  Let $p: X \to S$ be an \bS-fibration then for every $s \in S$ the fibre over $s$,
  \[
    \begin{tikzcd}
      X_s \arrow[r] \arrow[d] & X \arrow[d,"p"] \\
      \Delta^0 \arrow[r,"s"] & S
    \end{tikzcd}
  \]
  is an $\infty$-bicategory where the marked edges are precisely the equivalences.
\end{lemma}
\begin{proof}
  Observe that in $X_s$ the lean and thin triangles coincide since in a \bS-fibration a lean triangle lying over a thin triangle is itself thin. It follows that $X_s$ has the right lifting property against the class of scaled anodyne maps and thus it is an $\infty$-bicategory. Note that by definition the equivalences must be marked in $X_s$. Moreover, since $X_s$ lifts agains the class of maps \ref{mb:2coCartesianmorphs} one checks easily that marked morphisms are equivalences.
\end{proof}

\begin{lemma}
  The morphism of $\bS$-simplicial sets $(\Delta^2,\{ \Delta^{\{0,1\}},\Delta^{\{0,2\}} \} ,\sharp) \to (\Delta^2,\sharp,\sharp)$ is $\bS$-anodyne.
\end{lemma}
\begin{proof}
  The proof is dual to \cite[Lemma 3.11]{AGS_CartI}.
\end{proof}

\begin{lemma}
  The morphism of $\bS$ simplicial sets 
  \[
    \iota:(\Delta^3, \Delta^{\{0,1\}},\{ \Delta^{\{0,1,3\}} \} \subset U_0) \to (\Delta^3,\Delta^{\{0,1\}},\{ \Delta^{\{0,1,n\}} \} \subset \sharp),
  \]
  is $\bS$-anodyne where $U_0$ is the collection of all 2-faces except $\Delta^{\{1,2,3\}}$.
\end{lemma}
\begin{proof}
  Let $S=(S,E_S,T_S \subset \sharp)$ be an $\bS$ simplicial set and let $p:X \to S$ be an $\bS$-fibration. We will show that $p$ has the right lifting property against the map $\iota$. Once this claim is established we will factor $\iota$
  \[
  \begin{tikzcd}
    (\Delta^3, \Delta^{\{0,1\}},\{ \Delta^{\{0,1,3\}} \} \subset U_0):=A \arrow[r,"\alpha"] & X \arrow[r,"p"] & (\Delta^3,\Delta^{\{0,1\}},\{ \Delta^{\{0,1,n\}} \} \subset \sharp)=B
  \end{tikzcd}
  \]
  as an $\bS$-anodyne morphism followed by a $\bS$-fibration where $B$ is of the form $(S,E_S,T_S \subset \sharp)$. It will then follow that we can produce a solution to the lifting problem
  \[
    \begin{tikzcd}
      A \arrow[d,"\iota"] \arrow[r] & X \arrow[d] \\
      B \arrow[r,"\on{id}_B"]  \arrow[ur,dotted] & B
    \end{tikzcd}
  \]
  which exhibits $\iota$ as a retract of the $\bS$-anodyne map $\alpha$ thus concluding the proof.

  In order to complete the proof we must show the claim. Suppose we are given a lifting problem
   \[
    \begin{tikzcd}
      A \arrow[d,"\iota"] \arrow[r,"\sigma"] & X \arrow[d,"p"] \\
      B \arrow[r]  & S
    \end{tikzcd}
  \]
  Let $\sigma(1 \to 2)=u$, $\sigma(2\to 3)=v$ and $\sigma(1 \to 3)=\omega$. Since $p$ is a $\bS$-fibration we can solve the lifting problem 
  \[
   \begin{tikzcd}
      \Lambda^2_1 \arrow[r] \arrow[d] & X \arrow[d] \\
    \Delta^2 \arrow[r,swap,"d_0(p(\sigma))"] \arrow[ur,dotted,"\varphi"] & S
   \end{tikzcd}
  \]
  and produce a lean triangle $\varphi$. We consider a subsimplicial set of $Q \subset \Delta^4$ consisting in the following faces:
  \begin{itemize}
    \item The face missing the vertex $2$.
    \item The face missing the vertex $4$.
    \item The 2-dimensional face $\Delta^{\{2,3,4\}}$. 
  \end{itemize}
  We then produce a map $\theta: Q \to X$ as follows:
  \begin{itemize}
    \item We map the face missing the vertex $2$ via $\sigma$.
    \item We map the face missing the vertex $4$ via $s_1(d_3(\sigma))$.
    \item We map $\Delta^{\{2,3,4\}}$ via $\varphi$.  
  \end{itemize}
  and equipp $Q$ with the induced decorations. It follows that we can extend $Q$ to a map $\Xi: \Delta^4 \to X$ lying over $s_1(p(\sigma))$. Since $p$ has the right lifting property against the morphism \ref{mb:wonky4} we see that $d_0(\sigma)$ is lean if and only if the image of $\Delta^{\{1,2,4\}}$ under $\Xi$ is thin in $X$.

  We observe that in $d_1(\Xi)$ every face is lean scaled except possible the face missing the vertex 2. Again, as a consequence of \ref{mb:wonky4} it follows that this face must also be lean scaled. Morever this face lies over a thin simplex of $S$ so it must be itself thin. From now on we can focus our attention to $d_3(\Xi)$.

  Let us remark that in $d_3(\Xi)=\rho$ every face is thin scaled except possibly the face missing the vertex $0$ and the the edge $0 \to 1$ is marked. One checks easily that the pullback of $p$ along the simplex thin $d_2(p(\sigma))$ simplex yields a fibration of $\infty$-bicategories $X^{'} \to (\Delta^2,\sharp)$ where we can identify the image of $d_0(\rho)$ with a morphism in the mapping $\infty$-category of $X^{'}$. One easily shows that this morphism is an equivalence and thus it must be thin. This concludes the proof.
\end{proof}

\begin{definition}\label{def:gencof}
  We say that a map of of $\bS$ simplicial sets is a cofibration if its underlying map of simplicial sets is a monomorphism. One can easily verify that the class of cofibrations is generated by the following families of maps:
  \begin{itemize}
    \myitem{(C1)} The boundary inclusions $(\partial \Delta^n, \flat,\flat) \to (\Delta^n,\flat,\flat)$ for $n\geq 0$.
    \myitem{(C2)} The map $(\Delta^1,\flat,\flat) \to (\Delta^1,\sharp,\flat)$.
    \myitem{(C3)} The map $(\Delta^2,\flat,\flat) \to (\Delta^2,\flat,\flat \subset \sharp)$.
    \myitem{(C4)} The map $(\Delta^2,\flat,\flat \subset \sharp) \to (\Delta^2,\flat, \sharp)$.
  \end{itemize}
\end{definition}

\begin{proposition}\label{prop:pushoutproduct}
  Let $f:(X,E_X,T_X \subseteq C_X) \to (Y,E_Y,T_Y \subset C_Y)$ be a cofibration of $\bS$ simplicial sets and $g:(A,E_A,T_A \subseteq C_A)\to (B,E_B,C_B \subseteq T_B)$ be an $\bS$-anodyne morphism. Then the pushout-product:
  \[
    \begin{tikzcd}
     f \wedge g: X \times B \coprod\limits_{X \times A} A \times Y \arrow[r] & Y \times B
    \end{tikzcd}
  \]
  is again $\bS$-anodyne.
\end{proposition}
\begin{proof}
  The proof is almost identical to the proof of \cite[Proposition 3.14]{AGS_CartI} and left as an exercise.
\end{proof}

\begin{remark}\label{rem:funmb}
  Observe that given a pair of $\bS$ simplicial sets $X$, $Y$ we can produce a functor $\bS$ simplicial set $\on{Fun}^{\mathbf{mb}}(X,Y)$ in an obvious way.
\end{remark}

\begin{corollary}
  Let $p: Y \to S$ be an $\bS$-fibration. Then for every $\bS$ simplicial set $X$ the induced map 
  \[
    \begin{tikzcd}
      \on{Fun}^{\mathbf{mb}}(X,Y) \arrow[r] & \on{Fun}^{\mathbf{mb}}(X,S)
    \end{tikzcd}
  \]
  is an $\bS$-fibration.
\end{corollary}
\begin{proof}
  It follows from \autoref{prop:pushoutproduct} after looking at the adjoint lifting problems.
\end{proof}

\begin{definition}
  Let $p:Y \to S$ be an $\bS$-fibration and consider a map of $\bS$ simplicial sets $q:X \to S$. We define an $\infty$-bicategory of functors over $S$ as the pullback
  \[
    \begin{tikzcd}
      \on{Map}_S(X,Y) \arrow[r] \arrow[d] & \on{Fun}^{\mathbf{mb}}(X,Y) \arrow[d] \\
      \Delta^{0} \arrow[r,"q"] & \on{Fun}^{\mathbf{mb}}(X,S)
    \end{tikzcd}
  \]
\end{definition}

\begin{definition}
  Given a scaled simplicial set $(S,T_S)$ we define the category $\left(\on{Set}^{\mathbf{mb}}_\Delta\right)_{/S}$ of $\bS$ simplicial sets over $(S,\sharp,T_S \subset \sharp)$ as follows:
  \begin{itemize}
    \item The objects are maps $p:(X,E_X,T_X \subset C_X) \to (S,\sharp,T_S \subset \sharp)$.
    \item A morphism from $p:(X,E_X,T_X \subset C_X) \to (S,\sharp,T_S \subset \sharp)$ to $q:(Y,E_Y,T_Y \subset C_Y) \to (S,\sharp,T_S \subset \sharp)$ is given by a map $f: (X,E_X,T_X \subseteq C_X) \to (Y,E_Y,T_Y \subset C_Y) $ such that $q \circ f=p$.
  \end{itemize}
  An object of $\left(\on{Set}^{\mathbf{mb}}_\Delta\right)_{/S}$ is said to be a $(0,1)$-fibration if the corresponding map of $\bS$ simplicial sets is an $\bS$-fibration.
\end{definition}

\begin{definition}
  A morphism $f: A \to B$ in $\left(\on{Set}^{\mathbf{mb}}_\Delta\right)_{/S}$ is said to be:
  \begin{itemize}
    \item A cofibration if its underlying map of $\bS$ simplicial sets is a cofibration.
    \item A weak equivalence if for every $(0,1)$-fibration $p:X \to S$ the associated map of $\infty$-bicategories
    \[
        \begin{tikzcd}
          f^*:\on{Map}_S(B,X) \arrow[r,"\simeq"] & \on{Map}_S(A,X)
        \end{tikzcd}
      \]  
      is a bicategorical equivalence.
    \item  A trivial fibration if it has the right lifting property against every cofibration in $\left(\on{Set}^{\mathbf{mb}}_\Delta\right)_{/S}$. 
    \item A trivial cofibration if it is \emph{both} a weak equivalence and a cofibration.
  \end{itemize}
\end{definition}

\begin{lemma}
  Let $f:A \to B$ be a morphism over $S$ in $\left(\on{Set}^{\mathbf{mb}}_\Delta\right)_{/S}$ such that $f$ is a trivial fibration. Then $f$ is a weak equivalence.
\end{lemma}
\begin{proof}
  Note that since $f$ is a trivial fibration we can construct a section $g:B\to A$ such that $f \circ g= \on{id}_B$. Moreover, we can further produce a marked homotopy $A \times (\Delta^1)^\sharp \to A$ between the identity morphism on $A$ and the composite $g \circ f$ which is compatible with the projection map from $A$ to $S$. This pair of homotopy inverse morphisms thus define the desired equivalence on mapping $\infty$-bicategories
  \[
        \begin{tikzcd}
          f^*:\on{Map}_S(B,X) \arrow[r,"\simeq"] & \on{Map}_S(A,X)
        \end{tikzcd}
      \]  
     and thus $f$ is a weak equivalence. 
\end{proof}

\begin{proposition}\label{prop:mappingclasses}
  The $f:A \to B$ be a morphism in  $\left(\on{Set}^{\mathbf{mb}}_\Delta\right)_{/S}$. Given a $(0,1)$-fibration $p:X \to S$ let us consider the induced functor on mapping $\infty$-bicategories 
  \[
        \begin{tikzcd}
          f^*:\on{Map}_S(B,X) \arrow[r] & \on{Map}_S(A,X)
        \end{tikzcd}
      \]  
      Then it follows that 

  \begin{itemize}
    \item[i)] If $f$ is $\bS$-anodyne then $f^*$ is a trivial fibration of scaled simplicial sets.
     \item[ii)] If $f$ is a cofibration then $f^*$ is a fibration of $\infty$-bicategories.
     \item[iii)] If $f$ is a trivial cofibration then $f^*$ is a trivial fibration of scaled simplicial sets.
  \end{itemize}
\end{proposition}
\begin{proof}
  The first statement follows directly from \autoref{prop:pushoutproduct}. To see that $ii)$ holds we observe that again by \autoref{prop:pushoutproduct} that $f^*$ has the right lifting property against all scaled anodyne maps. Since the marked morphisms in the mapping $\infty$-bicategories are equivalences it follows that $f^*$ is an isofibration. In \cite{GHL_Equivalence}, the authors characterise the model structure on scaled simplicial sets as a Cisinski model structure. This in turn implies that $f^*$ is a fibration of scaled simplicial sets. The final claim follows from $ii)$ together with the definition of the class of weak equivalences.
\end{proof}

\begin{lemma}
  Let us consider a pushout diagram in $\left(\on{Set}^{\mathbf{mb}}_\Delta\right)_{/S}$
  \[
    \begin{tikzcd}
      A \arrow[r,"u"] \arrow[d,"v"] & B \arrow[d,] \\
      C \arrow[r,"i"] & P
    \end{tikzcd}
  \]
  where $u$ is a weak equivalence and $v$ is a cofibration. Then $i: B \to C$ is also a weak equivalence.
\end{lemma}
\begin{proof}
  Given a $(0,1)$-fibration $p:X \to S$ we observe that since $v$ is a cofibration we obtain a pullback diagram of $\infty$-bicategories
  \[
    \begin{tikzcd}
      \on{Map}_S(P,X) \arrow[r,"i^*"] \arrow[d] & \on{Map}_S(C,X) \arrow[d,"v^*"] \\
      \on{Map}_S(B,P) \arrow[r,"u^*"] & \on{Map}_S(A,X)
    \end{tikzcd}
  \]
  which shows that $i^*$ is a bicategorical equivalence and consequently we see that $i: B \to C$ is a weak equivalence in $\left(\on{Set}^{\mathbf{mb}}_\Delta\right)_{/S}$.
\end{proof}

\begin{proposition}
  An object $p:X \to S$ in $\left(\on{Set}^{\mathbf{mb}}_\Delta\right)_{/S}$ has the right lifting property against the class of trivial cofibrations if and only if it is a $(0,1)$-fibration.
\end{proposition}
\begin{proof}
  Observe that due to $i)$ in \autoref{prop:mappingclasses} it follows that every $\bS$-anodyne morphism is a trivial cofibration. Therefore, any object having the right lifting property against trivial cofibrations must be a $(0,1)$-fibration. To check the converse we consider a $(0,1)$-fibration $p:X \to S$ and a trivial cofibration $f:A \to B$.  Then, in order to produce a solution to the lifting problem
  \[
    \begin{tikzcd}
      A \arrow[r,"\alpha"] \arrow[d,"f"] & X \arrow[d,"p"] \\
      B \arrow[r] \arrow[ur,dotted] & S
    \end{tikzcd}
  \]
  we observe that since the map $f^*:\on{Map}_S(B,X) \to \on{Map}_S(A,X)$ is a trivial fibration and in particular surjective, we can pick a preimage of $\alpha \in \on{Map}_S(A,X)$ which yelds the solution to our problem.
\end{proof}

\begin{definition}
  Let $p:X \to S$ be a $(0,1)$-fibration and consider an object $A \to S$ in $\left(\on{Set}^{\mathbf{mb}}_\Delta\right)_{/S}$.  We set the following notation:
  \begin{enumerate}
    \item  We denote by $\on{Map}^{\text{th}}_S(A,X)$ the underlying $\infty$-category of the mapping $\infty$-bicategory.
    \item  We denote by $\on{Map}_S^{\simeq}(A,X)$ the underlying groupoid of the mapping $\infty$-bicategory.
  \end{enumerate}
\end{definition}

\begin{proposition}\label{prop:fibrewisecriterion}
  Let $f:X \to Y$ be a morphism in $\left(\on{Set}^{\mathbf{mb}}_\Delta\right)_{/S}$ where both $X$ and $Y$ are $(0,1)$-fibrations. Then the following are equivalent:
  \begin{itemize}
    \item[i)] For every $(0,1)$-fibration $Z \to S$ the induced map $f^*: \on{Map}_S(Y,Z) \to \on{Map}_S(X,Z)$ is an equivalence of $\infty$-bicategories.
    \item[ii)] For every $(0,1)$-fibration $Z \to S$ the induced map $f^*: \on{Map}^{\on{th}}_S(Y,Z) \to \on{Map}^{\on{th}}_S(X,Z)$ is an equivalence of $\infty$-categories.
    \item[iii)] For every $(0,1)$-fibration $Z \to S$ the induced map $f^*: \on{Map}^{\simeq}_S(Y,Z) \to \on{Map}^{\simeq}_S(X,Z)$ is a homotopy equivalence of groupoids.
    \item[iv)] The exists a morphism $g:Y \to X$ over $S$, which is a homotopy inverse to $f$.
    \item[v)] For every $s \in S$ the induce morphism on fibres $f_s:X_s \to Y_s$ is a bicategorical equivalence.
  \end{itemize}
\end{proposition}
\begin{proof}
  The implications $i) \implies ii) \implies iii)$ are clear. We commence the proof by showing that $iii) \implies iv)$. We consider the homotopy equivalence $\on{Map}^{\simeq}_S(Y,X) \to \on{Map}^{\simeq}_S(X,X)$ and pick an object $g\in \on{Map}^{\simeq}_S(Y,X)$ such that $g \circ f \simeq \on{id}_X$. To show that $g$ is the desired homotopy inverse to $f$ we need to show that $f \circ g \isom \on{id}_Y$. To see this we see that the map $\on{Map}_S^{\simeq}(Y,Y)  \to  \on{Map}_S^{\simeq}(X,Y)$ maps both $f \circ g$ and $\on{id}_Y$ to morphisms which are equivalent to $f$. Consequently we see that $f \circ g \isom \on{id}_Y$.

  Observe that $iv) \implies v)$ follows from the fact that since our homotopies are fibrewise they descend to equivalences on the corresponding fibres.

  In order to exhibit that $v) \implies i)$ we use the small object argument to factor $f: X \to Y$ as a composite $X \to \hat{X} \to Y$ where the first morphism is \bS-anodyne (and therefore a weak equivalence) and the second morphism has the right lifting property against the class of $\bS$-anodyne morphisms. In particular it follows that $\hat{X} \to S$ is again a $(0,1)$-fibration. The claim now follows from \autoref{prop:trivfibMB}.
\end{proof}

\begin{lemma}
  Given an object $A \to S$ in $\left(\on{Set}^{\mathbf{mb}}_\Delta\right)_{/S}$ then the projection map $\pi_A:A \times (\Delta^n,\sharp,\sharp) \to A$ is a weak equivalence.
 \end{lemma}
 \begin{proof}
   Let $\iota: \Delta^0 \to (\Delta^n,\sharp,\sharp)$ be the inclusion of the initial vertex. It is an easy exercise to show that $\iota$ is $\bS$-anodyne. It then follows from $\autoref{prop:pushoutproduct}$ that $\iota_A:A \to A \times (\Delta^n,\sharp, \sharp)$ is also $\bS$-anodyne. We conclude the proof by observing that $\pi_A \circ \iota_A=\on{id}_A$.
 \end{proof}

\begin{theorem}\label{thm:model}
  Let $S$ be a scaled simplicial set. Then there exists a left proper combinatorial simplicial model structure on $\left(\on{Set}^{\mathbf{mb}}_\Delta\right)_{/S}$, which is characterized uniquely by the following properties:
  \begin{itemize}
    \item[C)] A morphism $f:X \to Y$ in $\left(\on{Set}^{\mathbf{mb}}_\Delta\right)_{/S}$ is a cofibration if and only if $f$ induces a monomorphism on the underlying simplicial sets.
    \item[F)] An object $p:X \to S$ in $\left(\on{Set}^{\mathbf{mb}}_\Delta\right)_{/S}$ is fibrant if and only if it is a $(0,1)$-fibration.
  \end{itemize}
\end{theorem}
\begin{proof}
  The proof is totally analogous to the proof of Theorem 3.42 in \cite{AGS_CartI}.
\end{proof}

\begin{remark}
  We will refer to the model structure in the previous theorem as model structure on $(0,1)$-Cartesian fibrations over $S$.
\end{remark}

\begin{definition}
  Let $K=(K,E_K) \in \on{Set}_\Delta^+$ and let $p:X \to S$ be an object in $\left(\on{Set}^{\mathbf{mb}}_\Delta\right)_{/S}$. We define the tensor $K \otimes X$ as $I(K) \times X \to X \to S $ where $I(K)=(K,E_K,\sharp)$. Similarly, we define the cotensor $X^K$ by declaring that a map  $\bS$-simplicial sets  $\varphi:Y \to X^K$ over $\overline{\varphi}:Y \to S$ to be equivalent to the data of a commutative diagram
  \[
    \begin{tikzcd}
      (K,E_K,\sharp) \times Y \arrow[r] \arrow[d] & X \arrow[d,"p"] \\
      Y \arrow[r,"\overline{\varphi}"] & S.
    \end{tikzcd}
  \]
 
\end{definition}

\begin{theorem}
  The model category $\left(\on{Set}^{\mathbf{mb}}_\Delta\right)_{/S}$ is a $\on{Set}_\Delta^+$-enriched model category.
\end{theorem}
\begin{proof}
  It is clear that the construction $\on{Map}^{\on{th}}_S(\mathblank,\mathblank)$ equipps $\left(\on{Set}^{\mathbf{mb}}_\Delta\right)_{/S}$ with the structure of a $\Set_\Delta^{+}$-enriched model category. Since the tensor preserves colimits in both variables separately it will enough to show that given $i:L \to K$ a cofibration in $\on{Set}_\Delta^+$  and a cofibration $f:X \to Y$ in $\left(\on{Set}^{\mathbf{mb}}_\Delta\right)_{/S}$ the corresponding pushout-product map
  \[
   i\wedge f: L \tensor Y \coprod\limits_{L \tensor X} K \tensor X \xrightarrow{} K \tensor Y
  \]
is again cofibration which is a weak equivalence whenever $i$ or $f$ is. Note that $i \wedge f$ is clearly a cofibration so we can focus our attention in proving the weak equivalence part of the claim.

First, let us recall that given an anodyne morphism of marked simplicial sets $A \to B$ it follows that $I(A) \to I(B)$ is $\bS$-anodyne. It then follows as a consequence \autoref{prop:pushoutproduct} that we can assume that $i$ and $f$ are morphisms among fibrant objects in the corresponding model structures. 

We note that given a pair of fibrant objects $L$ and $p:X \to S$ it follows that $L \tensor X$ is again fibrant. To finish the proof we assume that $f:X \to Y$ is a weak equivalence ( the case for $i$ is totally analogous). Then it follows that for every $s \in S$ the map $(L \tensor X)_s \to (L \tensor Y)_s$ is identified with  the map
\[
  L \times X_s \xlongrightarrow{\simeq} L \times Y_s
\]
which is a bicategorical equivalence by assumption. It follows that the map $L \tensor X \to L \tensor Y$ is a weak equivalence in $\left(\on{Set}^{\mathbf{mb}}_\Delta\right)_{/S}$. We can know consider a pushout diagram
\[
  \begin{tikzcd}
    L \tensor X \arrow[r,"\simeq"] \arrow[d] & L \tensor  Y \arrow[d] \\
    K \tensor X \arrow[r,"\simeq"] &  L \tensor Y \coprod\limits_{L \tensor X} K \tensor X 
  \end{tikzcd}
\]
Moreover, using a similar argument as before we see that $K \tensor X \xrightarrow{\simeq} K \tensor Y$ is also a weak equivalence. The claim now follows from 2-out-of-3.
\end{proof}

\begin{proposition}\label{prop:mbbasechange}
  Let $f: S \to S'$ be a map of scaled simplicial sets then postcomposition with $f$ induces a left Quillen functor
  \[
    f_!: \left(\on{Set}_\Delta^{\mathbf{mb}}\right)_{/S} \llra \left(\on{Set}_\Delta^{\mathbf{mb}}\right)_{/S'}:f^*
  \]
  which is left adjoint to the pullback functor $f^*$.
\end{proposition}
\begin{proof}
  It is clear that $f_!$ preserves cofibrations. To finish the proof we only need to show that $f_!$ preserves weak equivalences. Given $\iota: A \to B$ and a fibrant object $p:Y \to S' $ we observe that we have a commutative diagram
  \[
    \begin{tikzcd}
      \on{Map}_{S}(B,f^* Y) \arrow[d,"\simeq"] \arrow[r,"\simeq"] &   \on{Map}_{S}(A,f^* Y) \arrow[d,"\simeq"] \\
        \on{Map}_{S'}(i_! B, Y) \arrow[r] &   \on{Map}_{S'}(i_! A, Y)
    \end{tikzcd}
  \]
  so the conclusion holds by 2-out-of-3.
\end{proof}

We finish this section by comparing the model structure in \autoref{thm:model} with the model structure constructed in Theorem 3.2.6 in \cite{LurieGoodwillie}.

\begin{definition}\label{def:inclusionmarked}
   Let $(S,T_S)$ be a scaled simplicial set and consider the category $\left(\on{Set}_\Delta^{+}\right)_{/S}$ of marked simplicial sets over $S$. We have a functor
   \[
      \begin{tikzcd}
        R: \left(\on{Set}_\Delta^{+}\right)_{/S} \arrow[r] & \left(\on{Set}_\Delta^{\mathbf{mb}}\right)_{/S}, & (X,E_X) \arrow[r] & (X,E_X,T_X \subset \sharp)\\
      \end{tikzcd}
    \] 
  where $T_X$ consistins in those triangles in $X$ lying over thin triangles in $S$.
 \end{definition} 

 \begin{theorem}
   Let $(S,T_S)$ be a scaled simplicial set and consider an object $(X,E_X)$ in $\left(\on{Set}_\Delta^{+}\right)_{/S}$ then $X=(X,E_X)$ is $\mathfrak{P}_{S}$-fibered (see Definition 3.2.1 and Example 3.2.9 in \cite{LurieGoodwillie}) if and only if $R(X)$ is a $(0,1)$-fibration. Moreover, if $Y$ is a $(0,1)$-fibration over $S$ such that every triangle of $Y$ is lean then there exists a $\mathfrak{P}_S$-fibered object $T$ such that $R(T)=Y$.
 \end{theorem}
 \begin{proof}
   Let us suppose that $X$ is $\mathfrak{P}_S$-fibered and let us show that $R(X)$ defines a $(0,1)$-fibration. We will show that $R(X)$ has the right lifting property against the class of $\bS$-anodyne morphisms. To this end we recall that an object in $\left(\on{Set}_\Delta^{+}\right)_{/S}$ is $\mathfrak{P}_S$-fibered if and only if it has the right lifting property against the class of $\mathfrak{P}_S$-anodyne morphisms described in Definition 3.2.10 in \cite{LurieGoodwillie}. We first check that $R(X)$ has the right lifting property against the class of maps given in \ref{mb:equivalences}. Indeed given a map from a Kan complex $K \to R(X)$ we can use the morphisms of type $(A_1)$ in \cite[Definition 3.2.10]{LurieGoodwillie} to see that every morphism of $K$ maps to a marked edge in $R(X)$. Similarly, our construction of $R$ guarantees the $R(X)$ has the right lifting property against morphisms of type \ref{mb:coCartoverThin}. The rest of the lifting problems follow immediately from the definition of the class of $\mathfrak{P}_S$-anodyne morphisms.

   The converse follows by a similar argument. To finish the proof we suppose that we are given  $(0,1)$-fibration of the form $Y=(Y,E_Y,T_Y \subset \sharp)$. Then \ref{mb:coCartoverThin} implies that $T_Y$ is simply the collection of triangles lying over thin triangles in $S$. Therefore, it we can consider $\hat{Y}=(Y,E_Y)$ and observe that $R(\hat{Y})=Y$. The previous part of the proof shows that $\hat{Y}$ must be $\mathfrak{P}_S$-fibered.
 \end{proof}

 \begin{proposition}
   Let $p:X\to S$ be a $(0,1)$-fibration then for every $s \in S$ the fiber $X_s$ is an $\infty$-category if and only if every triangle of $X$ is lean.
 \end{proposition}
 \begin{proof}
   It is clear that if every triangle of $X$ is lean the fibers must be $\infty$-categories. For the converse, let us assume that for every $s \in S$ the fiber $X_s$ is an $\infty$-category. Let $\sigma:\Delta^2 \to X$ and assume that $p(\sigma)=s_1(p(f))$ for some edge $\Delta^1 \to S$. We pick a marked edge lying over $p(f)$ and construct a $3$-simplex $\theta: \Delta^3 \to X$ lying over $s_1(p(\sigma))$ such that $\theta(0 \to 1)$ is our chosen marked edge and such that every face is lean except possible $d_1(\theta)=\sigma$. We conclude that $\sigma$ is lean since $p$ has the right lifting property against morphisms of type \ref{mb:wonky4}.

   For the general case we take a marked edge lying over $p(1 \to 2)$ and construct another $3$-simplex $\Xi: \Delta^3 \to X$ such that $\Xi(1 \to 2)$ is our chosen marked edge and such that $d_2(\Xi)=\sigma$. It follows that every face of $\Xi$ is lean except possible the face skipping the vertex $1$ and the face missing the vertex $2$. We conclude that the face missing the vertex $1$ is lean since it falls in the previous case. It then follows that $\sigma$ is lean. 
 \end{proof}

 \subsection{Marked-scaled simplicial sets.}
 In this section we consider the model structure of $\bS$ simplicial sets over $\Delta^0$ and show that its Quillen equivalent to the model structure on scaled simplicial sets given in \cite{LurieGoodwillie}. While doing this, we will observe that the collection of lean and thin triangles become redundant in this specific case. To deal with this issue we introduce a third model structure on simplicial sets equipped with a collection of marked edges and \emph{one} collection of triangles which we call marked-scaled simplicial sets.

 \begin{definition}
   Let $\left(\on{Set}_\Delta^{\mathbf{mb}}\right)_{/\Delta^0}=\on{Set}_\Delta^{\mathbf{mb}}$. We consider a functor $ R: \on{Set}_\Delta^{\mathbf{mb}} \to \on{Set}_\Delta^{\mathbf{sc}}$ which sends an $\bS$ simplicial set $ (X,E_X, T_X \subseteq C_X)$ to the scaled simplicial set $R (X,E_X, T_X \subseteq C_X)=(X,C_X)$. This functor has a left adjoint $L:\on{Set}_\Delta^{\mathbf{sc}} \to \on{Set}_\Delta^{\mathbf{mb}}$ which is given by $L(Y,T_Y)=(Y,\flat,\flat \subset T_Y)$.
 \end{definition}

 \begin{theorem}\label{thm:comparisonscaled}
   The functor $ L: \on{Set}_\Delta^{\mathbf{sc}} \to \on{Set}_\Delta^{\mathbf{mb}}  $ is a left Quillen equivalence.
 \end{theorem}
 \begin{proof}
   It is clear that $L$ preserves cofibrations and colimits so in order to show that $L$ is a left Quillen functor it will enough to show that $L$ preserves weak equivalences. Observe that given an object $(X,E_X,T_X \subset C_X)$ we have an anodyne morphism $(X,E_X,T_X \subset C_X) \to (X,E_X,C_X)$ since every triangle in $\Delta^0$ is thin. Using this observation it is easy to see that $L$ maps scaled anodyne morphisms to trivial cofibrations in $\on{Set}_\Delta^{\mathbf{mb}}$. This shows that it is enough to show that $L$ preserves weak equivalences among fibrant scaled simplicial sets. However, this is clear since a weak equivalence between fibrant scaled simplicial sets has an inverse up to homotopy.

   To conclude that $L$ is a left Quillen equivalence we first observe that $R \circ L= \on{id}$ by definition. This shows that it is only left to show that for any fibrant object  $(Y,E_Y,T_Y \subseteq C_Y) \in \on{Set}_\Delta^{\mathbf{mb}} $ the counit map $LR(Y) \to Y$ is a weak equivalence. Note that since our chosen $\bS$ simplicial set is fibrant $C_Y=T_Y$ and $(Y,T_Y)$ is an $\infty$-bicategory. In particular, we see that $LR(Y) \to Y$ is the weakly saturated class of morphisms of type \ref{mb:equivalences} and \ref{mb:coCartoverThin} in \autoref{def:mbsanodyne}.
 \end{proof}

\begin{definition}
  A marked-scaled simplicial set denoted by $(X,E_X,T_X)$ is given by
  \begin{enumerate}
    \item A simplicial set $X$.
    \item A collection of edges $E_X \subseteq X_1$ which contains \emph{all degenerate} edges. We refer to the elements of this collection as \emph{marked} edges.
    \item A collection of triangles $T_X \subseteq X_2$ which contains $\emph{all degenerate}$ triangles. We refer to the elements of this collection as \emph{thin} triangles.
  \end{enumerate}
  A morphism of marked-scaled simplicial sets $f:(X,E_X,T_X) \to (Y,E_Y,T_Y)$ is given by a map of simplicial sets such that $f(E_X) \subseteq E_Y$ and $f(T_X) \subseteq T_Y$. We denote by $\on{Set}_\Delta^{\mathbf{ms}}$ the category of marked-scaled simplicial sets.
\end{definition}

\begin{definition}\label{def:msanodyne}
  The set of \emph{generating marked-scaled anodyne maps} $\MS$ is the set of maps of marked-scaled simplicial sets consisting of
  \begin{enumerate}
    \myitem{(M1)}\label{MS:innerhorn} The inner horn inclusions 
    \[
    \bigl(\Lambda^n_i,\flat, \{\Delta^{\{i-1,i,i+1\}}\}\bigr)\rightarrow \bigl(\Delta^n,\flat, \{\Delta^{\{i-1,i,i+1\}}\}\bigr).
    \quad , \quad n \geq 2 \quad , \quad 0 < i < n ;
    \]
    
    \myitem{(M2)}\label{ms:wonky4} The map 
    \[
    (\Delta^4,\flat, T)\rightarrow (\Delta^4,\flat, T\cup \{\Delta^{\{0,3,4\}}, \ \Delta^{\{0,1,4\}}\}),
    \]
    where we define
    \[
    T\overset{\text{def}}{=}\{\Delta^{\{0,2,4\}}, \ \Delta^{\{ 1,2,3\}}, \ \Delta^{\{0,1,3\}}, \ \Delta^{\{1,3,4\}}, \ \Delta^{\{0,1,2\}}\};
    \]
    \myitem{(M3)}\label{ms:2coCartesianmorphs} The set of maps
    \[
    \Bigl(\Lambda^n_0,\{\Delta^{\{0,1\}}\}, \{ \Delta^{\{0,1,n\}} \}\Bigr) \to \Bigl(\Delta^n,\{\Delta^{\{0,1\}}\}, \{ \Delta^{\{0,1,n\}} \}\Bigr) \quad , \quad n \geq 2.
    \]
    
    \myitem{(M4)}\label{mb:2CartliftsExist} The inclusion of the initial vertex
    \[
    \Bigl(\Delta^{0},\sharp,\sharp \Bigr) \rightarrow \Bigl(\Delta^1,\sharp,\sharp \Bigr).
    \]
    \myitem{(MS1)}\label{ms:composeacrossthin} The map
    \[
    \Bigl(\Delta^2,\{\Delta^{\{0,1\}}, \Delta^{\{1,2\}}\},\sharp \Bigr) \rightarrow \Bigl(\Delta^2,\sharp,\sharp \Bigr).
    \]
    
    \myitem{(ME)}\label{ms:equivalences} For every Kan complex $K$, the map
    \[
    \Bigl( K,\flat,\sharp  \Bigr) \rightarrow \Bigl(K,\sharp, \sharp\Bigr).
    \]
    Which requires that every equivalence is a marked morphism.
  \end{enumerate}
  A map of $\MS$ simplicial sets is said to be \MS-anodyne if it belongs to the weakly saturated closure of \MS.
\end{definition}

\begin{remark}
  Observe that $(X,E_X,T_X)$ has the right lifting property against the class of $\MS$-anodyne morphisms if and only if $X$ is an $\infty$-bicategory, $E_X$ is the collection of equivalences in $X$ and $T_X$ is the collection of thin triangles. Consequently, we might call such marked-scaled simplicial sets $\infty$-bicategories as well.
\end{remark}

\begin{definition}
  We say that a morphism of marked-scaled simplicial sets is a cofibration if its underlying map of simplicial sets is a cofibration.
\end{definition}

\begin{proposition}
  Let $f:X \to Y$ be a cofibration in $\on{Set}_\Delta^{\mathbf{ms}}$ and $g: A \to B$ be an $\MS$-anodyne morphism. Then the pushout-product
  \[
    X \times B \coprod_{X \times A} Y \times A \xlongrightarrow{} Y \times B
  \]
  is again $\MS$-anodyne.
\end{proposition}

\begin{corollary}
  Given a marked-biscaled simplicial set $X$ which has the right lifting property against the class of $\MS$-anodyne morphisms it follows that $\on{Fun}^{\mathbf{ms}}(A,X)$ (compare to \autoref{rem:funmb}) has the right lifting property against the class $\MS$-anodyne morphisms for every $A \in \on{Set}_\Delta^{\mathbf{ms}}$.
\end{corollary}

\begin{definition}
  A morphism of marked-scaled simplicial sets $i:A \to B$ is said to be a \emph{weak equivalence} if for every $\infty$-bicategory $X$ the induced map
  \[
    i^*: \on{Fun}^{\mathbf{ms}}(B,X) \xlongrightarrow{\simeq} \on{Fun}^{\mathbf{ms}}(A,X)
  \]
  is a bicategorical equivalence.
\end{definition}

\begin{remark}\label{rem:mstensor}
  In a similar way as before given $K \in \on{Set}_\Delta^+$ and $X \in \on{Set}_\Delta^{\mathbf{ms}}$ we define the tensor $K \times X:=I(K) \times X$ where $I(K)=(K,E_K,\sharp)$. Similarly, we define the cotensor $X^{K}:= \on{Fun}^{\mathbf{ms}}(I(K),X)$.
\end{remark}

\begin{theorem}
  There exists a left proper combinatorial simplicial model category on $\on{Set}_\Delta^{\mathbf{ms}}$, which is characterized uniquely by the following properties:
  \begin{itemize}
    \item[C)] A morphism $f:X \to Y$ in  $\on{Set}_\Delta^{\mathbf{ms}}$ us a cofibration if and only if $f$ induces a monomorphism on the underlying simplicial set.
    \item[F)] An object $X$ in $\on{Set}_\Delta^{\mathbf{ms}}$ is fibrant if and only if it was the right lifting property against the class of marked-anodyne morphisms.
  \end{itemize}
  Moreover the tensor and cotensor in \autoref{rem:mstensor} equipps $\on{Set}_\Delta^{\mathbf{ms}}$ with the structure of a $\on{Set}_\Delta^+$-enriched model category.
\end{theorem}

\begin{theorem}
  The functor $L: \on{Set}_\Delta^{\mathbf{sc}} \to \on{Set}_\Delta^{\mathbf{ms}}$ sending a scaled simplicial set $(X,T_X)$ to the marked-scaled simplicial set $(X,\flat,T_X)$ is a left Quillen equivalence.
\end{theorem}
\begin{proof}
  The proof is almost identical to the proof of \autoref{thm:comparisonscaled} and thus omitted.
\end{proof}

\section{Local fibrations}
In this section we give a model independent definition of a locally coCartesian fibration of $(\infty,2)$-categories or in our terminology a local $(0,1)$-fibration. We will use the model structure of \autoref{thm:model} to describe the theory of locally coCartesian fibrations of $(\infty,2)$-categories in terms of $(0,1)$-fibrations over $(S,M_S)$ where $(S,T_S)$ is an $\infty$-bicategory and $M_S\subset T_S$ is a subcollection of the thin triangles.

\begin{definition}\label{def:mstri}
  Let $(S,T_S)$ be an $\infty$-bicategory we define $M_S \subseteq T_S$ as the subcollection of triangles consisting in those \emph{thin} triangles such that the edge $\Delta^{\{0,1\}}$ or the edge $\Delta^{\{1,2\}}$ is an equivalence in $S$. We call the elements of $M_S$ the invertible 2-morphisms of $S$.
\end{definition}

\begin{remark}
  Observe that the collection $M_S$ has been studied in \cite{GHL_Gray} to give a definition of oplax unital functors of $\infty$-bicategories.
\end{remark} 

\begin{definition}
  We will use boldface letters $\bcat{S}:=(S,T_S)$ to describe fibrant scaled simplicial sets.
\end{definition}

\begin{remark}
  During this section we will make the notational convetion of denoting by $\left(\on{Set}^{\mathbf{mb}}_\Delta\right)_{/S}$ the category of $\bS$ simplicial sets over $(S,M_S)$ where $(S,T_S)$ is an $\infty$-bicategory and $M_S$ is the collection of 2-simplices defined in \autoref{def:mstri}.
\end{remark}

\begin{proposition}\label{prop:piscatfib}
  Let $p: X \to S$ be a $(0,1)$-fibration. We define a scaled simplicial set $\bcat{X}$ whose underlying simplicial set is the underlying simplicial set of $X$ and where a triangle is declared to be thin if it is lean in $X$ and its image under $p$ belongs to $T_S$. Then the associated functor $p:\bcat{X} \to \bcat{S}$ is a bicategorical fibration.
\end{proposition}
\begin{proof}
  It is clear that $p:\bcat{X} \to \bcat{S}$ has the right lifting property against the class of scaled anodyne maps thus it follows that $\bcat{X}$ is itself and $\infty$-bicategory. To finish the proof we will need to show that $p$ is an isofibration. Given an equivalence $e:\Delta^1 \to \bcat{S}$ and a lift of the source $\Delta^0 \to \bcat{X}$ we can produce a marked edge $\hat{e}: \Delta^1 \to X$ lying over $e$. We observe that the 2-simplices needed for exhibiting $e$ as an equivalence in $\bcat{S}$ are contained in $M_S$ in particular this can be used to show that $\hat{e}$ is an equivalence in $\bcat{X}$. 
\end{proof}

\begin{proposition}\label{prop:cartenriched}
  Let $p: X \to S$ be a $(0,1)$-fibration. Then for every pair of objects $a,b \in X$ the induced functor on mapping $\infty$-categories
  \[
    \begin{tikzcd}
      \bcat{X}(a,b) \arrow[r] & \bcat{S}(p(a),p(b))
    \end{tikzcd}
  \]
  is a Cartesian fibration. Moreover, the composition functors $\bcat{X}(a,b)\times \bcat{X}(b,c) \to \bcat{X}(a,c)$ preserve Cartesian edges.
\end{proposition}
\begin{proof}
  We pick a model for the mapping $\infty$-category discussed in \autoref{def:mappingcat}. Given a morphism $\alpha:\Delta^1 \to \bcat{S}(p(a),p(b))$ and a lift of the target $g: \Delta^0 \to \bcat{X}(a,b)$ we can consider a morphism of type \ref{mb:innerhorn} to produce an edge in $\bcat{X}(a,b)$ which is lean in the $\bS$ simplicial set $X$. We can similarly translate lifting problems of the form,
  \[
    \begin{tikzcd}
      \Lambda^n_n \arrow[r] \arrow[d] & \bcat{X}(a,b) \arrow[d] \\
      \Delta^n \arrow[r] \arrow[ur,dotted] & \bcat{S}(p(a),p(b))
    \end{tikzcd}
  \]
  where the last edge in the top horizontal morphism is mapped to lean 2-simplex to lifting problems of the form
  \[
    \begin{tikzcd}
      \Lambda^{n+1}_{n} \arrow[d] \arrow[r] & X \arrow[d] \\
      \Delta^{n+1} \arrow[r] \arrow[ur,dotted] & S
    \end{tikzcd}
  \]
  where the triangle $\Delta^{\{n-1,n,n+1\}}$ is mapped to a lean triangle in $X$ and thus admits a solution. We conclude that $\bcat{X}(a,b) \to \bcat{S}(p(a),p(b))$ is a Cartesian fibration. To finish the proof we consider the composition functor
  \[
    \begin{tikzcd}
      \bcat{X}(a,b)\times \bcat{X}(b,c) \arrow[r] & \bcat{X}(a,c)
    \end{tikzcd}
  \]
Given a pair of Cartesian edges $\alpha: f \to g$ and $\beta: u  \to v$  the composition map yields a commutative diagram in $\bcat{X}(a,c)$ of the form
\[
  \begin{tikzcd}
    f \circ u \arrow[r] \arrow[d] & f \circ v  \arrow[d] \\
    g \circ u \arrow[r] & g \circ v
  \end{tikzcd}
\]
which tells us that it will suffice to check that precomposition and postcomposition with a 1-morphism preserves Cartesian 2-morphisms. This follows from the fact that $p$ has the right lifting property against morphisms of type \ref{mb:wonky4}.
\end{proof}

\begin{definition}\label{def:cartenriched}
  A bicategorical fibration $p: \bcat{X} \to \bcat{S}$ is said to be \emph{Cartesian-enriched} if
  \begin{itemize}
     \item For every $a,b \in \bcat{X}$ the morphisms $\bcat{X}(a,b) \to \bcat{S}(p(a),p(b))$ are Cartesian fibrations.
     \item For every $a,b,c \in \bcat{X}$ the composite maps $\bcat{X}(a,b) \times \bcat{X}(b,c) \to \bcat{X}(a,c)$ preserve Cartesian edges.
   \end{itemize} 
   Given two Cartesian-enriched bicategorical fibrations $p: \bcat{X} \to \bcat{S}$ and $q: \bcat{Y} \to \bcat{S}$ we say that a functor $f: \bcat{X} \to \bcat{Y}$ is Cartesian-enriched if preserves Cartesian edges in the mapping categories.
\end{definition}

\begin{definition}\label{def:locally01edge}
  Let $p: \bcat{X} \to \bcat{S}$ be a bicategorical fibration. An edge $e: a \to b$ in $\bcat{X}$ is said to be locally $(0,1)$-Cartesian (or a local $(0,1)$-Cartesian edge) if:
  \begin{itemize}
    \item[i)] Given $g: a \to c$ in $\bcat{X}$ and a commutative diagram (represented by a thin simplex) in $\sigma:\Delta^2 \to \bcat{S}$
    \[\begin{tikzcd}
  & {p(b)} \\
  {p(a)} & {} & {p(c)}
  \arrow["{p(e)}", from=2-1, to=1-2]
  \arrow["\alpha", from=1-2, to=2-3]
  \arrow["{p(g)}"', from=2-1, to=2-3]
\end{tikzcd}\]
such that $\alpha$ is an equivalence, then there exists a morphism $\hat{\alpha}: b \to c$ such that $p(\hat{\alpha})=\alpha$ and a  thin 2-simplex  $\hat{\sigma}$ exhibiting  $e \circ \hat{\alpha} \isom g$ such that $p(\hat{\sigma})=\sigma$.
\item[ii)] Given any $\phi:b \to c$ such that $\phi \circ e \isom g$  such that $p(\phi)=\alpha$  as above, then for any other $\varphi:b \to c$ precomposition along $e$ induces a pullback diagram of spaces
\[
   \begin{tikzcd}
    \on{Map}_{\bcat{X}(b,c)}(\phi,\varphi) \arrow[r] \arrow[d] &    \on{Map}_{\bcat{X}(a,c)}(\phi\circ e,\varphi \circ e ) \arrow[d] \\
      \on{Map}_{\bcat{S}(p(b),p(c))}(p(\phi),p(\varphi)) \arrow[r] & \on{Map}_{\bcat{S}(p(a),p(c))}(p(\phi\circ e),p(\varphi \circ e ))
   \end{tikzcd}
 \] 
  \end{itemize}
\end{definition}

\begin{remark}\label{rem:comp}
  We observe that a local $(0,1)$-Cartesian edge lying over an equivalence in $\bcat{S}$ is necessarily an equivalence in $\bcat{X}$. Moreover, the composition of a local $(0,1)$-Cartesian edge with an equivalence is again locally $(0,1)$-Cartesian.
\end{remark}

\begin{definition}
  Let $p:\bcat{X} \to \bcat{S}$ be a bicategorical fibration. We say that an edge $e: a \to b$ is $(0,1)$-Cartesian if for ever $c \in \bcat{X}$ precomposition along $e$ induces a pullback diagram of $\infty$-categories
  \[
    \begin{tikzcd}
      \bcat{X}(b,c) \arrow[r] \arrow[d] & \bcat{X}(a,c) \arrow[d] \\
      \bcat{S}(p(b),p(c)) \arrow[r] & \bcat{S}(p(a),p(c)).
    \end{tikzcd}
  \]
\end{definition}

\begin{proposition}\label{prop:deflocedge}
  Let $p: \bcat{X} \to \bcat{S}$ be a bicategorical fibration and suppose further that $p$ is Cartesian-enriched. Given an edge $e:a \to b$ in $\bcat{X}$ then the following are equivalent:
  \begin{itemize}
    \item[i)] The edge $e$ is locally $(0,1)$-Cartesian.
    \item[ii)] For every $\ell \in \bcat{X}$ such that $p(\ell)=p(b)$ we have a pullback diagram of $\infty$-categories
    \[
      \begin{tikzcd}
        \bcat{X}_{p(b)}(b,\ell) \arrow[r,"(-)\circ e"] \arrow[d] & \bcat{X}(a,\ell) \arrow[d] \\
        \Delta^0 \arrow[r,"p(e)"] & \bcat{S}(p(a),p(b))
      \end{tikzcd}
    \]
    where $X_{p(b)}$ is the fibre of $p$ over $p(b)$.
    \item[iii)] The edge $e$ is $(0,1)$-Cartesian in the pullback along $p(e)$, $\bcat{X}_{p(e)}=\bcat{X} \times_{\bcat{S}}\Delta^1 \to \Delta^1$.
  \end{itemize}
\end{proposition}
\begin{proof}
  To show that $i) \implies ii)$ we observe that since $\bcat{X}(a,\ell) \to \bcat{S}(p(a),p(b))$ is a Cartesian fibration of $\infty$-categories it follows that the strict fibre over $p(e)$ is already a model for the $\infty$-categorical pullback. Therefore it suffices to show that the map
  \[
   e^*: \bcat{X}_{p(b)}(b,\ell) \longrightarrow  \bcat{X}_{p(e)}(a,\ell),
  \]
  is an equivalence of $\infty$-categories. Observe that the first condition in \autoref{def:locally01edge} guarantees that $e^*$ is essentially surjective. Given a pair of objects $\phi, \varphi \in  \bcat{X}_{p(b)}(b,\ell)$ it follows from condition $ii)$ in \autoref{def:locally01edge} that we have a pullback diagram of spaces
  \[
   \begin{tikzcd}
    \on{Map}_{\bcat{X}(b,\ell)}(\phi,\varphi) \arrow[r] \arrow[d] &    \on{Map}_{\bcat{X}(a,\ell)}(\phi\circ e,\varphi \circ e ) \arrow[d] \\
      \on{Map}_{\bcat{S}(p(b),p(b))}(p(\phi),p(\varphi)) \arrow[r] & \on{Map}_{\bcat{S}(p(a),p(b))}(p(\phi\circ e),p(\varphi \circ e ))
   \end{tikzcd}
 \] 
so in particular after taking fibres we have a homotopy equivalence of spaces
\[
  \on{Map}_{\bcat{X}(b,\ell)}(\phi,\varphi)_{\on{id}} \xlongrightarrow{\simeq} \on{Map}_{\bcat{X}(a,\ell)}(\phi\circ e,\varphi \circ e )_{p(e)}
\]
which we identify with the action of $e^*$ on mapping spaces which shows our claim.

Note that  $ii) \iff iii)$ follows immediately from the definition so it will suffice to show that $ii) \implies i)$. 

First let us remark that since $p:\bcat{X} \to \bcat{S}$ is a bicategorical fibration, is in particular an isofibration. So, in order to show that $e$ is locally $(0,1)$-Cartesian, we particularize the conditions in \autoref{def:locally01edge} to the case where $p(\alpha)$ is degenerate edge. Let $\ell \in \bcat{X}$ such that $p(\ell)=p(b)$ and consider an edge $u: a \to \ell$ such that $p(u) \isom p(e)$ in $\bcat{S}(p(a),p(b))$. Since $p$ is Cartesian enriched we can pick an equivalence in $\bcat{X}(a,\ell)$, $\hat{u} \isom u$ such that $p(\hat{u})=p(e)$. Then our assumptions guarantee the existence of an object $\phi \in \bcat{X}(b,\ell)$ such that $p(\phi)$ is degenerate and such that
\[
  \phi \circ e \isom \hat{u} \isom u,
\]
this shows that condition the first condition in \autoref{def:locally01edge} holds. Let $\phi: b \to \ell$ in $\bcat{X}$ such that $p(\phi)$ is degenerate on $p(b)$. To show that we have the necessary pullback square of spaces it will be enough to show that for every $\Xi \in \on{Map}_{\bcat{S}(p(b),p(b))}(p(\phi),p(\varphi))$ the associated morphism on fibres
\[
  \on{Map}_{\bcat{X}(b,\ell)}(\phi,\varphi)_{\Xi} \xlongrightarrow{\simeq} \on{Map}_{\bcat{X}(a,\ell)}(\phi\circ e,\varphi \circ e )_{\Xi \circ p(e)}
\]
is a homotopy equivalence. Note that our assumptions guarantee that this holds whenever $\Xi$ is the identity morphism. Since the map $\bcat{X}(b,\ell) \to \bcat{S}(p(b),p(b))$ is a Cartesian fibration we pick a Cartesian lift of $\Xi$ in $\bcat{X}(b,\ell)$ which we denote by $i:\hat{\varphi} \to \varphi$ and we note that the enrichment of $p$ implies that $i\circ e$ is again a Cartesian morphism in $\bcat{X}(a,\ell)$. The fact that $i$ is Cartesian allows us to construct a commutative diagram of spaces
\[
  \begin{tikzcd}
     \on{Map}_{\bcat{X}(b,\ell)}(\phi,\hat{\varphi})_{\on{id}} \arrow[d,"\simeq"] \arrow[r,"\simeq"] &   \on{Map}_{\bcat{X}(a,\ell)}(\phi\circ e,\hat{\varphi}\circ e)_{p(e)} \arrow[d,"\simeq"] \\
      \on{Map}_{\bcat{X}(b,\ell)}(\phi,\varphi)_{\Xi} \arrow[r] & \on{Map}_{\bcat{X}(a,\ell)}(\phi\circ e,\varphi \circ e )_{\Xi \circ p(e)}
  \end{tikzcd}
\]
so the conclusion follows by 2-out-of-3.
\end{proof}

\begin{definition}
  A bicategorical fibration $p:\bcat{X} \to \bcat{S}$ is said to be a local $(0,1)$-fibration if the following conditions hold:
  \begin{enumerate}
     \item The map $p$ is Cartesian-enriched (see \autoref{def:cartenriched}).
     \item For every $a \in \bcat{X}$ and every $e:p(a) \to b$ in $\bcat{S}$ there exists a local $(0,1)$-Cartesian edge $\hat{e}:a \to \hat{b}$ such that $p(\hat{e})=e$.
   \end{enumerate} 
   We say that a commutative diagram  
   \[
     \begin{tikzcd}
       \bcat{X}\arrow[rr,"f"]  \arrow[dr,swap,"p"] && \bcat{Y} \arrow[dl,"q"] \\
       & \bcat{S} &   
     \end{tikzcd}
   \]
   where $p$ and $q$ are local $(0,1)$-fibrations is a morphism of local $(0,1)$-fibrations if $f$  is a morphism of Cartesian-enriched fibrations and it maps local $(0,1)$-Cartesian edges in $p$ to local $(0,1)$-Cartesian edges in $q$.
\end{definition}

\begin{lemma}\label{lem:commsq}
  Let $p:\bcat{X} \to \bcat{S}$ be a local $(0,1)$-fibration and consider a commutative diagram in $\bcat{X}$ 
  \[
    \begin{tikzcd}
      a \arrow[r,"\simeq"] \arrow[d,swap,"f"] & u \arrow[d,"g"] \\
      b \arrow[r,"\simeq"] & v 
    \end{tikzcd}
  \]
  such that the horizontal morphisms are equivalences in $\bcat{X}$. Then $f$ is a local $(0,1)$-Cartesian edge if and only if $g$ is.
\end{lemma}
\begin{proof}
  Left as an exercise.
\end{proof}

\begin{theorem}\label{thm:hominv}
  Let us consider a commutative diagram of $\infty$-bicategories where the vertical morphisms are bicategorical fibrations
  \[
    \begin{tikzcd}
      \XX \arrow[r,"\simeq"] \arrow[d,swap,"p"] & \CC \arrow[d,"q"] \\
      \SB \arrow[r,"\simeq"] & \DD 
    \end{tikzcd}
  \]
  and the horizontal morphisms are bicategorical equivalences. Then $p$ is a local $(0,1)$-fibration if and only if $q$ is.
\end{theorem}
\begin{proof}
  Observe that we can view our diagram as an injectively fibrant-cofibrant diagram in the arrow category of $\on{Set}_\Delta^{\mathbf{sc}}$. This guarantees the existence weak equivalences  $\CC \to \XX$ and $\DD \to \SB$ making the obvious diagram commute. Therefore we might assume without loss of generality that $q$ is a local $(0,1)$-fibration. After inspecting the associated diagram in mapping categories we learn that $p$ must be Cartesian enriched. To finish the proof we need to show that $p$ has a sufficient supply of local $(0,1)$-Cartesian edges.

  First, we observe that our commutative diagram is in fact a pullback diagram of $\infty$-bicategories. Therefore we have a weak equivalence $\varphi:\XX \to \bcat{P}$ where $\hat{q}:\bcat{P} \to \SB$ denotes the strict pullback of $q$ along the bottom horizontal morphism. It follows that $\hat{q}$ is a local $(0,1)$-fibration. Factoring $\varphi$ as a cofibration (which is necessary a trivial cofibration) and a trivial fibration we might assume without loss of generality that $\varphi$ is a trivial cofibration. We conclude that we have a section $\xi:\bcat{P} \to \XX$ such that $\xi \circ \varphi=\on{id}$. 

  Given $a \in \XX$ and an edge $e:p(a) \to b$ in $\SB$ we pick a lift in $\bcat{P}$ of $e$ with source $\varphi(a)$ which we denote $\hat{e}$. We finally consider $\xi (\hat{e})=\tau$. We claim that $\tau$ is a local $(0,1)$-Cartesian edge of $p$. To see this we note that due to \autoref{lem:commsq} we have that $\varphi(\tau)$ is a local $(0,1)$-Cartesian edge of $\bcat{P}$. The existence of a section $\xi$ the fact that $\varphi(\tau)$ is locally $(0,1)$-Cartesian shows that condition $i)$ in \autoref{def:locally01edge} holds. The second condition follows immediately from the fact that $\varphi$ is a bicategorical equivalence.
\end{proof}

\begin{lemma}\label{lem:innerhorn}
  Let $p:\XX \to \SB$ be a local $(0,1)$-fibration. Let $\sigma: \Delta^2 \to \XX$ be a 2-simplex whose associated 2-morphism in $\XX(\sigma(0),\sigma(2))$ is Cartesian. Given a lifting problem
   \[
    \begin{tikzcd}
      \Lambda^n_i \arrow[d]  \arrow[r,"\varphi"] \arrow[d] & \on{F}_{\XX} \arrow[d,"q"] \\
      \Delta^n \arrow[r]\arrow[ur,dotted] & S
    \end{tikzcd}
  \]
  such that restriction of $\varphi$ to $\Delta^{\{i-1,i,i+1\}}$ equals $\sigma$, then the dotted arrow exists.
\end{lemma}
\begin{proof}
  In virtue of \autoref{thm:hominv} we might assume that our functor is of the form $p: \Nsc(\CC) \to \Nsc(\DD)$. Then it follows that our lifting problem is equivalent to
  \[
    \begin{tikzcd}
      \mathfrak{C}^{\mathbf{sc}}[\Lambda^n_i](0,n) \arrow[d]  \arrow[r] \arrow[d] &  \CC(\varphi(0),\varphi(n)) \arrow[d] \\
      \mathfrak{C}^{\mathbf{sc}}[\Delta^n](0,n) \arrow[r]\arrow[ur,dotted] & \DD(p\varphi(0),p\varphi(n)).
    \end{tikzcd}
  \]
  We conclude that the dotted arrow exists since the left-hand side can be written as an iterated pushout of anodyne morphisms in the Cartesian model structure.
\end{proof}

\begin{definition}
  Let $\pi: \scr{C} \to \scr{D}$ be a fibration of $\infty$-categories. We say that an object $x \in \scr{C}$ is \emph{$p$-initial} if for every $y \in \scr{C}$ the functor $\pi$ yields a homotopy equivalence $\on{Map}_{\scr{C}}(x,y)\isom \on{Map}_{\scr{D}}(\pi(x),\pi(y))$.
\end{definition}

\begin{remark}\label{rem:pinitial}
  Observe that an object $x \in \scr{C}$ as above is $p$-initial if and only for every $n\geq 1$ the lifting problems
  \[
    \begin{tikzcd}
      \partial \Delta^n \arrow[r,"\varphi"] \arrow[d] & \scr{C} \arrow[d,"\pi"] \\
      \Delta^n \arrow[r] \arrow[ur,dotted] & \scr{D}
    \end{tikzcd}
  \]
  admit a solution provided $\varphi(0)=x$.
\end{remark}

\begin{lemma}\label{lem:lambda0loc}
  Let $p:\XX \to \SB$ be a local $(0,1)$-fibration. Then an edge $e:\Delta^1 \to \XX$ is locally $(0,1)$-Cartesian if and only if for every $n\geq 2$ the lifting problems of the form
   \[
    \begin{tikzcd}
      \Lambda^n_0 \arrow[d]  \arrow[r,"\varphi"] \arrow[d] & \XX \arrow[d,"p"] \\
      \Delta^n \arrow[r,"\phi"]\arrow[ur,dotted] & \SB
    \end{tikzcd}
  \]
  admit a solution provided $\varphi(0\to 1)=e$, $\varphi(0\to 1 \to n)$ is thin whenever $n\geq 2$ and $\phi(1 \to n)$ is an equivalence in $\SB$. If $n=2$ we require $\phi(0 \to 1 \to n)$ to be thin, $\phi(1 \to n)$ to be an equivalence in $\SB$ and that our solution is a thin simplex in $\bcat{X}$.
\end{lemma}
\begin{proof}
First let us remark that if we can produce solutions to those lifting problems $e$ must be a $(0,1)$-Cartesian edge once restricted to $\Delta^1$ and the claims follow from \autoref{prop:deflocedge}. We now prove the converse.

  It is clear that the case $n=2$ is precisely the first condition in \autoref{def:locally01edge}. To tackle the cases $n\geq 3$ we will assume once more that  $p: \Nsc(\CC) \to \Nsc(\DD)$. Observe that since $\varphi(0 \to 1 \to n)$ is thin we can solve the lifting problem
   \[
    \begin{tikzcd}
      \mathfrak{C}^{\mathbf{sc}}[\Lambda^n_0](0,n) \arrow[d]  \arrow[r] \arrow[d] &  \CC(\varphi(0),\varphi(n)) \arrow[d] \\
      \mathfrak{C}^{\mathbf{sc}}[\Delta^n](0,n) \arrow[r]\arrow[ur,dotted] & \DD(p\varphi(0),p\varphi(n)).
    \end{tikzcd}
  \]
  To conclude the proof we must show that we can produce the dotted arrow below
   \[
    \begin{tikzcd}
      \mathfrak{C}^{\mathbf{sc}}[\Lambda^n_0](1,n) \arrow[d]  \arrow[r] \arrow[d] &  \CC(\varphi(1),\varphi(n)) \arrow[d,"\phi"] \\
      \mathfrak{C}^{\mathbf{sc}}[\Delta^n](1,n) \arrow[r]\arrow[ur,dotted] &  \CC(\varphi(0),\varphi(n) \times_{\DD(p\varphi(0),\varphi(n))} \DD(p\varphi(1),p\varphi(n)).
    \end{tikzcd}
  \]
  However, by $ii)$ in \autoref{def:locally01edge} the object $1n$ on the left-hand side gets mapped to a $\phi$-initial object in $\CC(\varphi(1),\varphi(n))$. Since the left-most vertical map can be obtained as an iterated pushout along boundary inclusions $\partial \Delta^n \to \Delta^n$, where the initial object is always $1n$ we conclude that the dotted arrow above can be constructed.
\end{proof}

\begin{lemma}\label{lem:wonky}
  Let $p:\XX \to \SB$ be a local $(0,1)$-fibration. Suppose that we are given a simplex $\sigma:\Delta^4 \to \XX$ such that the following collection of triangles 
  \[
     T=\{\Delta^{\{0,2,4\}}, \ \Delta^{\{ 1,2,3\}}, \ \Delta^{\{0,1,3\}}, \ \Delta^{\{1,3,4\}}, \ \Delta^{\{0,1,2\}}\}
   \] 
   gets mapped to 2-simplices representing Cartesian 2-morphisms in the corresponding mapping categories. Then the triangles $\Delta^{\{0,1,4\}}$ and $\Delta^{\{0,3,4\}}$ also represent Cartesian 2-morphisms in $\XX(\sigma(0),\sigma(4))$.
\end{lemma}
\begin{proof}
  As usual, we will assume that our functor is of the for $p:\Nsc(\CC) \to \Nsc(\DD)$. This allows us to reduce our problem to show that certain edges in $\mathcal{P}=\mathfrak{C}^{\mathbf{sc}}[\Delta^4](0,4)$ get mapped to Cartesian edges in $\CC(\sigma(0),\sigma(4))$. More specefically, we view $\mathcal{P}$ as the following poset
  
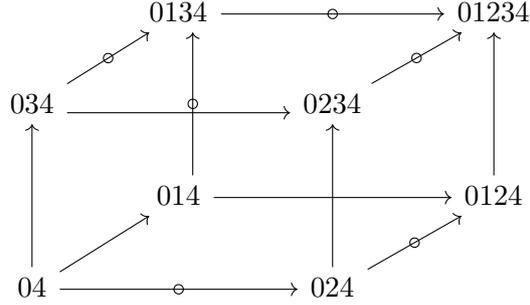
\[\begin{tikzcd}
  & 0134 && 01234 \\
  034 && 0234 \\
  {} & 014 && 0124 \\
  04 && 024
  \arrow[from=4-1, to=3-2]
  \arrow[from=3-2, to=3-4]
  \arrow[from=4-1, to=4-3,circled]
  \arrow[from=4-3, to=3-4,circled]
  \arrow[from=4-1, to=2-1]
  \arrow[shift right=1,from=2-1, to=2-3]
  \arrow[from=2-1, to=1-2,circled]
  \arrow[from=1-2, to=1-4,circled]
  \arrow[shift right=2,from=3-2, to=1-2,circled]
  \arrow[from=4-3, to=2-3]
  \arrow[from=3-4, to=1-4]
  \arrow[from=2-3, to=1-4,circled]
\end{tikzcd}\]
where the circled arrows are mapped by assumption to Cartesian edges (note that to see this is crucial to use that $p$ is Cartesian-enriched) and we wish to show that $04 \to 014$ and $04 \to 034$ are mapped to Cartesian edges in $\CC(\sigma(0),\sigma(4))$. 

Since $\pi:\CC(\sigma(0),\sigma(4)) \to \DD(p\sigma(0),p\sigma(4))$ is a Cartesian fibration this is equivalent to require that certain morphisms are equivalences in the fibre over $\pi(04)=\alpha$. Using the functoriality of $\pi$ we can move the diagram above to a diagram in the fibre over $\alpha$ where now the circled arrow are equivalences. It is ease to see that we can produce now an inverse as in Proposition 3.1.13 in \cite{LurieGoodwillie}.
\end{proof}

\begin{theorem}\label{thm:comparison}
  Let $p: X \to S$ be a fibrant object in $\left(\on{Set}^{\mathbf{mb}}_\Delta\right)_{/S}$. Then its associated map of scaled simplicial sets (see \autoref{prop:piscatfib}) $p:\bcat{X} \to \bcat{S}$ is a local $(0,1)$-fibration. Conversely, any local $(0,1)$-fibration defines canonically a fibrant object in $\left(\on{Set}^{\mathbf{mb}}_\Delta\right)_{/S}$.
\end{theorem}
\begin{proof}
  Let $p: X \to S$ be a $(0,1)$-local fibration. Then it follows from \autoref{prop:piscatfib} and \autoref{prop:cartenriched} that $p:\bcat{X} \to \bcat{S}$ is Cartesian-enriched. To show that our map in question defines a local $(0,1)$-fibration it will enough to show that the marked edges in $X$ are local $(0,1)$-Cartesian edges. However, one easily sees that a marked edge $e$ is $(0,1)$-Cartesian over $\Delta^1$ and thus the claim follows.

  To prove the converse we need to construct a fibrant object in $\left(\on{Set}^{\mathbf{mb}}_\Delta\right)_{/S}$ from a local $(0,1)$-fibration $p: \XX \to \bcat{S}$. We define an $\bS$-simplicial set $\on{F}_{\XX}$ as follows:
  \begin{itemize}
    \item The underlying simplicial set of $\on{F}_{\XX}$ is $\XX$.
    \item An edge is declared to be marked if an only if it is a local $(0,1)$-Cartesian edge in $\XX$.
    \item A triangle is declared to be lean if its associated 2-morphism is a Cartesian edge in $\XX(a,b)$.
    \item A triangle is declared to be thin if it is lean and its image in $\SB$ belongs to $M_S$. 
  \end{itemize}
This definition clearly yields a map $q:\on{F}_{\XX} \to (S,\sharp,M_S \subseteq \sharp)$ so it is only left to show that $q$ has the right lifting property against the class of morhpisms in \autoref{def:mbsanodyne}. It follows from \autoref{lem:innerhorn} and \autoref{lem:lambda0loc} that $q$ has the RLP property with respect to the morphisms of type \ref{mb:innerhorn} and \ref{mb:2coCartesianmorphs} in \autoref{def:mbsanodyne}. \autoref{lem:wonky} shows $q$ has the right lifting property against the class of morphisms of type \ref{mb:wonky4}. The rest of the lifting problems follow immediately and thus our result is proved.
\end{proof}

\begin{proposition}\label{prop:localequiv}
 Let $f:\bcat{X} \to \bcat{Y}$ be a morphism of local $(0,1)$-fibrations. Then the following are equivalent:
 \begin{itemize}
   \item[i)] The map $f$ is a bicategorical equivalence. 
   \item[ii)] For every $s \in \bcat{S}$ the map $f$ induces an equivalence on fibres $f_s: \bcat{X}_s \xlongrightarrow{\simeq} \bcat{Y}_s$.
 \end{itemize}
\end{proposition}
\begin{proof}
  Let us suppose that $f$ is a bicategorical equivalence and let $s \in \bcat{S}$. First we show that $f_s$ is essentially surjective. Given $y_s \in \bcat{Y}_s$, we use that $f$ is essentially surjective to get some $x \in \bcat{X}$ such that $f(x) \simeq y$. We denote by $u$ the image of the equivalence $f(x)\isom y$ under $q: \bcat{Y} \to \bcat{S}$ and use the fact that $p: \bcat{X} \to \bcat{S}$ is an isofibration to get an equivalence $v:x \to x_s$ where $p(x_s)=s$. It follows that $f(v)$ is again an equivalence and therefore defines a local $(0,1)$-Cartesian edge in $\bcat{Y}$ which allows to construct an equivalence $f(x_s) \isom y_s$ lying over the identity on $s$. 
  To show fully faithfulness of $f_s$ we consider $a,b \in \bcat{X}_s$ and observe that since we have a map of Cartesian fibrations
    \[
     \begin{tikzcd}
       \bcat{X}(a,b)\arrow[rr]  \arrow[dr] && \bcat{Y}(f(a),f(b)) \arrow[dl] \\
       & \bcat{S}(s,s) &   
     \end{tikzcd}
   \]
  which is an equivalence by assumption. It then follows that we have an equivalence after taking the fibre over the identity map on $s$. This morphism is then identified using \autoref{prop:deflocedge} with the map
  \[
     \bcat{X}_s(a,b) \xlongrightarrow{\simeq} \bcat{Y}_s(f(a),f(b))
  \]
  which shows that $f_s$ is fully faithful. 

  Show that the converse holds we note that by assumptions $f$ is already essentially surjective. It will then be enough to show that for every $a,b \in \bcat{X}$ and every $\alpha:p(a) \to p(b)$ the induced morphism on fibres
  \[
    \bcat{X}(a,b)_{\alpha} \xlongrightarrow{} \bcat{Y}(f(a),f(b))_{\alpha}
  \]
is a categorical equivalence. By picking a Cartesian lift $e:a \to \hat{b}$ such that $p(e)=\alpha$ we can again use \autoref{prop:deflocedge} to produce a commutative diagram
\[
  \begin{tikzcd}
    \bcat{X}_{p(b)}(\hat{b},b) \arrow[r,"\simeq"] \arrow[d,"\simeq"] &  \bcat{Y}_{p(b)}(f(\hat{b}),f(b)) \arrow[d,"\simeq"] \\
      \bcat{X}(a,b)_{\alpha} \arrow[r] & \bcat{Y}(f(a),f(b))_{\alpha}
  \end{tikzcd}
\]
where we use 2-out-of-3 to conclude that the bottom horizontal morphism is a weak equivalence and thus our claim holds.
\end{proof}

\begin{proposition}\label{prop:trivfibMB}
  Let $T$ be a scaled simplicial set and consider a morphism $f:X \to Y$ of $(0,1)$-fibrations in $\left(\on{Set}^{\mathbf{mb}}_\Delta\right)_{/T}$. Suppose further that the following conditions hold:
  \begin{enumerate}
    \item The map $f$ has the right lifting property against the class of $\bS$-anodyne morphisms.
    \item For every $t \in T$ the induced morphism $f_t: X_t \to Y_t$ is a bicategorical equivalence. 
  \end{enumerate}
  Then $f$ is a trivial fibration of $\bS$-simplicial sets.
\end{proposition}
\begin{proof}
  We claim that $f$ is a trivial fibration of $\bS$-simplicial sets if and only if for every minimally scaled simplex $\Delta^n_\flat$ and every morphism $\sigma: \Delta^n_\flat \to T$ the restricted morphism
  \[
    f_{|\sigma}: X \times_{\Delta^n_{\flat}}\Delta^n_{\flat} \xlongrightarrow{} Y \times_{\Delta^n_{\flat}}\Delta^n_{\flat}
  \]
  is a trivial fibration of $\bS$-simplicial sets. One direction is obviously true. Let us assume that $f_{|\sigma}$ is always a trivial fibration. Then it is clear that $f$ has the right-lifting property against the morphisms:
  \begin{itemize}
    \item $(\partial\Delta^n,\flat,\flat) \to (\Delta^n,\flat,\flat)$.
    \item $(\Delta^2,\flat,\flat) \to (\Delta^2,\flat,\flat\subset \sharp)$.
    \item $(\Delta^1,\flat,\sharp) \to (\Delta^1,\sharp,\sharp)$. 
  \end{itemize}
  Since a thin triangle in $X$ is just a lean triangle lying over a thin triangle in $T$ it follows that $f$ also detects thin triangles and the claim holds.

  Let us assume without loss of generality that $T=\Delta^n_\flat$ and consider the associated diagram (see \autoref{thm:comparison}) of local $(0,1)$-fibrations
  \[
    \begin{tikzcd}
      \XX \arrow[dr,swap,"p"] \arrow[rr,"f"] && \bcat{Y} \arrow[dl,"q"] \\
      &  \Delta^n & 
    \end{tikzcd}
  \]
  First let us show that $f$ is a trivial fibration of $\infty$-bicategories. Observe that our assumptions together with \autoref{prop:localequiv} imply that $f$ is a bicategorical equivalence. Moreover, $f$ has the right lifting property against the class of scaled anodyne maps given in \autoref{def:scanodyne}. It will then suffice to show that $f$ is an isofibration. Given an equivalence in $e:\Delta^1 \to \bcat{Y}$ it follows that its image in $\Delta^n$ must be degenerate. Since this lifting problem is ocurring in a fibre and by our assumptions the maps $f_t: X_t \to Y_t$ are trivial fibrations of scaled simplicial sets it follows that $f$ is an isofibration.

  To finish the proof, we must show that $f$ detects local $(0,1)$-Cartesian edges. Given an edge $e:\Delta^1 \to \XX$ such that $f(e)$ is a local $(0,1)$-Cartesian edge we consider a local $(0,1)$-Cartesian edge $u: \Delta^1 \to \XX$ such that $u(0)=e(0)$ and such that $f(u)=f(e)$. It follows that we have an edge $\alpha:u(1) \to e(1)$ such that $\alpha \circ u \isom e$. Moreover, $p(\alpha)$ is an equivalence and lies over a degenerate morphism in $\Delta^n$. We see then, that $\alpha$ must be an equivalence in $\XX$ and consequently $e$ is a local $(0,1)$-Cartesian edge.
\end{proof}

\begin{definition}\label{def:sigmalocaledge}
  Let $p:\XX \to \SB$ be bicategorical fibration and let $\sigma:\Delta^2 \to \SB$ be a thin triangle. We say that an edge $e:a \to b$ in $\XX$ lying over $\sigma(0 \to 1)$ is $\sigma$-local if the following condidtions hold:
\begin{itemize}
  \item[i)] For every $g:a \to c$ in $\XX$ lying over $\sigma(1 \to 2)$ there exists some $\hat{\alpha}: b \to c$ and a thin simplex $\theta$ exhibiting $e \circ \hat{\alpha} \isom g$ such that $p(\theta)=\sigma$.
  \item[ii)] For any $\phi: b \to c$ such that $e \circ \phi \isom g$ with associated simplex $\tau$ such that $p(\tau)=\sigma$ and for any $\varphi:b \to c$ precomposition along $e$ induces a pullback diagram of spaces
  \[
   \begin{tikzcd}
    \on{Map}_{\bcat{X}(b,c)}(\phi,\varphi) \arrow[r] \arrow[d] &    \on{Map}_{\bcat{X}(a,c)}(\phi\circ e,\varphi \circ e ) \arrow[d] \\
      \on{Map}_{\bcat{S}(p(b),p(c))}(p(\phi),p(\varphi)) \arrow[r] & \on{Map}_{\bcat{S}(p(a),p(c))}(p(\phi\circ e),p(\varphi \circ e )).
   \end{tikzcd}
 \] 
\end{itemize}
\end{definition}

\begin{remark}
  Observe that \autoref{def:sigmalocaledge} shows that an edge is $(0,1)$-Cartesian if and only if it is $\sigma$-local for every thin simplex $\sigma:\Delta^2 \to \SB$. Similarly, an edge is locally $(0,1)$-Cartesian if and only if it is $\sigma$-local for every  invertible 2-morphism (see \autoref{def:mstri}).
\end{remark}

\begin{proposition}\label{prop:ulocalcompose}
  Let $p:\XX \to \SB$ be a local $(0,1)$-fibration. Then a local $(0,1)$-Cartesian edge, $e:a \to b$ is $\sigma$-local if and only if for every local $(0,1)$-Cartesian edge $u:b \to c$ such that the composite $v\isom u \circ e$ lies over $\sigma$ then $v$ is also a local $(0,1)$-Cartesian edge.
\end{proposition}
\begin{proof}
  Let us assume that $e$ is $\sigma$-local and suppose that we have $u$ and $v$ as above. Let us suppose that we have $h:a \to d$ and let us show that condition $i)$ in \autoref{def:locally01edge} is satisfied.

  Since $p$ is a local $(0,1)$-fibration and in particular, Cartesian-enriched we only need to show that $v$ is $(0,1)$-Cartesian once after pulling back along $p(v)$. Therefore, we can assume without loss of generality that $p(h)=p(v)$. We observe that since $e$ is $\sigma$-local we can obtain a morphism $\alpha:b \to d$ such that $\alpha \circ e \isom h$. Furtheremore, we can use that $u$ is a local $(0,1)$-Cartesian edge to get a morphism $\phi:c \to d$ such that $\alpha \isom \phi \circ u$. It follows that $h \isom v \circ \phi$ and so the first condition holds.

  Given $\phi:c \to d$ as above and another $\varphi: c \to d$ such that $p(\varphi)=\on{id}$ we construct the following commutative diagram
  \[
    \begin{tikzcd}
      \on{Map}_{\XX(c,d)}(\phi,\varphi) \arrow[d] \arrow[r] & \on{Map}_{\XX(b,d)}(\phi\circ u,\varphi \circ u) \arrow[r] \arrow[d] &  \on{Map}_{\XX(a,d)}(\phi\circ v,\varphi \circ v) \arrow[d] \\
       \on{Map}_{\SB(p(c),p(d))}(\on{id},\on{id}) \arrow[r] &  \on{Map}_{\SB(p(b),p(d))}(p(u),p(u)) \arrow[r] &  \on{Map}_{\SB(p(a),p(d))}(p(v),p(v))            
    \end{tikzcd}
  \]
  We observe that the outer commuative diagram is obtained by pasting two pullback diagrams so it must be itself a pullback diagram. It follows that $v$ is a local $(0,1)$-Cartesian edge.

  We wish now to show that the converse holds. Let $h: a \to d$ be an edge over $\sigma(0\to 2)$. We take a local $(0,1)$-Cartesian lift $u: b \to c$ of $\sigma(1 \to 2)$. Since by assumption $v \isom u \circ e$ is again local $(0,1)$-Cartesian we obtain a certain $\Xi: c \to d$ such that $h \isom \Xi \circ v$. We can then set $\phi= \Xi \circ u$.  It is then clear that $ \phi \circ e= \Xi \circ u \circ e \isom \Xi \circ v \isom h$ and thus condition $i)$ in \autoref{def:sigmalocaledge} holds.  Let $\phi: b \to d$ as above and assume we are given any other $\varphi: b \to d$.  We wish to show that the associated commutative diagram (see $ii)$ in \autoref{def:sigmalocaledge}) of spaces is Cartesian. We note that a totally analogous argument as in \autoref{prop:deflocedge} shows that it is enough to show that the associated map of fibres
  \[
     \on{Map}_{\bcat{X}(b,d)}(\phi,\varphi)_{p(u)}  \xlongrightarrow{\simeq}  \on{Map}_{\bcat{X}(a,e)}(\phi\circ e,\varphi \circ e)_{p(v)} 
  \]
is an equivalence whenever $p(\phi)=p(\varphi)$. Since $u$ is a local $(0,1)$-Cartesian edge we can find morphisms $\widetilde{\phi},\widetilde{\varphi}: c \to d$ such that $\widetilde{\phi} \circ u \isom \phi$ and $\widetilde{\varphi}\circ u =\varphi$. We can then produce the morphisms
\[
  \on{Map}_{\XX(c,d)}(\widetilde{\phi},\widetilde{\varphi})_{\on{id}} \xlongrightarrow{\simeq}\on{Map}_{\bcat{X}(b,d)}(\phi,\varphi)_{p(u)}  \xlongrightarrow{}  \on{Map}_{\bcat{X}(a,e)}(\phi\circ e,\varphi \circ e)_{p(v)}. 
\]
We conclude the proof by noting that the composite map must also be a weak equivalence since $v$ is by assumption a local $(0,1)$-Cartesian edge.
\end{proof}

\begin{definition}
  Let $\SB$ be an $\infty$-bicategory and let $U$ be a subcollection of the thin triangles in $\SB$ which contains the invertible 2-morphisms of $\SB$ (see \autoref{def:mstri}). We say that a local $(0,1)$-fibration $p: \XX \to \SB$ is $\mathcal{U}$-local if given a pair of local $(0,1)$-Cartesian edge $u,v:\Delta^1 \to \XX$ and a thin 2-simplex $\sigma$ pictured below
 \[
    \begin{tikzcd}
     & b \arrow[dr,"v"] & \\
     a \arrow[ur,"u"] \arrow[rr,"w"] & & c
   \end{tikzcd} 
 \]
such that $p(\sigma) \in \mathcal{U}$ then we have that $w$ is also locally $(0,1)$-Cartesian. If $\mathcal{U}$ consists in all thin triangles we say that $p:\XX \to \SB$ is a $(0,1)$-Cartesian fibration.
\end{definition}

\begin{theorem}\label{thm:ulocal}
  Let $(S,T_S)$ be a fibrant scaled simplicial set and let $U \subset T_S$ be a subset containing all invertible 2-morphisms denote $\overline{S}=(S,U)$. Given a fibrant object $\left(\on{Set}^{\mathbf{mb}}_\Delta\right)_{/\overline{S}}$ then its associated map of scaled simplicial sets $p: \XX \to \SB$ defines a $\mathcal{U}$-local fibration. Conversely any $\mathcal{U}$-local fibration defines canonically a fibrant object in $\left(\on{Set}^{\mathbf{mb}}_\Delta\right)_{/\overline{S}}$.
\end{theorem}
\begin{proof}
  Let us assume that we are given a fibrant object in $\left(\on{Set}^{\mathbf{mb}}_\Delta\right)_{/\overline{S}}$. In particular, we can use \autoref{thm:comparison} to obtain a local $(0,1)$-fibration $p:\XX \to \SB$. Since our original object has the right lifting property against the class \ref{mb:composeacrossthin} it follows that our local $(0,1)$-Cartesian edges compose across triangles which lie over triangles in $\mathcal{U}$ and the claim follows.

  We now show the converse.  Note that due to \autoref{prop:ulocalcompose} our local $(0,1)$-Cartesian edges are $\sigma$-local with respect to the elements of $\mathcal{U}$. The only thing that we need to prove is that given an $\mathcal{U}$-local fibration we can produce the dotted arrow below
  \[
    \begin{tikzcd}
      \Lambda^n_0 \arrow[r,"f"] \arrow[d] & \XX \arrow[d,"p"] \\
      \Delta^n \arrow[r] \arrow[ur,dotted] & \SB
    \end{tikzcd}
  \]
  where $f(0\to 1)$ is locally $(0,1)$-Cartesian and $f(0 \to 1 \to n)$ lands in $\mathcal{U}$. The proof of this fact is essentially the same as the proof in \autoref{lem:lambda0loc} and therefore, left as an exercise.
\end{proof}

\section{The Grothendieck construction }
Let $S=(S,T_S)$ be a scaled simplicial set and let $\mathfrak{C}^{\mathbf{sc}}[S]$ denote the scaled rigidification (\autoref{def:rigidification}) of $(S,T_S)$. The goal of this section is to prove the following theorem.
\begin{thm*}
  Let $S$ be a scaled simplicial set. Then there exists a Quillen equivalence 
  \[
    \SSt_{S}: \left(\on{Set}_\Delta^{\mathbf{mb}}\right)_{/S}  \llra  \on{Fun}(\mathfrak{C}^{\mathbf{sc}}[S],\on{Set}_\Delta^{\mathbf{ms}}): \UN_{S}
  \]
  between the model structure on $(0,1)$-Cartesian fibrations over $S$ and the projective model structure of $\on{Set}_\Delta^+$-enriched functors with values in marked-scaled simplicial sets.
\end{thm*}

Our first order of business will be to define the left adjoint $\SSt_{S}$ which will be given by a 2-categorical enhancement of the straightening functor constructed in Section 3 of \cite{LurieGoodwillie}. Before we present our main construction, we need to give some preliminary definitions.

\begin{definition}
  Let $X,Y \in \on{Set}_{\Delta}^{\mathbf{ms}}$. We define the Gray tensor product $X \tensor Y \in \on{Set}_\Delta^{\on{sc}}$ (see Definition 4.1.1 in \cite{GHL_LaxLim}) as follows:
  \begin{enumerate}
    \item The underlying simplicial set of $X \tensor Y$ is given by $X \times Y$, the Cartesian product of the underlying simplicial sets.
    \item Given a simplex $\sigma: \Delta^2 \to X \tensor Y$ let us denote $\sigma_X$ and $\sigma_Y$ the projections to the corresponding factors in the Cartesian product. We say that $\sigma$ is scaled in $X \tensor Y$ if and only if the following conditions holds
   \begin{itemize}
     \item[i)] The simplex $\sigma$ is both scaled in $X$ and in $Y$.
     \item[ii)] The restriction $\sigma_X(1 \to 2)$ is marked in $X$ \emph{or} the restriction $\sigma_Y(0 \to 1)$ is marked in $Y$.
   \end{itemize}
  \end{enumerate}
\end{definition}

\begin{definition}\label{def:pninitialdef}
  Let $n\geq 0$. We define a poset $P_n$ as follows:
  \begin{itemize}
    \item The objects are given by subsets $S \subseteq [n]$ such $S \neq \emptyset$ and $\max(S)=n$.
    \item We define a partial order on $P_n$ by declare $S \leq T$ whenever $\min(S) \leq \min(T)$ and there exists some $U$ such that $\min(U)=\min(S)$ and $\max(U)=\min(T)$ and such that $S \subseteq U \cup T$. 
  \end{itemize}
\end{definition}

\begin{remark}\label{rem:realorder}
  Observe that in the definition above $U \leq V$ if and only if $\min(U) \leq \min(V)$ and for every $x \in U$ such that $x \geq \min(V)$ then $x \in V$. Moreover, we can identify those inequalities $U < V$ in $P_n$ which cannot be decomposed as $U < W < V$ as
  \begin{itemize}
    \item[O1)] We have $U< V$ with $\min(U)=\min(V)$  and $V= U \cup \{s\}$.
    \item[O2)] We have $U < V$ with $V=U \setminus {\min(U)}$ with $\min(V)=\min(U)+1$.
  \end{itemize}
\end{remark}

\begin{remark}\label{rem:ust}
  We observe that given $S,T \in P_n$ such that $S \leq T$ we can have several subsets $U$ as above such that $S \subseteq U \cup T$. Moreover, we can order such subsets by inclusion and define $U_{S,T}$ to be the minimal subset such that $S \subseteq U_{S,T} \cup T$. Let $\min{S}=s$ and let $\min{T}=t$ we can then describe as $U_{S,T}=\{s,t\} \cup \{s <i <t \enspace | \enspace i \in S\}$.
\end{remark}

\begin{definition}\label{def:pnscaled}
  Let $\mathcal{P}_n=\on{N}(P_n)$. We promote $\mathcal{P}_n$ to a scaled simplicial set as follows. Given a 2-simplex $\sigma$ represented by $S \leq T \leq W$ we declare $\sigma$ to be thin if $U_{S,W}=U_{S,T} \cup U_{T,W}$.
\end{definition}

\begin{remark}\label{rem:maptopn}
  Let $\Delta^n_{\flat}=(\Delta^n,\flat,\flat)$ and let $\Delta^1 \tensor_\flat \Delta^n=\Delta^1_{\flat} \tensor \Delta^n_{\flat}$. We consider $\mathfrak{C}^{\mathbf{sc}}[\Delta^1 \tensor_\flat \Delta^n]$. Recall that given $(i,j) \leq (k,\ell)$ in $\Delta^1 \times \Delta^n$ we have that $\mathfrak{C}^{\mathbf{sc}}[\Delta^1 \tensor_\flat \Delta^n]((i,j),(k,\ell))$ is given by the nerve of the poset of chains $C$,
  \[
    (i,j)=(i_0,j_0)< (i_1,j_1)< \cdots <(i_{\alpha-1},j_{\alpha-1})<(i_{\alpha},j_\alpha)=(k,\ell)
  \]
  ordered by refinement. Let us suppose that $i=0$ and that $k=1$. Then, given a chain $C=\{(i_{\alpha},j_{\alpha})\}_{\alpha \in A}$ we can define $m_C$ to be the biggest index in $A$ such that $i_{m_C}=0$. This allows us to define a map
\[
  \pi_{j,\ell}:\mathfrak{C}^{\mathbf{sc}}[\Delta^1 \tensor_\flat \Delta^n]\left((0,j),(1,\ell)\right) \xlongrightarrow{} \mathcal{P}_{[j,\ell]}, \enspace C \mapsto \bigcup_{\alpha \geq m_C}j_{\alpha}.
\]
This assignment is clearly a map of posets which sends marked edges in our mapping simplicial set to identities in $\mathcal{P}_{[j,l]}$. We use the map $\pi_{j,\ell}$ to equipp the left-hand side with the scaling induced by $\mathcal{P}_{[j,l]}$.
\end{remark}

\begin{definition}\label{def:msenriched}
 We define a colimit preserving functor $\Pi: \on{Set}_\Delta^{\mathbf{mb}} \to \on{Cat}_{\Delta}^{\mathbf{ms}}$ with values in the category of $\on{Set}_\Delta^{\mathbf{ms}}$-enriched categories by specifying its values on the generators under colimits of $\on{Set}_\Delta^{\mathbf{mb}}$ as follows:
 \begin{itemize}
   \item[1)] Given a minimally marked and biscaled simplex $\Delta^n_\flat=(\Delta^n,\flat,\flat)$ we define $\Pi(\Delta^n)$ to have as underlying $\on{Set}_{\Delta}^+$-category the scaled rigidification of $\Delta^1 \tensor_\Pi \Delta^n$ where the later denotes the Gray tensor product $\Delta^1 \tensor_{\flat} \Delta^n$ defined in \autoref{rem:maptopn} with the additional scaling consisting in all of the 2-simplices contained in $\Delta^{\{0\}} \times \Delta^n$. Given $(i,j) <(k,\ell)$ in $\Pi(\Delta^n_\flat)$ we equipp the mapping simplicial sets with a scaling by declaring every triangle to be scaled if $i \neq 0$ and $k \neq 1$. If $i=0$ and $k=1$ we scale $\Pi(\Delta^n_\flat)((0,j),(1,\ell))$ according to \autoref{rem:maptopn}.
   \item[2)] Given a lean scaled $2$-simplex, i.e. $\Delta^2_\dagger=(\Delta^2,\flat,\flat\subset \sharp)$ we define $\Pi(\Delta^2_\dagger)$ from $\Pi(\Delta^2_\flat)$ by scaling every triangle in the mapping simplicial sets.
   \item[3)] Given a thin scaled $2$-simplex, i.e. $\Delta^2_\sharp=(\Delta^2,\flat,\sharp)$ we define $\Pi(\Delta^2_\sharp)$ from $\Pi(\Delta^2_\dagger)$ by additionally marking every morphism in $\Pi(\Delta^2_\sharp)((0,0),(1,2))$ which gets maps under the map in \autoref{rem:maptopn} to the morphism $02 \to 012$ in $\mathcal{P}_2$.
   \item[4)] Given a marked edge $(\Delta^1)^\sharp=(\Delta^1,\sharp)$ we can identify $\Pi((\Delta^1)^{\sharp})= \mathfrak{C}^{\mathbf{sc}}[\Delta^1 \times \Delta^1]$.
 \end{itemize}
 One easily checks that our choice of decorations is compatible with composition and thus our definition yields well defined $\on{Set}_\Delta^{\mathbf{ms}}$-enriched categories and that our definition is functorial on the set of generators of $\on{Set}_{\Delta}^{\mathbf{mb}}$. Since $ \on{Cat}_{\Delta}^{\mathbf{ms}}$ is cocomplete our functor can be extended by colimits and the definition is complete.
\end{definition}

\begin{definition}\label{def:Pi_S}
  Let $j:\on{Cat}_{\Delta }^+\to \on{Cat}_{\Delta}^{\mathbf{ms}}$ be the functor that scales every 2-simplex in the mapping simplicial sets. Given a scaled simplicial set $S$ we define a functor
  \[
    \Pi_S: \left(\on{Set}_\Delta^{\mathbf{mb}} \right)_{/S} \xlongrightarrow{} \on{Cat}_{\Delta}^{\mathbf{ms}}, \enspace X \mapsto \Pi(X)\coprod_{j\circ \mathfrak{C}^{\mathbf{sc}}[X]}j\circ \mathfrak{C}^{\mathbf{sc}}[S]
  \]
  where $ \mathfrak{C}^{\mathbf{sc}}[X]$ denotes the scaled rigidification of the \emph{underlying} scaled simplicial set of $X$ and where the morphism $j\circ \mathfrak{C}^{\mathbf{sc}}[X] \to \Pi(X)$ is given by the inclusion of $\Delta^{\{1\}}\times X$.

   We define a further functor, 
\[
  C_S: \left(\on{Set}_\Delta^{\mathbf{mb}} \right)_{/S} \to \on{Cat}_{\Delta}^{\mathbf{ms}}, \enspace X \mapsto \Pi_S(X)\coprod\limits_{j\circ  \mathfrak{C}^{\mathbf{sc}}[X_{\sharp}]} \Delta^0 
\]
where $X_{\sharp}$ denotes the underlying scaled simplicial set of $X$ equipped with the maximal scaling and the morphism $j\circ  \mathfrak{C}^{\mathbf{sc}}[X_{\sharp}] \to \Pi_S(X)$ is induced by the inclusion $\Delta^{0} \times X_{\sharp} \to \Delta^1 \tensor_\Pi X $.
\end{definition}

\begin{remark}
  From this point on we will drop the notation $j \circ \mathfrak{C}^{\mathbf{sc}}$ and we will view $\on{Set}_\Delta^+$-enriched categories as a full subcategory of $\on{Set}_\Delta^{\mathbf{ms}}$-enriched categories consisting in those enriched categories whose mapping simplicial sets are fully scaled.
\end{remark}

\begin{remark}\label{rem:basechange}
  Let $f: S \to S'$ be a map of scaled simplicial sets. Given $p:X \to S$ in $\left(\on{Set}_\Delta^{\mathbf{mb}} \right)_{/S}$  we claim that we have an isomorphism of $\on{Set}_\Delta^{\mathbf{ms}}$-enriched categories
  \[
    C_S(X)\coprod_{\mathfrak{C}^{\mathbf{sc}}[S]} \mathfrak{C}^{\mathbf{sc}}[S'] \xlongrightarrow{\simeq} C_{S'}(f_! X), 
  \]
  where $f_!X$ denotes the value of the functor $C_S'$ at the object $f \circ p: X \to S'$. The isomorphism on the underlying $\on{Set}_\Delta^+$-categories is clear. The only thing to show is that the scaling on mapping simplicial sets of the form  $ C_{S'}(f_! X)((0,x),(1,s'))$ is the same for both $\on{Set}_\Delta^{\mathbf{ms}}$-enriched categories. This follows after a direct inspection since the scaling on those simplicial sets  (which does not factor through some mapping simplicial set between objects $(1,s)$ and $(1,s')$ ) is independent of the base.
\end{remark}

\begin{definition}\label{def:st01}
  Let us denote by $v$ the collapsed point in the definition of $C_S(X)$. Then for every $p:X \to S$ and every morphism of $\on{Set}_\Delta^+$ enriched categories $\phi: \mathfrak{C}^{\mathbf{sc}}[S] \to \scr{C}$ we can define a functor
\[
  \SSt_\phi(X):\scr{C} \xlongrightarrow{} \on{Set}_\Delta^{\mathbf{ms}}, \enspace c \mapsto \on{Map}_{C_\phi} (v,c) , \enspace \text{where } C_{\phi}:= C_S(X) \coprod_{\mathfrak{C}^{\mathbf{sc}}[S]} \scr{C},
\]
which we call the \emph{straightening} of $p:X \to S$. This definition extends to a functor
\[
  \SSt_\phi: \left(\on{Set}_\Delta^{\mathbf{mb}} \right)_{/S}  \xlongrightarrow{} \on{Fun}(\scr{C}, \on{Set}_\Delta^{\mathbf{ms}})
\]
with values in the category of $\on{Set}_\Delta^+$-enriched functors. If $\phi$ is an isomorphism we will use the notation $\SSt_S$.
\end{definition}

\begin{remark}
  One can easily check that $ \SSt_\phi$ preserves all colimits in $\left(\on{Set}_\Delta^{\mathbf{mb}} \right)_{/S}$. It follows from the adjoint functor theorem that there exists a functor
  \[
    \UN_\phi: \on{Fun}(\C, \on{Set}_\Delta^{\mathbf{ms}}) \xlongrightarrow{} \left(\on{Set}_\Delta^{\mathbf{mb}} \right)_{/S}
  \]
  which we call the \emph{unstraightening} functor.
\end{remark}

\begin{proposition}\label{prop:Stbasechange}
  Let $f:S \to S^{\prime}$ be a map of scaled simplicial sets and commutative diagram of $\on{Set}_\Delta^+$-enriched categories
  \[
    \begin{tikzcd}
      \mathfrak{C}^{\mathbf{sc}}[S] \arrow[r,"\mathfrak{C}^{\mathbf{sc}}{[f]}"] \arrow[d,"\phi"]  & \mathfrak{C}^{\mathbf{sc}}[S^{\prime}] \arrow[d,"\phi'"] \\
      \C \arrow[r,"\psi"] & \C^{'}
    \end{tikzcd}
  \]
  then the following diagram commutes up to invertible natural transformation
  \[
    \begin{tikzcd}
      \left(\on{Set}_\Delta^{\mathbf{mb}} \right)_{/S} \arrow[r,"\SSt_\phi"] \arrow[d,"f_!"] & \on{Fun}(\C, \on{Set}_\Delta^{\mathbf{ms}}) \arrow[d,"\psi_!"] \\
      \left(\on{Set}_\Delta^{\mathbf{mb}} \right)_{/S^{\prime}} \arrow[r,"\SSt_{\phi^{\prime}}"] & \on{Fun}(\C^{\prime}, \on{Set}_\Delta^{\mathbf{ms}})
    \end{tikzcd}
  \]
  where $\psi_!$ is the left adjoint to the restriction functor $\psi^*$.
\end{proposition}
\begin{proof}
  Let $\theta=\psi \circ \phi$. It follows by direct inspection together with \autoref{rem:basechange} that $\SSt_\phi' \circ f_! \isom  \SSt_{\theta}$. In order to finish the proof we must show that $\psi_! \circ \SSt_\phi \isom \SSt_\theta$. It is clear that both functors agree except possibly on the scaling. However, given $p:X \to S$ the scaling in $\SSt_\theta(X)(c')$ is freely generated by the scaling in certain mapping simplicial sets of $\Pi(X)$ and consequently the claim holds. 
\end{proof}

\begin{remark}\label{rem:unbasechange}
  Observe  that in the situation above passing to right adjoints we obtain equivalences of functor $\UN_\phi \circ \psi^* \isom f^* \circ \UN_{\phi^{\prime}}$.
\end{remark}

\begin{lemma}\label{lem:stcof}
  For any simplicial set $(S,T_S)$ the functor $\SSt_S$ preserves cofibrations.
\end{lemma}
\begin{proof}
  Since $\SSt_S$ preserves colimits it will be enough to prove the claim on the generating class of cofibrations given in \autoref{def:gencof}. Moreover, given a cofibration $\alpha: A \to B$ we can use \autoref{prop:Stbasechange} to reduce to the case where $S=B$. The case $\emptyset \to \Delta^0$ is obviously true. For the rest of the generators we have that $B=(\Delta^n,\flat)$ for $n\geq 1$ or $B=(\Delta^2,\sharp)$ and that the map  $\SSt_B A (i) \to \SSt_B B(i)$ is the identity except when $i=n$ in which case the map is a cofibration. It is immediate to see that the map in this situation $\SSt_B A \to \SSt_B B$ has the left lifting property against the class of trivial fibrations.
\end{proof}

\begin{remark}
   Let $\iota:\on{Set}_\Delta^+ \to \on{Set}_\Delta^{\mathbf{ms}}$ be the functor defined by $\iota(X,E_X)=(X,E_X,\sharp)$. Given a scaled simplicial set $(S,T_S)$ let $\iota_*\on{St}_S: \left(\on{Set}_\Delta^{+}\right)_{/S} \to \on{Fun}(\mathfrak{C}^{\mathbf{sc}}[S],\on{Set}_\Delta^\mathbf{ms})$ be the straightening functor given in Definition 3.5.4 in \cite{LurieGoodwillie} post-composed with the enriched functor $\iota$. Recall the definition of the functor $R$ (see \autoref{def:inclusionmarked}), we can then define a functor
   \[
      \SSt_S \circ R: \left(\on{Set}_\Delta^+\right)_{/S} \xlongrightarrow{} \on{Fun}(\mathfrak{C}^{\mathbf{sc}}[S],\on{Set}_\Delta^\mathbf{ms}).
    \] 
    It follows that $ \iota_*\on{St}_S$ and $\SSt_S \circ R$ differ only in the scaling thus inducing a natural transformation $\eta_S:\SSt_S \circ R \xRightarrow{}  \iota_*\on{St}_S$.
\end{remark}

\begin{proposition}
 Let $(S,T_S)$ be a scaled simplicial set then the natural transformation $\eta_S:\SSt_S \circ R \xRightarrow{}  \iota_*\on{St}_S$ is an object-wise weak equivalence.
\end{proposition}
\begin{proof}
  It is clear that both functors preserve colimits. Moreover, a totally analogous proof to that of \autoref{lem:stcof} shows that $\iota_*\on{St}_S$ preserves cofibrations. It is also not hard to verify that $\iota_*\on{St}_S$ satisfies similar base change properties as those in \autoref{prop:Stbasechange}. We conclude that it will be enough to show that $\eta_k=\eta_{\Delta^k_\flat}(\Delta^k)$ is an equivalence for $k\geq 0$ and similarly for $\eta_1^{\sharp}=\eta_{(\Delta^1)^\sharp}((\Delta^1)^\sharp)$. It is easy to check that $\eta_0,\eta_1,\eta_2$ and $\eta_1^\sharp$ are all isomorphisms.

  For $k\geq 3$ let us define $\mathbb{L}^k=\Pi_{\Delta^k_\flat}((\Delta^k,\flat,\flat \subset \sharp))$ and $\mathbb{L}^k_\sharp$ by scaling every triangle in the mapping simplicial categories of $\mathbb{L}^k$. It is not hard to see that our problem can be reduced to showing that the map
  \[
    \varphi_k: \mathbb{L}^k \to \mathbb{L}^k_\sharp
  \]
  is an equivalence of $\on{Set}_\Delta^{\mathbf{ms}}$-enriched categories. Using induction on $k$ we can assume that the mapping simplicial sets $\mathbb{L}^k((0,i),(1,j))$ are maximally scaled except if $i=0$ and $j=k$. Let $\mathbb{L}^{k}((0,0),(1,k))=A^k$ and similarly $\mathbb{L}^k_\sharp((0,0),(1,k))=A^k_\sharp$. We observe that for every face $d_i: \Delta^{k-1} \to \Delta^{k}$ we have a commutative diagram
  \[
     \begin{tikzcd}
       A^{k-1} \arrow[r,"\simeq"] \arrow[d,"\alpha_i"] & A^{k-1}_\sharp \arrow[d] \\
       A^k \arrow[r] & A^k_\sharp
     \end{tikzcd}
   \] 
   where the top horizontal morphism is a weak equivalence. We can therefore we can assume inductively that the triangles of $A^k$ which are in the image of the maps $\alpha_i$ for $i=0,\dots,k$ are all scaled.

   Let $\sigma:C_0 \subset C_1 \subset C_2$ be a triangle in $A^k$ and let us fix the notation $C_j=\{(\epsilon^j_i,a_i^j\}_{i=0}^{\ell}$. Let $e=C_0 \subset C_1$ be an edge in $A^k$. We define $S(e)$ as the set of non-degenerate simplices with initial vertex $C_0$ and final vertex $C_1$. We finally set $|e|=\max\{\dim(\phi) | \enspace \phi \in S(e)\}$. Note that this a well defined number. Given a 2-simplex $\sigma$ as above we define $|\sigma|=|d_1(\sigma)|$. We will show that we can scale $\sigma$ using a pushout along a $\mathbf{MS}$-anodyne morphism using induction on $|\sigma|=\ell$. The case $\ell=2$ follows easily after direct inspection. So we will assume from this point on that $\ell>1$.

  Let us suppose that the claim holds for those 2-simplices $\tau$ such that $|d_2(\tau)|=1$ and for those 2-simplices $\sigma$ such that $|\sigma|< \ell-1$. Then given $\sigma: C_0 \subset C_1 \subset C_2$ such that $|d_2(\sigma)|>1$ and such that $|\sigma|=\ell$ we can construct a 3-simplex $\rho: C_0 \subset D \subset C_1 \subset C_2$ such that the following holds:
   \begin{enumerate}
     \item We have $|d_2d_3(\rho)|=1$ in particular $d_i(\rho)$ is thin scaled for $i=2,3$.
     \item We have  $|d_0(\rho)|=\ell-1$ and is therefore scaled.
   \end{enumerate}
   It follows that we can fully scale $\rho$ using a pushout along a morphism of type \ref{ms:wonky4} in \autoref{def:msanodyne}. Therefore, we have reduced our problem to proving the claim above. {}

   Let $\sigma: C_0 \subset C_1 \subset C_2$ such that $|d_2(\sigma)|=1$. We observe that unless $C_0=(0,0)<(1,k)$ then $\sigma$ is scaled. Otherwise we could express $\sigma$ as a certain composition in the category $\mathbb{L}^k$ and it would follow from the induction hypothesis that $\sigma$ is thin. 

   We can now see that $C_1=(0,0)<(\epsilon,a) <(1,k)$ which leads us to consider cases depending on the parameter $\epsilon \in \{ 0,1\}$.

   \begin{itemize}
     \item[$\epsilon=1$)] Then we can assume without loss of generality that $C_2$ contains an element of the form $(0,x)$ in $C_2$ with $x \neq 0$. This is true since otherwise the maps $\pi_{0,k}$ in \autoref{rem:maptopn} would show that $\sigma$ is already scaled. Moreover, we can further assume that there is only element of the form $(0,x)$ in $C_2$. Indeed, if we had some $(0,y)<(0,x)$  then we can produce a 3-simplex $\rho:C_0 \to C_1 \to \widetilde{C}_2 \to C_2$ where $\widetilde{C}_2=C_2 \setminus \{(0,y)\}$. Since $d_0(\rho)$ and $d_1(\rho)$ must be scaled by definition it follows that we can scale $d_2(\rho)$ if and only if we can scale $d_3(\rho)$ and thus the claims follows. Additionally, we note that if $x\neq 1$ then $\sigma$ factors through one of the morphisms $\alpha_j: A^{k-1} \to A^k$ above. We finally see that in this case $\pi_{0,k}(\sigma)$ is given by a simplex of the form $0n \to 0an \to S$ with $\min(S)=1$ and it is consequently scaled in $\mathcal{P}_k$.  
     \item[$\epsilon=0$)] Observe that if $C_2$ contains an element of the form $(0,x)$ with $x<a$ we can define $\widetilde{C}_2$ as above an produce a $3$-simplex $\rho:C_0 \to C_1 \to \widetilde{C}_2 \to C_2$ which shows that we can scale $\sigma=d_2(\rho)$ if and only if we can scale $d_3(\rho)$. In a totally analogous way as in the case $\epsilon=1$ we can assume that $a=1$. If $C_2$ does not contain any element of the form $(0,z)$ with $z>1$ then $\sigma$ must be already scaled. Moreover, we can assume without loss of generality that $C_2$ only contains one element of the form $(0,z)$ using a similar argument as before by constructing a certain $\widetilde{C}_2$. We can assume that in this case $z=2$ since otherwise, the simplex factors through a certain morphism $\alpha_j:A^{k-1}\to A^k$. If $C_2$ does not contain an element of the form $(1,s)$ with $s\neq k$ it follows by direct inspection that $\sigma$ is already scaled. If this is not the case we consider $D$ which is obtained from $C_2$ by discarding every element of the form $(1,s)$ with $s \neq k$. Then we get a 3-simplex $\Xi: C_0 \to C_1 \to D \to C_2$. One easily checks that every face of $\Xi$ is scaled except possibly $d_2(\Xi)=\sigma$ and thus our result follows.
   \end{itemize}
\end{proof}

\begin{corollary}\label{cor:maxlean}
  Let $(S,T_S)$ be a scaled simplicial set and consider a weak equivalence $u:(A,E_A) \to (B,E_B)$ in $\left( \on{Set}_\Delta^+\right)_{/S}$. Then the functor $\SSt_S$ sends the morphism $(A,E_A,T_A \subset \sharp) \to (B,E_B,T_B \subset \sharp)$ to a weak equivalence in $\on{Fun}(\mathfrak{C}^{\mathbf{sc}}[S],\on{Set}_\Delta^\mathbf{ms})$.
\end{corollary}

\begin{remark}\label{rem:reductiontoPi}
  Let $i:A \to B$ be a cofibration of $\bS$ simplicial sets which we view as a morphism in $\left(\on{Set}_\Delta^{\mathbf{mb}}\right)_{/B}$ and recall the definition of $\Pi_B$ in \autoref{def:Pi_S}. We define $\Pi_B(A)^{\uparrow}$ as the pushout
  \[
    \begin{tikzcd}
      j\circ \mathfrak{C}^{\mathbf{sc}}[A_{\sharp}] \times \Delta^{\{0\}} \arrow[r] \arrow[d] &  j\circ \mathfrak{C}^{\mathbf{sc}}[B_{\sharp}] \times \Delta^{\{0\}} \arrow[d] \\
      \Pi_B(A) \arrow[r] & \Pi_B(A)^{\uparrow}
    \end{tikzcd}
  \]
  It follows from our definition that in order to check that $\SSt_B(A) \xRightarrow{} \SSt_B(B)$ is a pointwise weak equivalence it suffices to check that the induced map $\Pi_B(A)^{\uparrow} \xrightarrow{} \Pi_B(B)$ is a weak equivalence of $\on{Set}_\Delta^{\mathbf{ms}}$-categories.
\end{remark}

\begin{remark}\label{rem:order}
 Let $\theta:\Delta^{n+1} \to \Delta^1 \times \Delta^n$ be a non-degenerate simplex and let $i \in \Delta^{n+1}$ be the biggest element such that $\theta(i)=(i,0)$. We will use the notation $\theta=\sigma_i$ and give an order in the set of non-degenerate simplices in $\Delta^1 \times \Delta^n$ of maximal dimension by declaring $\sigma_i < \sigma_j$ if $i<j$.
\end{remark}

\begin{definition}\label{def:Kn}
  Let $\mathbb{K}_n=\Pi_{\Delta^n_\flat}((\Delta^n,\flat,\flat))$ and let $K_n=\mathbb{K}_n((0,0),(1,n))$. For every $\sigma_i$ as in \autoref{rem:order} be define $K^i_n$ as the subposet (with the inherited decorations) of $K_n$ consisting in those chains whose elements are in the image of $\sigma_i$. Observe that $K^i_n$ is isomorphic as a simplicial set to $C_n=(\Delta^{1})^n$. 

  Given $0<s<n+1$ we define $d_s(K^j_n)$ as the simplicial set of $K^j_n$ consisting in those chains whose elements are in the image of $d_s(\sigma_j)$.
\end{definition}

\begin{lemma}\label{lem:keyscaling}
  Let $0<j<n$ and let $d_{j+1}(K^{j}_n) \subset K^j_n$ be the simplicial subset (with the induced decorations) consisting in those chains whose elements factor through $d_{j+1}(\sigma_j)$. We view  $\Delta^1 \times d_{j+1}(K^{j}_n) $ as marked-scaled simplicial set as follows:
  \begin{itemize}
     \item The marking consists in those marked edges in the Cartesian product together with the edges contained in $\Delta^1 \times \{e\}$ where $e: C_0 \to C_1$ is a marked edge in $K^j_n$ such that $C_0$ contains the element $(0,j)$.
     \item The scaling is given given by the usual scaling on the Cartesian product.
   \end{itemize}  
   Then we have an isomorphism of marked-scaled simplicial sets $\Delta^1 \times d_{j+1}(K^{j}_n)  \isom K^j_n$
 \end{lemma}
\begin{proof}
  We only need to check that the decorations agree on each side concide. To show the claim regarding the marking we note that the marked edges of $K_n^j$ always factor through a face of the cube $C_n$ (see \autoref{def:Kn}) and thus the conclusion follows after direct inspection.

 To see that the scaling of $K^j_n$ is simply given by the product scaling we consider a triangle $\sigma:C_0 \to C_1 \to C_2$ and we define $D_i=C_i \setminus \{(1,j)\}$ which yields another triangle $\varphi: D_0 \to D_1 \to D_2$. We claim that $\sigma$ is scaled in $K_n^j$ if and only if $\varphi$ is. Observe that since $\varphi$ lies always in $d_{j+1}(K^{j}_n)$ this will be enough to show the claim regarding the scaling.

  We set the notation $\pi_{0,n}(C_i)=S_i$ and $\pi_{0,n}(D_i)=T_i$. Then we have the following:
 \begin{enumerate}
  
   \item We have that that $S_i=T_i$ if $(0,j) \in C_i$ and $T_i=S_i \setminus \{j\}$ otherwise.
   \item We have $\min(S_i)=\min(T_i)$.
    \item Given $C \in d_{j+1}(K^j_n)$ and set  $V=\pi_{0,n}(C)$ then it follows that $\min(V)  \leq j$.
 \end{enumerate}
 Our final claim is that for $i<j$ we have that the subsets (see \autoref{rem:ust}) $U_{S_i,S_j}=U_{T_i,T_j}$. Observe that $j \in U_{S_i,S_j}$ if and only if $j \in \{\min(S_i), \min(S_j)\}$ or if $j \in S_i$ and $\min(S_i)<j < \min(S_j)$. However by 3) above this cannot be the case. The claim now follows.
\end{proof}

\begin{lemma}\label{lemma:filtration}
  Let $(A,E_A,T_A) \subset (B,E_B,T_B)$ be an inclusion of marked-scaled simplicial sets such that $E_A=E_B$. Suppose that there exists some vertex $v \in A$ with the following property:
  \begin{itemize}
    \item For every simplex $\sigma:\Delta^n \to B$ which does not factor through $A$ then $v$ is the final vertex of $\sigma$. 
  \end{itemize}
  Let $M_A$ (resp. $M_B$) be the collection of marked edges in the Cartesian product (of marked-scaled simplicial sets) $\Delta^1 \times A$ (resp. $\Delta^1 \times B$) together with the edge $\Delta^{1}\times \{v \}$. Then the induced morphism
  \[
    j:\Delta^1 \times A \coprod_{\Delta^{\{1\}} \times A} B \xlongrightarrow{} \Delta^1 \times B
  \]
  where both simplicial sets are equipped with the product scaling and the marking given by $M_A$ and $M_B$ respectively, is a trivial cofibration in $\on{Set}_\Delta^{\mathbf{ms}}$.
\end{lemma}
\begin{proof}
First let us assume that the claim holds for $E_A=E_B=\flat$. Then it follows that for a general marking the map $j$ is obtained as a pushout of the map $j_{\flat}$ where the later map is the inclussion associated to the minimal marking. This shows that the general result will follow.

  Working simplex by simplex we can reduce the problem to the cases
  \begin{itemize}
    \item[i)] $(\partial \Delta^n,\flat,\flat) \to (\Delta^n,\flat,\flat)$ for $n\geq 1$.
    \item[ii)] $(\Delta^2,\flat,\flat) \to (\Delta^2,\flat,\sharp)$. 
  \end{itemize}
  where $v$ is given by the final vertex.

  To check $i)$ we use a totally analogous argument to that of Lemma 3.5.12 in \cite{LurieGoodwillie} which tells us that in this case that the map $j$ above is in the weakly saturated class of morphisms of type \ref{MS:innerhorn} in \autoref{def:msanodyne} and of type

     \begin{itemize}
       \item[*)]  $(\Lambda^n_n,\Delta^{\{n-1,n\}},\Delta^{\{0,n-1,n\}}) \xlongrightarrow{} (\Delta^n,\Delta^{\{n-1,n\}},\Delta^{\{0,n-1,n\}})$ 
     \end{itemize}
   
  and thus the claim holds. To prove $ii)$ we note that we can scale the remaining simplices using pushouts along morphisms of type \ref{ms:wonky4} in \autoref{def:msanodyne} together with the morphism
  \begin{itemize}
    \item[$\diamond$)] $(\Delta^3,\Delta^{\{n-1,n\}},U_3) \xlongrightarrow{} (\Delta^3,\Delta^{\{n-1,n\}},\sharp)$ 
  \end{itemize}
  where $U_3$ is the collection of all triangles except $\Delta^{\{0,1,2\}}$.
\end{proof}

\begin{lemma}\label{lem:stinnerhorn}
  Let $i:A \to B$ be a morphism of type \ref{mb:innerhorn} in \autoref{def:mbsanodyne}. Then the induced natural transformation $\SSt_B(A) \xRightarrow{} \SSt_B(B)$ is a pointwise weak equivalence.
\end{lemma}
\begin{proof}
  Since $\SSt_B$ preserves colimits and cofibrations it will be enough to prove the claim in the specific case where $A=(\Lambda^n_i,\flat,\flat \subset \Delta^{\{i-1,i,i+1\}})$ and $B=(\Delta^n,\flat,\flat \subset \Delta^{\{i-1,i,i+1\}})$.  We will show according to \autoref{rem:reductiontoPi} that the map $\Pi_B(A)^{\uparrow} \to \Pi_B(B)$ is an equivalence of $\on{Set}_\Delta^{\mathbf{ms}}$-categories.

  We observe that the induced morphism of marked-scaled simplicial sets $\Pi_B(A)^{\uparrow}(x,y) \to \Pi_B(B)(x,y)$ is an isomorphism except when $x=(0,0)$ and $y=(1,n)$. Recall the definition of $K_n$ in \autoref{def:Kn} an equipp this marked-scaled simplicial set with the decorations induced from $\Pi_B((0,0),(1,n))$. We similarly define $\Lambda^n_i K=\Pi_B(A)^{\uparrow}((0,0),(1,n))$. We define a filtration 
  \[
    \Lambda^n_i K=A_{-1}\xlongrightarrow{} A_0 \xlongrightarrow{} \cdots \xlongrightarrow{}A_{n-1} \xlongrightarrow{} A_n=K_n
  \]
  where $A_s$ the subsimplicial set of $K_n$ containing every simplex which factors through $K_n^{j}$ for $j\leq s$ (see \autoref{def:Kn} for a definition of $K_n^j$). We will show that each of the step in this filtration is a weak equivalence. The case $n=2$ follows easily by a direct computation. For the rest of the proof we will assume that $n\geq 3$.

   For $0\leq j\leq n$ we consider the pullback-pushout diagram
  \[
    \begin{tikzcd}
      Q_n^j \arrow[r] \arrow[d] & K^j_n \arrow[d] \\
      A_{j-1} \arrow[r] & A_j
    \end{tikzcd}
  \]
  We will show that the top horizontal morphism is a trivial cofibration. First we consider the case $0\leq j <n$. We describe $Q_n^j$ as a simplicial subset of $K^j_n$ which contains every face of the cube $C_n$, except those that factor through $d_{\alpha}(K^j_n)$ for $\alpha \notin \Phi(i)$ where
  \[
    \Phi(i)=\begin{cases}
      \{j+1,i+1\}\enspace, \enspace \text{ if } j<i, \\
      \{j+1\}\enspace, \enspace \text{ if } j=i ,\\
      \{j+1,i\}\enspace, \enspace \text{ if }j>i.
    \end{cases}
  \]
  We produce a 2-step filtration $Q_n^j \to Z_n^j \to K_n^j$ where $Z_n^j$ is obtained from $Q_n^j$ by attaching the simplices in $d_{\beta}(K^j_n)$ where $\beta \neq  j+1$ and $\beta \in \Phi(i)$.

   We observe that for $n\geq 3$ we have that every marked edge in $d_{j+1}(K^j_n)$ factors through $Z_n^j$. Therefore we can use \autoref{lemma:filtration} with $B=d_{j+1}(K^j_n)$ and  $A=(Z_n^j)_{|d_{j+1}(K^j_n)}$ to obtain a trivial cofibration of marked-scaled simplicial sets
   \[
     \varphi:\Delta^1 \times (Z_n^j)_{|d_{j+1}(K^j_n)} \xlongrightarrow{} \Delta^1 \times d_{j+1}(K^j_n).
   \]
  As a consequence of \autoref{lem:keyscaling} we see that the scaling of $\Delta^1 \times A$ and the scaling of  $Z^j_n$ coincide except possible in those triangles coming from $\Delta^{\{i-1,i,i+1\}}$ and similarly for $\Delta^1 \times B$ and $K^j_n$. We further note that every marked edge of $K_n^j$ factors through $Z_n^j$. After direct inspection we observe that we can produce a pushot diagram
\[
  \begin{tikzcd}
    \Delta^1 \times (Z_n^j)_{|d_{j+1}(K^j_n)}  \arrow[r] \arrow[d] &  \Delta^1 \times d_{j+1}(K^j_n) \arrow[d] \\
    Z_n^j \arrow[r] & K_n^j.
  \end{tikzcd}
\]
Therefore we can add the remaining decorations via a pushout of $\varphi$ along a cofibration which shows that the last step in the filtration is a trivial cofibration. To show that $Q^j_n \to Z_n^j$ is a trivial cofibration we consider a pushout-pullback diagram
  \[
   \begin{tikzcd}
      Z^{u}_{n-1} \arrow[r] \arrow[d] & d_{\beta}(K^j_n)=K_{n-1}^{u} \arrow[d] \\
    Q^j_n \arrow[r] & Z^j_n
   \end{tikzcd}
  \]
  and conclude by the previous argument or by a direct computation if $n=3$.

  To finish the proof we will show that $Q^n_n \to K^n_n$ is a trivial cofibration. In this case $Q^n_n$ contains every simplex that factors through a face of $C_n$ except those factoring through $d_i(K^n_n)$. For every $k \in [n]$ we define a chain
  \[
    D_k=(0,0)<(0,1)<\cdots <(0,k)<(1,n).
  \]
  We observe that using morphisms of type \ref{ms:wonky4} in \autoref{def:msanodyne} and morphisms of type $\diamond)$ as in the proof of \autoref{lemma:filtration} we can scale the triangle $D_{i-1} \subset D_{i} \subset D_{i+1}$ in $A_{n-1}$ and in $A_n$. Recall that for every chain $C=\{(i_\alpha,j_\alpha)\}_{\alpha \in A}$ we defined $m_C$ to be the biggest element in $A$ such that $i_{m_C}=0$. We can use this parameter to define a map
  \[
   r_n: K^n_n \xlongrightarrow{} (\Delta^n,\flat,\Delta^{\{i-1,i,i+1\}}), \enspace C \mapsto j_{m_C}
  \]
Moreover $r_n$ admits a section $s_n$ which sends $j$ to $D_j$ as defined above. It follows that there exists a marked homotopy between $s_n \circ r_n$ and the identity map on $K^n_n$. Furthermore, $r_n$ restricts to a map $\hat{r}_n:Q^n_n \to \Lambda^n_i$. One checks that since $Q^n_n$ and $\Lambda^n_i$ can be expressed as  iterated pushouts along cofibrations indexed by (n-1)-dimensional faces of $Q^n_n$ and the map $r_n$ restricts to an equivalence in each of its faces, that $\hat{r}_n$ is also a weak equivalence. We conclude that we have a commutative diagram
\[
  \begin{tikzcd}
    Q^n_n \arrow[d,"\simeq"] \arrow[r] & K^n_n \arrow[d,"\simeq"] \\
    (\Lambda^n_i,\flat,\Delta^{\{i-1,i,i+1\}}) \arrow[r,"\simeq"] & (\Delta^n,\flat,\Delta^{\{i-1,i,i+1\}}) 
   \end{tikzcd}
\]
and so the result follows by 2-out-of-3.
\end{proof}

\begin{lemma}\label{lem:0horn}
  Let $i:A \to B$ be a morphism of type \ref{mb:2coCartesianmorphs} in \autoref{def:mbsanodyne}. Then the induced natural transformation $\SSt_B(A) \xRightarrow{} \SSt_B(B)$ is a pointwise weak equivalence.
\end{lemma}
\begin{proof}
  The proof will mirror the strategy of the previous lemma. Again, we observe that the induced morphism of mapping simplicial sets $\Pi_B(A)^{\uparrow}(x,y) \to \Pi_B(B)(x,y)$ is an isomorphism except when $x=(0,0)$ and $y=(1,n)$ or when $x=(0,1)$ and $y=(1,n)$. However we observe that since the edge $(0,0) \to (0,1)$ will be collapsed in order to define the value of the functor $\SSt_B$ it will suffice to construct the analogous filtration 
  \[
    \Lambda^n_0 K=A_{-1}\xlongrightarrow{} A_0 \xlongrightarrow{} \cdots \xlongrightarrow{}A_{n-1} \xlongrightarrow{} A_n=K_n
  \]
  and show that each step is a trivial cofibration. As before, we will leave the case $n=2$ as an easy exercise and focus our attention to the cases $n\geq 3$. Note that in this case, we decorations coming from the marked edge $0 \to 1$ and the thin triangle $0 \to 1 \to n$ are already contained in $\Lambda^n_0 K$.

   For $0\leq j <n$ we consider pullback-pushout diagrams
  \[
    \begin{tikzcd}
      Q_n^j \arrow[r] \arrow[d] & K^j_n \arrow[d] \\
      A_{j-1} \arrow[r] & A_j
    \end{tikzcd}
  \]
  where $Q_n^j$ is the simplicial subset of $K^j_n$ which contains every face of $C_n$ except the $d_{j+1}(K_n^j)$ and if $j>0$ the face $C_n^0$ consisting in those chains that have the element $(0,1)$. The proof at this point is totally analogous to the proof of \autoref{lem:stinnerhorn}. We construct $Z_n^j$ by adding to $Q_n^j$ the face $C_n^0$ and conclude by \autoref{lemma:filtration} that each step in the filtration $Q_n^j \to Z_n^j \to K_n^j$ is given by a trivial cofibration.

  To show that $A_{n-1} \to A_n$ is a trivial cofibration we need to work a little bit harder. First we consider a commutative diagram where we are using circled arrows to represent marked morphisms
  \[
    \begin{tikzcd}
      {(0,0)<(1,0)<(1,n)} \arrow[r,circled]  & {(0,0)<(1,0)<(1,1)<(1,n)} \\
      {(0,0)<(1,n)} \arrow[u,circled] \arrow[r] \arrow[d] & {(0,0)<(1,1)<(1,n)} \arrow[u,circled] \arrow[d,circled] \\
      {(0,0)<(0,1)<(1,n)} \arrow[r,circled] & {(0,0)<(0,1)<(1,1)<(1,n)}.
    \end{tikzcd}
  \]
We note that every 2-simplex in this diagram is scaled and therefore we can mark every morphism via a pushout along a trivial cofibration. Let us remark that we can markx this morphisms in both $A_{n-1}$ and $A_n$ since $n\geq 3$. Recall the definition of $D_i$ in \autoref{lem:stinnerhorn} and consider a 3-simplex
\[
 \rho_W: D_0 \subset D_1 \subset (0,0)<(0,1)<(0,n)<(1,n) \subset (0,0)<(0,1)<(0,n)<(1,n) \cup W
\]
where $W$ is any chain starting at $(0,1)$ and ending at $(0,n)$. It follows that every face of $\rho_W$ is scaled except possible $d_2(\rho_W)$. Therefore we might scale that face using a pushout along a morphism of type \ref{ms:wonky4} in \autoref{def:msanodyne}. Note that any possible $\rho_W$ factors through $A_{n-1}$ except in the case where $W$ is the maximal chain. A similar argument as in the previous lemma shows that we have a commutative diagram
\[
  \begin{tikzcd}
    A_{n-1} \arrow[r] \arrow[d,"\simeq"] & A_n \arrow[d,"\simeq"] \\
    (\Lambda^n_0,\{\Delta^{\{0,1\}}\},\{\Delta^{\{0,1,n\}}\}) \arrow[r,"\simeq"] & (\Delta^n,\{\Delta^{\{0,1\}}\},\{\Delta^{\{0,1,n\}}\})
  \end{tikzcd}
\]
and the claim follows from 2-out-of-3.
\end{proof}

\begin{lemma}\label{lem:wonky}
   Let $i:A \to B$ be a morphism of type \ref{ms:wonky4} in \autoref{ms:wonky4}. Then the induced natural transformation $\SSt_B(A) \xRightarrow{} \SSt_B(B)$ is a pointwise weak equivalence.
\end{lemma}
\begin{proof}
  Let $L=\Pi_B(A)^\uparrow((0,0),(1,4))$ and let $K_4=\Pi_B(B)((0,0),(1,4))$. The only thing we need to show is that the map $L \to K_4$ is a weak equivalence. Using the notation from the previous proofs it follows that the missing scaled simplices in $L$ are all contained in $K_4^4$. Therefore if we denote by $L^4_4$ the restriction of $K^4_4$ to $L$ we see that it will be enough to show that the induced map $L^4_4 \to K^4_4$ is a weak equivalence. However, one can easily construct a commutative diagram
  \[
    \begin{tikzcd}
      L^4_4 \arrow[d,"\simeq"] \arrow[r] & K^4_4 \arrow[d,"\simeq"] \\
      (\Delta^4,\flat,T) \arrow[r,"\simeq"] & (\Delta^4,\flat,T')
    \end{tikzcd}
  \]
  where vertical maps are weak equivalences and the scaling in the bottom horizontal map is that of \ref{mb:wonky4} in \autoref{def:mbsanodyne}.
\end{proof}

\begin{proposition}\label{prop:stanodyne}
   Let $(S,T_S)$ be a scaled simplicial set and let $i:A \to B$ be an $\bS$-anodyne morphism in $\left( \on{Set}_\Delta^{\mathbf{mb}}\right)_{/S}$. Then $\SSt_S(i)$ is a weak equivalence in  $\on{Fun}(\mathfrak{C}^{\mathbf{sc}}[S],\on{Set}_\Delta^{\mathbf{ms}})$.
\end{proposition}
\begin{proof}
  Using \autoref{prop:Stbasechange} we can assume that $S=B$. The rest of the proof will consist in verifying that the claim holds for each of the generators given in \autoref{def:mbsanodyne}. We proceed case by case:
  \begin{itemize}
    \item[\ref{mb:innerhorn}] This follows from \autoref{lem:stinnerhorn}.
    \item[\ref{mb:wonky4}] This follows from \autoref{lem:wonky}.
    \item[\ref{mb:2coCartesianmorphs}] This follows from \autoref{lem:0horn}.
    \item[\ref{mb:2CartliftsExist}] This follows from \autoref{cor:maxlean}.
    \item[\ref{mb:composeacrossthin}] This follows from \autoref{cor:maxlean}.
    \item[\ref{mb:coCartoverThin}] This follows by explicit verification.
    \item[\ref{mb:equivalences}] This follows from \autoref{cor:maxlean}.
  \end{itemize}
  The result follows.
\end{proof}

\begin{proposition}\label{prop:stLeftquillen}
  Let $(S,T_S)$ be a scaled simplicial set and let $\phi: \mathfrak{C}^{\mathbf{sc}}[S] \to \scr{C}$ be a functor of $\on{Set}_\Delta^+$-enriched categories. Then the straightening functor
  \[
     \SSt_\phi: \left( \on{Set}_\Delta^{\mathbf{mb}}\right)_{/S}\xlongrightarrow{} \on{Fun}(\scr{C},\on{Set}_\Delta^{\mathbf{ms}})
   \] 
   is a left Quillen functor.
\end{proposition}
\begin{proof}
  We will show that $\SSt_\phi$ preserves cofibrations and weak equivalences. First, we point out that due to \autoref{prop:Stbasechange} it will be enough to consider the case $\phi=\on{id}$. In this case, we saw in \autoref{lem:stcof} that our functor preserves cofibrations.

  To address the claim regarding weak equivalences, we see that \autoref{prop:stanodyne} implies $\SSt_S$ preserves $\bS$-anodyne morphisms. We can therefore  restrict our attention to showing that $\SSt_S$ preserves weak equivalences between fibrant objects. To this end it will enough to show the following:
  \begin{itemize}
      \item[*] Let $f,g: X \to Y$ be morphisms between fibrant objects such that $\SSt_S(f)$ is a weak equivalence. Then given a homotopy $H:X \times (\Delta^1)^\sharp  \to Y$ between $f$ and $g$ it follows that $\SSt_S(g)$ is also a weak equivalence.
  \end{itemize}
  The claim follows after noting that we have an anodyne morphism $X \times \Delta^{\{0\}} \to X \times (\Delta^1)^\sharp$ due to \autoref{prop:pushoutproduct} which implies that $\SSt_S(H)$ is a weak equivalence as well as the map induced by the projection onto $X$ $\SSt_S(p):  \SSt_S(X \times (\Delta^1)^\sharp) \to \SSt_S(X)$. 
\end{proof}
\subsection{Straightening over a point}
The goal of this section is to prove the following result.
\begin{proposition}\label{prop:stpoint}
The straightening-unstraightening adjunction over the point
\[
    \SSt_{*}:\left(\on{Set}_\Delta^{\mathbf{mb}}\right)_{/\Delta^0} \llra \on{Set}_\Delta^{\mathbf{ms}}: \UN_*
  \]  
  is a Quillen equivalence.
\end{proposition}
To do this we will construct a left Quillen equivalence 
\[
  L_*:\left(\on{Set}_\Delta^{\mathbf{mb}}\right)_{/\Delta^0} \xlongrightarrow{} \on{Set}_\Delta^{\mathbf{ms}}
\]
and a natural transformation $\alpha:\SSt_* \xRightarrow{} L$ which is pointwise a weak equivalence of marked-scaled simplicial sets.
\begin{proposition}
  Let $L_*:\left(\on{Set}_\Delta^{\mathbf{mb}}\right)_{/\Delta^0} \xlongrightarrow{} \on{Set}_\Delta^{\mathbf{ms}}$ be the functor that assigns to an $\bS$ simplicial set $(X,E_X,T_X \subseteq C_X)$ the marked scaled simplicial set $(X,E_X,C_X)$. Then $L_*$ is a left Quillen equivalence.
\end{proposition}
\begin{proof}
  The functor $L_*$ admits a right adjoint $R_*$ which is given by $R_*(X,E_X,T_X)=(X,E_X,T_X)$. We observe that $L_* \circ R_*=\on{id}$ and that the unit $\on{id}\xRightarrow{} R_* \circ L_*   $  is given by $(X,E_X,T_X \subset C_X) \to (X,E_X,C_X)$ which is in the weakly saturated class of morphisms of type \ref{mb:coCartoverThin} in \autoref{def:mbsanodyne}.

  To finish the proof we observe that $L_*$ preserves cofibrations and maps \bS-anodyne morphisms to $\mathbf{MS}$-anodyne morphisms. It is then easy to see that $L_*$ preserves weak equivalences between fibrant objects and the result follows.
\end{proof}
Recall the definition of the maps $\pi_{0,n}$ in \autoref{rem:maptopn}. Then postcomposing this map with the morphism $m_n: \mathcal{P}_n \to (\Delta^n,\flat)$ which assigns to every $S \in \mathcal{P}_n$ the value $m_n(S)=\min(S)$ we obtain a map of marked scaled simplicial sets
\[
  \alpha_n^{\flat}: \SSt_*((\Delta^n,\flat,\flat)) \xlongrightarrow{} (\Delta^n,\flat,\flat)
\]
One can easily produce marked variants of this maps $\alpha^{\sharp}_2$ and $\alpha_1^{\sharp}$ associated to the $\bS$ simplicial sets $(\Delta^2,\flat,\sharp)$ and $(\Delta^1,\sharp,\sharp)$. We would like to remark that $\SSt_*(\Delta^2,\flat,\flat \subset \sharp)=\SSt_*(\Delta^2,\flat, \sharp)$ which justifies why we are defining only one $\alpha_2^\sharp$. Since our definitions are functorial with respect to monotone morphisms $[k] \to [n]$ this collection of maps assemble into a natural transformation $\alpha:\SSt_* \xRightarrow{} L_*$.

\begin{proposition}
  The natural transformation $\alpha:\SSt_* \xRightarrow{} L_*$ is a pointwise weak equivalence.
\end{proposition}
\begin{proof}
  Since both functors are left Quillen and both model categories are left proper, a simplex by simplex argument shows that it will enough to show that the components $\alpha_n^{\flat}$, $\alpha_2^\sharp$ and $\alpha_1^\sharp$ are weak equivalences. We further observe $\alpha_2^\sharp$ and $\alpha_1^{\sharp}$ can be obtained from their undecorated countermarks via pushouts along cofibrations. This shows that we can restrict our attention to $\alpha_n^\flat$ for $n\geq 0$.

  Let $S=(\Delta^n,\flat)$ and consider $\Pi_S(\Delta^n,\flat,\flat)$. Denote by $\Phi_n$ the $\on{Set}_\Delta^{\mathbf{ms}}$-category obtained from  $\Pi_S(\Delta^n,\flat,\flat)$ by marking every edge in the mapping simplicial sets of the form  
  \[
    \Pi_S(\Delta^n,\flat,\flat)((1,a),(1,b)).
  \]
  We we will show that the map $\hat{\alpha}_n^\flat:K_n=\Phi_n((0,n),(1,n)) \xlongrightarrow{} (\Delta^n,\flat)$ is a weak equivalence. Note that $\alpha_n^\flat$ is obtained from $\hat{\alpha}_n^\flat$ after identifying certain simplices. It is easy to see that we can mark every edge in $K_n$ whose image under $\hat{\alpha}_n^\flat$ becomes degenerate using pushouts along $\bS$-anodyne morphisms. We consider a filtration
  \[
    A_{-1}=K^0_n \to A_1 \to A_2 \to \cdots \to A_{n-1} \to A_n=K_n
  \]
  where $A_{i}$ is obtained from $A_{i-1}$ by attaching those simplices contained in $K_n^{i}$ (see \autoref{def:Kn}) where $K_n^i$ has the decorations induced from $\Phi_n$. We further denote by $\overline{A}_i$ the image of $A_i$ under the collapse map in the definition of $\SSt_*((\Delta^n,\flat,\flat))(*)$ and similarly denote $\overline{K}^{i}_n$.

  We will show that the restriction of $\alpha_n^\flat$ to each $\overline{A}_i$ defines a weak equivalence
  \[
    \alpha_{n,i}^{\flat}:\overline{A}_i \xlongrightarrow{} \left(\Delta^{[0,i]},\flat,\flat\right).
  \]
  Since weak equivalences are stable under filtered colimits this will imply the result. Assume that $\alpha_{n,j}^\flat$ is a weak equivalence for $j\leq i-1$ and consider the pullback-pushout square
  \[
    \begin{tikzcd}
      Q_n^i \arrow[r] \arrow[d] & K_n^i \arrow[d] \\
      A_{i-1} \arrow[r] & A_i
    \end{tikzcd}
  \]
  Observe that $\hat{\alpha}_n^\flat$ induces a commutative diagram
  \[
 \begin{tikzcd}
      Q_n^i \arrow[r] \arrow[d] & K_n^i \arrow[d,"r_i"] \\
      \Delta^{[0,i-1]}\arrow[r] & \Delta^{[0,i]}
    \end{tikzcd}
  \]
  We claim that the vertical morphisms are weak equivalences. Note that we can define for $0\leq j \leq i$
  \[
     C_j=(0,0)<(0,1)\cdots<(0,j)<(1,i)<\cdots< (1,n)
   \] 
   which provides us with a section $s_i:\Delta^{[0,i]} \to K_n^i$ sending $j$ to $C_j$. One checks that $r_i \circ s_i=\on{id}$ and that there is a marked homotopy between the identity of $K_n^i$ and $s_i \circ r_i$. Moreover, the section $s_i$ and the homotopy restrict to $Q_n^i$. It is immediate to see that both the section an the homotopy can be factored through the quotient simplicial sets $\overline{K}_n^i$ and $\overline{Q}_n^i$ which shows that $\alpha_{n,i}^\flat$ is again a weak equivalence.
\end{proof}

\begin{corollary}\label{cor:unconservative}
  Let $S$ be a scaled simplicial set, and let $\phi:\mathfrak{C}^{\mathbf{sc}}[S] \to \scr{C}$ be an $\on{Set}^+_\Delta$-enriched functor. Assume that $\phi$ is essentially surjective, and let $\alpha: \scr{F} \to \scr{F}^{\prime}$ be a map between fibrant objects of $\on{Fun}(\scr{C},\on{Set}^{\mathbf{ms}}_\Delta)$. Then the following conditions are equivalent:
  \begin{enumerate}
    \item The map $\alpha$ is a weak equivalence in  $\left(\on{Set}^{\mathbf{ms}}_\Delta\right)^{\scr{C}}$.
    \item For every $C \in \scr{C}$ the induced map $\alpha_C:\scr{F}(C) \to \scr{F}^{\prime}(C)$ is a weak equivalence.
    \item For every vertex $s \in S$ the induced map on fibres
    \[
      \UN_\phi(\scr{F})_{s} \xlongrightarrow{} \UN_\phi(\scr{F}^\prime)_{s}
    \]
    is a bicategorical equivalence.
    \item The map $\UN_\phi(\alpha)$ is a weak equivalence in $\left(\on{Set}_\Delta^{\mathbf{mb}}\right)_{/S}$.
  \end{enumerate}
\end{corollary}
\begin{proof}
  The equivalence $1 \iff 2$ is immediate from the definition. Since both $\scr{F}$ and $\scr{F}^{\prime}$ are fibrant it follows from \autoref{prop:stLeftquillen} that $\UN_\phi(\alpha)$ is a map between fibrant objects in $\left(\on{Set}_\Delta^{\mathbf{mb}}\right)_{/S}$. Then the equivalence $3 \iff 4$ follows from $v)$ in \autoref{prop:fibrewisecriterion}. To finish the proof we need to show that $2 \iff 3$. However, the fact that $\phi$ is surjective allows us to reduce to the case where $S=\Delta^0$ and conclude by \autoref{rem:unbasechange} and \autoref{prop:stpoint}.
\end{proof}

\subsection{Straightening over a simplex}
In this section we establish the key element in the proof of our main theorem.
\begin{proposition}\label{prop:stsimplex}
Let $\Delta^n_\flat=(\Delta^n,\flat)$ be the minimally scaled $n$-simplex. Then the straightening-unstraightening adjunction over $\Delta^n_\flat$
\[
    \SSt_{\Delta^n}:\left(\on{Set}_\Delta^{\mathbf{mb}}\right)_{/\Delta^n_\flat} \llra \on{Fun}(\mathfrak{C}^{\mathbf{sc}}[\Delta^n_\flat],\on{Set}_\Delta^{\mathbf{ms}}): \UN_{\Delta^n}
  \]  
  is a Quillen equivalence.
\end{proposition}
Let us comment on the general structure of the proof before we dive into the details. In order to prove the result need to show that given an object $X \in \left(\on{Set}_\Delta^{\mathbf{mb}}\right)_{/\Delta^n_\flat}$ and an equivalence of $\on{Set}^+_\Delta$-enriched functors
\[
  \SSt_{\Delta^n}(X) \xlongrightarrow{\simeq} \scr{F}
\]
where $\scr{F}$ is a fibrant functor then the adjoint map $X \to \UN_{\Delta^n}(\scr{F})$ is a weak equivalence in $\left(\on{Set}_\Delta^{\mathbf{mb}}\right)_{/\Delta^n_\flat}$. Using 2-out-of-3 we can assume without loss of generality that $p:X \to \Delta^n$ is a fibrant object. Now, in order to check that our adjoint map is a weak equivalence it suffices by \autoref{prop:fibrewisecriterion} to check that the induced morphisms on fibres
\[
  \varphi_i:X_i \xlongrightarrow{\simeq} \UN_{\Delta^n}(\scr{F})_i
\]
are bicategorical equivalences for $0\leq i \leq n$. Let us assume that we have proved \autoref{prop:stsimplex} for $0\leq k \leq n-1$ and note the base case was already shown in \autoref{prop:stpoint}. We claim that for every $0\leq i \leq n-1$ the map $\varphi_i$ is a bicategorical equivalence. 

We consider the morphism $i: \Delta_\flat^{[0,n-1]}\to \Delta_\flat^{n}$ and denote by $\overline{X}$ the pullback of $X$ along $i$. Similarly we denote by $\overline{\scr{F}}=j^*\scr{F}$ the restriction of $\scr{F}$ along $\mathfrak{C}^{\mathbf{sc}}[i]=j$. We observe the following:
\begin{enumerate}
  \item For every functor $\scr{G}:\mathfrak{C}^{\mathbf{sc}}[\Delta^{n}_\flat] \to \on{Set}_\Delta^{\mathbf{ms}}$ we have that $j_! (j^* \scr{G})(i) \to \scr{G}(i)$ is an isomorphism for $0\leq i < n$.
  \item For every $p:X \to \Delta^n$ we have that the map $\SSt_{\Delta^n_\flat}(\overline{X})(i) \to \SSt_{\Delta^n_\flat}(X)(i)$ is an isomorphism for $0\leq i<n$.
  \item Using \autoref{prop:Stbasechange} and \autoref{rem:unbasechange} we obtain a commutative diagram
  \[
    \begin{tikzcd}
      j_! \SSt_{\Delta^{n-1}_\flat}\overline{X} \arrow[r] \arrow[d] & j_! j^*(\scr{F}) \arrow[d] \\
       \SSt_{\Delta^n_\flat}X \arrow[r] & \scr{F}
    \end{tikzcd}
  \]
\end{enumerate}
We conclude that $\SSt_{\Delta^{n-1}_\flat}\overline{X} \to \overline{\scr{F}}$ is a weak equivalence. It follows from our induction hypothesis that $\overline{X} \to \UN_{\Delta^{n-1}_\flat}(\overline{\scr{F}})$ is a weak equivalence and thus the claim is proved. 

We have reduced our problem to showing that $\varphi_n$ is a weak equivalence. We claim that is enough to show the following.
\begin{itemize}
  \item[*)] The map $\SSt_{\Delta^n}(X_n)(n) \to \SSt_{\Delta^n}(X)(n)$ is a weak equivalence. 
\end{itemize}
Indeed, let $r^n:\Delta^0 \to \Delta^n_\flat$ denote the inclusion of the terminal vertex. Given $K \in \on{Set}_\Delta^{\mathbf{ms}}$ it follows that we have an isomorphism $ r^n_!(K)(n) \isom K$ which in turn shows  that the adjoint morphism to
\[
  r^n_!(\SSt_* X_n)\isom \SSt_{\Delta^n}(X_n) \xlongrightarrow{} \SSt_{\Delta^n}(X) \xlongrightarrow{} \scr{F}
\]
which is given by $\SSt_* X_n \to \scr{F}(n)$ is a weak equivalence. We can now use \autoref{prop:stpoint} to conclude that we have a bicategorical equivalence
\[
  X_n \xlongrightarrow{\simeq} \UN_{\Delta^n}(\scr{F})_n.
\]
Therefore, we will devote the rest of this section to the proof of the claim $*)$ above.

\begin{definition}\label{def:Oupslash}
  Let $I$ be a finite linearly ordered set and let $i \in I$. We define a poset $\mathcal{O}^{I}_{i\upslash}$ whose elements are given by subsets $S \subseteq I$ such that $S \neq \emptyset$ and such that $\min(S)=i$. We declare $S \leq T$ if $S \subseteq T$. We observe that we have a map
  \[
    \pi_I:\mathcal{O}^{I}_{i \upslash} \to \Delta^{I}, \enspace S \mapsto \max(S).
  \]
  We upgrade $\pi_I:\mathcal{O}^{I}_{i \upslash} \to \Delta^I$ to an object of $\left(\on{Set}_\Delta^{\mathbf{mb}}\right)_{/\Delta^{I}_\flat}$ as follows:
  \begin{itemize}
    \item We declare an edge $S \to T$ to be marked if $T=S \cup \max(T)$.
     \item We declare every triangle of $\mathcal{O}^{I}_{i \upslash}$ to be lean.
     \item We declare a triangle to be thin if its image in $\Delta^I$ is degenerate.
   \end{itemize} 
\end{definition}

\begin{definition}
  To ease the notation we set $\OO^n=\mathfrak{C}^{\mathbf{sc}}[\Delta^n_\flat]$ (see \autoref{def:OI}).
\end{definition}

\begin{remark}
  Let $I=[n]$. For every $i\leq j$ we view $\OO^n(i,j)$ as a $\bS$ simplicial set by declaring every triangle to be thin scaled and only degenerate edges to be marked. We further note that we have functors 
  \[
    \OO^n(i,j)\times \mathcal{O}^{n}_{j \upslash} \to \mathcal{O}^{n}_{i \upslash}; \enspace (S,T) \mapsto S \cup T,
  \]
  which preserve the decorations.
\end{remark}

\begin{definition}\label{def:mappingsimplex}
  Let $\scr{F}:\mathfrak{C}^{\mathbf{sc}}[\Delta^n_\flat] \to \on{Set}_\Delta^{\mathbf{ms}}$ be a $\on{Set}_\Delta^+$-enriched functor. We define a $\bS$ simplicial set $\scr{M}(\scr{F})$ over $\Delta^n_\flat$ as follows:
  \begin{enumerate}
     \item Let $i\in [n]$. We upgrade the marked-scaled simplicial set $\scr{F}(i)$ to a $\bS$ simplicial set by declaring the collection of thin triangles and lean triangles to coincide.
    \item We define $\scr{M}(\scr{F})$ as the coequalizer of the diagram in the category of $\bS$ simplicial sets
    \[
       \begin{tikzcd}
         \coprod\limits_{i<j }\scr{F}(i) \times \OO^n(i,j)\times \mathcal{O}^n_{j\upslash} \arrow[r,shift right=0.4ex] \arrow[r,shift left=0.4ex] & \coprod\limits_{i} \scr{F}(i) \times \mathcal{O}^n_{i\upslash}, \enspace \enspace i,j \in [n], 
       \end{tikzcd}
     \] 
     where the maps in the diagram are given by 
     \[
       \scr{F}(i)\times \OO^n(i,j)\times \mathcal{O}^n_{j\upslash} \xlongrightarrow{} \scr{F}(j)\times \mathcal{O}^n_{j\upslash}, \enspace  \enspace \scr{F}(i)\times \OO^n(i,j)\times \mathcal{O}^n_{j\upslash} \xlongrightarrow{} \scr{F}(i) \times  \mathcal{O}^n_{i\upslash}. 
     \]
   \item The maps $\scr{F}(i) \times \mathcal{O}^n_{i \upslash}\to \mathcal{O}^n_{i \upslash} \to \Delta^n$ assemble into a functor $\scr{M}(\scr{F}) \to \Delta^n$.
  \end{enumerate}
\end{definition}

\begin{remark}\label{rem:overlinef}
  We observe that given $0\leq i <n$ we have a map of simplicial sets
  \[
   \Xi^{i}: \Delta^1 \times \mathcal{O}^{n-1}_{i \upslash} \xlongrightarrow{}\mathcal{O}^n_{i\upslash}, \enspace \Xi^{i}(\epsilon,S)=\begin{cases}
       S, \enspace \text{ if }\epsilon=0,\\
       S \cup \{n\}, \enspace \text{ if }\epsilon=1
   \end{cases}
  \]   
  which use this map (actually an isomorphism) to equipp $\Delta^1 \times \mathcal{O}^{n-1}_{i \upslash}$ with the induced decorations from $\mathcal{O}^n_{i\upslash}$.

Given a functor $\scr{F}:\mathfrak{C}^{\mathbf{sc}}[\Delta^n_\flat] \to \on{Set}_\Delta^{\mathbf{ms}}$ let us denote by $\overline{\scr{F}}$ the restriction of $\scr{F}$ to $\mathfrak{C}^{\mathbf{sc}}[\Delta^{n-1}_\flat]$  along the map $i:\Delta^{[0,n-1]}_{\flat} \to \Delta^n_\flat$.  We can use the maps $\Xi^{i}$ to construct a map of simplicial sets
  \[
   \Xi_{\scr{F}}:\Delta^1 \times \scr{M}(\overline{\scr{F}}) \xlongrightarrow{}\scr{M}(\scr{F}).
  \]
Finally, we equipp $\Delta^1 \times \scr{M}(\overline{\scr{F}})$ with the decorations induced (via $\Xi_{\scr{F}}$) from $\scr{M}(\scr{F})$ and denote the resulting $\bS$ simplicial set by $ \Delta^1  \overline{\tensor} \scr{M}(\scr{F})$.
\end{remark}

\begin{remark}\label{rem:propertiesofM}
  The construction of $\scr{M}(\scr{F})$ defines a colimit-preserving functor
  \[
    \scr{M}(\mathblank): \on{Fun}(\mathfrak{C}^{\mathbf{sc}}[\Delta^n_\flat],\on{Set}_\Delta^{\mathbf{ms}}) \xlongrightarrow{}  \left(\on{Set}_\Delta^{\mathbf{mb}}\right)_{/\Delta^n_{\flat}}
  \]
  which enjoys several properties:
  \begin{itemize}
    \item[i)] For every $j \in [n]$ we have an isomorphism $\scr{M}(\scr{F})_j \isom \scr{F}(j)$.
    \item[ii)] Given a functor $\scr{F}:\mathfrak{C}^{\mathbf{sc}}[\Delta^n_\flat]$ let us denote by $\overline{\scr{F}}$ as in \autoref{rem:overlinef}. Then we have a pushout diagram
    \[
      \begin{tikzcd}
       \Delta^{\{1\}}  \times \scr{M}(\overline{\scr{F}}) \ \arrow[r] \arrow[d] & \scr{F}(n) \arrow[d] \\
       \Delta^1  \overline{\tensor} \scr{M}(\scr{F}) \arrow[r] & \scr{M}(\scr{F})
      \end{tikzcd}
    \]
     of $\bS$ simplicial sets.
     \item[iii)] The functor $\scr{M}(\mathblank)$ preserves cofibrations.
  \end{itemize}
 
\end{remark}

\begin{lemma}\label{lem:01laxgray}
  Let $A \to \Delta^n_\flat$ be an object of $\left(\on{Set}_\Delta^{\mathbf{mb}}\right)_{/\Delta^n_\flat}$ we define an $\bS$ simplicial set $\Delta^1 \overline{\tensor} A \to \Delta^{n+1}_\flat$ as follows:
  \begin{itemize}
    \item The underlying simplicial set is given by the Cartesian product.
    \item The projection map $\Delta^1 \times A \to \Delta^{n+1}_{\flat}$ is induced by the map
    \[
     r: \Delta^1 \times \Delta^n \xlongrightarrow{} \Delta^{n+1}, \enspace r(\epsilon,i)=\begin{cases}
       i, \enspace \text{ if }i=0,\\
       n, \enspace \text{ if }i=1.
     \end{cases}
    \]
    \item Let $(e_1,e_A) :\Delta^1 \to \Delta^1 \times A $ is marked if $e_A$ is marked in $A$ and $e_1=0 \to 0$ \emph{or} if $e_A$ is degenerate.
    \item A triangle is lean if and only if it is lean in $A$.
    \item A triangle is thin if and only if it is lean and its image in $\Delta^{n+1}$ is degenerate.
  \end{itemize}
  Then the map $\Delta^{\{0\}} \times A \to \Delta^1 \overline{\tensor} A$ is $\bS$-anodyne.
\end{lemma}
\begin{proof}
  The claim follows from an standard simplex by simplex argument as is left as an exercise to the reader.
\end{proof}

\begin{remark}\label{rem:mfOzero}
  Note that the scaling of $\Delta^1 \overline{\tensor} \scr{M}(\overline{\scr{F}})$ given in \autoref{rem:overlinef} is precisely that of \autoref{lem:01laxgray} so in particular we obtain an anodyne morphism $\scr{M}(\overline{\scr{F}}) \to \Delta^1 \overline{\tensor} \scr{M}(\overline{\scr{F}})$. Moreover, in the particular case where $\scr{F}$ is the correspresentable functor on the object $0 $ it follows $\scr{M}(\scr{F})=\mathcal{O}^{n}_{0 \upslash}$ so applying \autoref{lem:01laxgray} $n$ times we obtain an anodyne morphism
  \[
    \Delta^0 \xlongrightarrow{\simeq} \mathcal{O}^{n}_{0 \upslash}.
  \]
where the map above selects the subset $\{0\}$.
\end{remark}

\begin{definition}
  Let $\scr{F}:\mathfrak{C}^{\mathbf{sc}}[\Delta^n_\flat] \to \on{Set}_\Delta^{\mathbf{ms}}$ be a $\on{Set}_\Delta^+$-enriched functor. We define a scaled simplicial set $\mathbb{M}(\scr{F})$ whose thin triangles are precisely the lean triangles of $\scr{M}(\scr{F})$.
\end{definition}

\begin{lemma}\label{lem:mappingsimplexfibrewise}
  Let $\scr{F}:\mathfrak{C}^{\mathbf{sc}}[\Delta^n_\flat] \to \on{Set}_\Delta^{\mathbf{ms}}$ be a fibrant $\on{Set}_\Delta^+$-enriched functor. Given a $(0,1)$-fibration $p:X \to \Delta^n_\flat$ and a morphism over $\Delta^n_\flat$
  \[
    \alpha:\scr{M}(\scr{F}) \xlongrightarrow{} X
  \]
  such that for every $j \in [n]$ the map $\alpha$ induces bicategorical equivalences $\scr{F}(j) \isom X_j$. Then it follows that $\alpha$ is a weak equivalence in $\left(\on{Set}_\Delta^{\mathbf{mb}}\right)_{/\Delta^n_\flat}$.
\end{lemma}
\begin{proof}
  Since $p:X \to \Delta^n_\flat$ is a $(0,1)$-fibration it follows that we can construct an associated $\infty$-bicategory $\bcat{X}$ by declaring lean triangles to be scaled. We claim that is enough to show the following:
  \begin{itemize}
    \item[$\diamond$)] The associated map $\alpha: \mathbb{M}(\scr{F}) \to \bcat{X}$ is a bicategorical equivalence. 
  \end{itemize}
  Indeed, given a $(0,1)$-fibration $q: Z \to \Delta^n_\flat$ with associated $\infty$-bicategory $\bcat{Z}$ it follows from our claim that we have a bicategorical equivalence
  \[
    \phi:\on{Fun}^{\mathbf{sc}}(\bcat{X},\bcat{Z})\xlongrightarrow{\simeq}\on{Fun}^{\mathbf{sc}}(\mathbb{M}(\scr{F}),\bcat{Z}).
  \]
  Moreover, we obtain a commutative diagram
  \[
    \begin{tikzcd}
      \on{Map}^{\mathbf{sc}}_{\Delta^n}(\bcat{X},\bcat{Z}) \arrow[r,"\psi"] \arrow[d] &  \on{Map}^{\mathbf{sc}}_{\Delta^n}(\mathbb{M}(\scr{F}),\bcat{Z}) \arrow[d] \\
        \on{Fun}^{\mathbf{sc}}(\bcat{X},\bcat{Z}) \arrow[r,"\phi"] & \on{Fun}^{\mathbf{sc}}(\mathbb{M}(\scr{F}),\bcat{Z})
    \end{tikzcd}
  \]
  where $\on{Map}^{\mathbf{sc}}_{\Delta^n}(-,-)$ denotes the full subcategory on maps which preserve the marked edges and commute with the projection maps. We observe that higher simplices in the aforementioned scaled simplicial set are also compatible with the projection maps. Since a simplex in $\scr{M}(\scr{F})$ is thin if and only if it is lean and its image is thin in $\Delta^n_\flat$ we see that if $\psi$ is a bicategorical equivalence it will follow that $\alpha$ is a weak equivalence $\left(\on{Set}_\Delta^{\mathbf{mb}}\right)_{/\Delta^n_\flat}$. It is clear by construction that $\psi$ is fully faithful so it will suffice to show that it is essentially surjective.

  Give $u \in \on{Map}^{\mathbf{sc}}_{\Delta^n}(\mathbb{M}(\scr{F}),\bcat{Z})$ we can find some $v \in \on{Fun}^{\mathbf{sc}}(\bcat{X},\bcat{Z})$ such that $\phi(v) \isom u$. Consequently, it will be enough to show that $v$ factors through $ \on{Map}^{\mathbf{sc}}_{\Delta^n}(\bcat{X},\bcat{Z})$. Let $x \in X$ such that $p(x)=i$ and pick an equivalence $\alpha(y) \xrightarrow{\simeq} x$. We then see that 
  \[
    v(\alpha(y)) \isom u(y), \enspace \enspace v(\alpha(y))\isom v(x) \implies q(v(x))=i.
  \]
  Therefore, we can focus our attention into proving the statement  $\diamond)$ above. Let $x_i \in \mathbb{M}(\scr{F})$ be an object represented by a pair $(a_i,\{i\})$ in $\scr{F}(i)\times \mathcal{O}^n_{i\upslash}$. We consider a marked morphism $f:x_i \to \hat{x}_i$ given by $(a_i,\{i\}) \to (a_i,\{in\})$. Given an object $x_n$, lying over $n$, we claim the following:
  \begin{itemize}
    \item[$\star$)] Restriction along $f$ induces a weak equivalence of marked simplicial sets
    \[
      \mathfrak{C}^{\mathbf{sc}}[\mathbb{M}(\scr{F})](\hat{x}_i,x_n)\xlongrightarrow{\simeq}  \mathfrak{C}^{\mathbf{sc}}[\mathbb{M}(\scr{F})](x_i,x_n).
    \]
    
  \end{itemize}
  It is easy to see that $\diamond)$ follows from $\star)$ together with a routinary inductive argument.

    We observe that we have cofibrations of $\on{Set}_\Delta^{+}$-enriched categories
    \[
      \mathfrak{C}^{\mathbf{sc}}[\mathbb{M}(\overline{\scr{F}})] \xrightarrow{} \mathfrak{C}^{\mathbf{sc}}[\Delta^1 \times \mathbb{M}(\overline{\scr{F}})], \enspace  \enspace \enspace \mathfrak{C}^{\mathbf{sc}}[\mathbb{M}(\overline{\scr{F}})] \xrightarrow{} \mathfrak{C}^{\mathbf{sc}}[\Delta^1] \times  \mathfrak{C}^{\mathbf{sc}}[\mathbb{M}(\overline{\scr{F}})]
    \]
    and a diagram
    \[
      \begin{tikzcd}
         \mathfrak{C}^{\mathbf{sc}}[\mathbb{M}(\overline{\scr{F}})] \arrow[d] \arrow[r] & \mathfrak{C}^{\mathbf{sc}}[\Delta^1 \times \mathbb{M}(\overline{\scr{F}})] \arrow[r] \arrow[d] & \mathfrak{C}^{\mathbf{sc}}[\Delta^1] \times  \mathfrak{C}^{\mathbf{sc}}[\mathbb{M}(\overline{\scr{F}})] \arrow[d] \\
         \mathfrak{C}^{\mathbf{sc}}[\scr{F}(n)] \arrow[r] & \mathfrak{C}^{\mathbf{sc}}[\mathbb{M}(\scr{F})] \arrow[r] & \mathbb{P}(\scr{F}) 
      \end{tikzcd}
    \]
    where both squares are pushouts. It is not hard to see using the fact $\mathfrak{C}^{\mathbf{sc}}[-]$ is a left Quillen equivalence that the top right-most horizontal map is a weak equivalence. We conclude that the bottom right-most horizontal morphism is also a weak equivalence. It is easy to see by direct inspection that the analogous claim to $\star)$ holds for $\mathbb{P}(\scr{F})$ and thus the proof is finished.
\end{proof}

\begin{proposition}\label{prop:reducetoMF}
  Let $p:X \to \Delta^n_\flat$ be a $(0,1)$-fibration. Then there exists a projectively fibrant-cofibrant functor $\scr{F}:\mathfrak{C}^{\mathbf{sc}}[\Delta^n_\flat] \to \on{Set}_\Delta^{\mathbf{ms}}$ an a weak equivalence $\scr{M}(\scr{F}) \to X$ in $\left(\on{Set}_\Delta^{\mathbf{mb}}\right)_{/\Delta^n_\flat}$.
\end{proposition}
\begin{proof}
  Using \autoref{lem:mappingsimplexfibrewise} it will be enough to construct a map $\scr{M}(\scr{F}) \to X$ inducing categorical equivalences on fibres. We proceed using induction on $n$, the case $n=0$ being clear. Let us suppose that the claim holds for $n-1$ and let $i:\Delta^{[0,n-1]}_\flat \to \Delta^n_\flat$. We denote by $\overline{X}$ the restriction of $X$ along $i$. Using our induction hypothesis we obtain a projectively fibrant-cofibrant functor $\scr{\overline{F}}:\mathfrak{C}^{\mathbf{sc}}[\Delta^{n-1}_\flat] \to \on{Set}_\Delta^{\mathbf{ms}}$ and fibrewise equivalence $\scr{M}(\overline{\scr{F}}) \to \overline{X}$. We use \autoref{rem:mfOzero} to provide a solution to the lifting problem
  \[
    \begin{tikzcd}
      \scr{M}(\overline{\scr{F}}) \arrow[r] \arrow[d] & X \arrow[d,"p"] \\
      \Delta^1 \overline{\tensor} \scr{M}(\overline{\scr{F}}) \arrow[r] \arrow[ur,dotted] & \Delta^n
    \end{tikzcd}
  \]
  which provides us with a map $\scr{M}(\overline{\scr{F}}) \to X_n$. We factor this map as
  \[
    \scr{M}(\overline{\scr{F}}) \xlongrightarrow{u} \hat{X}_n \xlongrightarrow{v} X_n
  \]
  where $u$ is a cofibration and $v$ is a trivial fibration. Note that it follows that $\hat{X}_n$ is also an $\infty$-bicategory. We can use the map $u$ to extend $\overline{\scr{F}}$ to a fibrant-cofibrant functor $\scr{F}:\mathfrak{C}^{\mathbf{sc}}[\Delta^n_\flat] \to \on{Set}_\Delta^{\mathbf{ms}}$ such that $\scr{F}(n)=\hat{X}_n$. The claim now follows.
\end{proof}

\begin{proposition}
  Let $p:X \to \Delta^n_\flat$ be a $(0,1)$-fibration. Then there exists an equivalence of marked-scaled simplicial sets
  \[
    \SSt_{\Delta^n}(X_n) (n) \xlongrightarrow{\simeq} \SSt_{\Delta^n}(X) (n)
  \]
  where $X_n$ denotes the fibre over $n$ of $X$.
\end{proposition}
\begin{proof}
  We note that due to \autoref{prop:reducetoMF} we have a fibrant-cofibrant functor $\scr{F}:\mathfrak{C}^{\mathbf{sc}}[\Delta^n_\flat] \to \on{Set}_\Delta^{\mathbf{ms}}$ and a weak equivalence $\scr{M}(\scr{F}) \to X$. We will show that for every projectively cofibrant functor $\scr{G}$  the map
  \[
    \eta_{\scr{G}}:\SSt_{\Delta^n}(\scr{G}(n)) (n) \xlongrightarrow{\simeq} \SSt_{\Delta^n}(\scr{M}(\scr{G})) (n)
  \]
  is a weak equivalence of marked-scaled simplicial sets. We observe that we have a pair of functors
  \[
    L_i: \on{Fun}(\mathfrak{C}^{\mathbf{sc}}[\Delta^n_\flat],\on{Set}_\Delta^{\mathbf{ms}}) \xlongrightarrow{} \on{Set}_\Delta^{\mathbf{ms}}, \enspace i=1,2
  \]
  given by $L_1(\scr{G})=\SSt_{\Delta^n}(\scr{G}(n)) (n)$ and $L_2=\SSt_{\Delta^n}(\scr{M}(\scr{G})) (n)$ which preserve colimits and cofibrations together with a natural transformation $\eta: L_1 \xRightarrow{} L_2$. We say that a functor $\scr{G}$ is \emph{good} if $\eta_{\scr{G}}$ is a weak equivalence. To finish the proof we need to show that every cofibrant functor is good.

  For every $-1\leq j \leq n$ let $i_{j}: \Delta^{[0,j]}_\flat \to \Delta^n_\flat$ be the obvious inclusion with the convention $\Delta^{[0,-1]}=\emptyset$. Given a functor $\scr{G}$ we define $\scr{G}_j$ as the result of first restricting $\scr{G}$ along $\mathfrak{C}^{\mathbf{sc}}[i_j]=f^j$ and then applying the left Kan extension along $f^j$. We set $\scr{G}_{-1}$ to be the initial functor.  We further denote $r^j: \mathfrak{C}^{\mathbf{sc}}[\Delta^0] \to \mathfrak{C}^{\mathbf{sc}}[\Delta^n_\flat]$ the functor that picks the object $j$ for $0\leq j \leq n$.

  Note that given a projectively cofibrant functor $\scr{G}$ then it follows that the canonical map $\scr{G}_{j-1}(j) \to \scr{G}(j)$ is a cofibration for $0\leq j \leq n$. We further note that we we have a pushout diagram 
  \[
    \begin{tikzcd}
      r^j_! \scr{G}_{j-1}(j) \arrow[r] \arrow[d] & r^j_!\scr{G}(j) \arrow[d] \\
      \scr{G}_{j-1} \arrow[r] & \scr{G}_j
    \end{tikzcd}
  \]
  where $r^j_!$ denotes the left Kan extension functor algon $r^j$. Since the top horizontal map is a cofibration it follows that in order to show that $\scr{G}_j$ is good it is enough to show that $ r^j_! \scr{G}_{j-1}(j),  r^j_!\scr{G}(j)$ and $\scr{G}_{j-1}$ are good. We note the following
  \begin{itemize}
    \item If $j>0$ it follows that $r^j_!A (i)=\emptyset$ for $i< j$. In particular, it follows that $\scr{M}(r^j_!A)$ factors through  $\Delta^{[j,n]}_\flat$. We can now induction on $n$ to see that $r^j_! A$ is good for $j>0$.
  \end{itemize}
  Finally, we can use induction on $j$ to reduce our problem to show that for every $K \in \on{Set}_\Delta^{\mathbf{ms}}$ the functor 
  \[
    \underline{K}: \mathfrak{C}^{\mathbf{sc}}[\Delta^n_\flat] \xlongrightarrow{} \on{Set}_\Delta^{\mathbf{ms}}; \enspace i \mapsto K \times \OO^n(0,i)
  \]
  is good. Note that we can use further simplify our computation to the cases where $K=\Delta^k_\flat$ for $k\geq 0$, $\Delta^2_\sharp$ and $(\Delta^1)^\sharp$. We will only show the case $K=\Delta^k_\flat$ for $k \geq 0$ the other cases will follow by a totally analogous argument.

  We can identify $\SSt_\Delta(\Delta^k \times \mathcal{O}^{n}_{0\upslash})(n)$ as a quotient of the (decorated) poset of chains of $\Delta^1 \times \Delta^k \times \mathcal{O}^{n}_{0 \upslash}$ starting at $(0,0,\{0\})$ at ending at some element $(1,\ell,S)$ with $\max(S)=n$. We define a map 
  \[
    \psi:\SSt_\Delta(\Delta^{k}_\flat \times \mathcal{O}^{n}_{0\upslash})(n) \xlongrightarrow{} \Delta^k_\flat \times \OO^n(0,n); 
  \]
  by sending a chain $C=\{(\epsilon_i,k_i,S_i)\}_{i=0}^{i=\ell}$ to $(k_{i_{m_C}},S_{m_C} \cup \{\max(S_{j})\}_{j\geq m_C})$ where $m_C$ is the biggest index such that $\epsilon_i=0$. We consider the map $p_0:\Delta^k \to \Delta^n_\flat$ where $\Delta^k$ has the minimal decorations and where $p_0$ is constant on the vertex $0$ we can now look at the commutative diagram
  \[
   \begin{tikzcd}
    \SSt_{\Delta^n}(\Delta^k)(n) \arrow[dr,"u",swap] \arrow[r,"\varphi"] & \SSt_\Delta(\Delta^{k}_\flat \times \mathcal{O}^{n}_{0\upslash})(n)  \arrow[d,"\psi"] &  \SSt_{\Delta^n}(\Delta^k \times \OO^n(0,n))(n) \arrow[l,"\phi",swap] \arrow[dl,"v"] \\
   & \Delta^k_\flat \times \OO^n(0,n) &  
   \end{tikzcd}
  \]
and make the following observations:
\begin{enumerate}
  \item It follows from \autoref{lem:01laxgray}, \autoref{rem:mfOzero} and \autoref{prop:pushoutproduct} that $\varphi$ is a weak equivalence.
  \item The map $u: \SSt_{\Delta^n}(\Delta^k)(n) \simeq \SSt_*(\Delta^k)\times \OO^n(0,n) \xlongrightarrow{} \Delta^k_\flat \times \OO^n(0,n)$ can be identified with a product of the natural transformation $\alpha$ at $\Delta^k$ considered in the proof of \autoref{prop:stpoint} and the identity map on $\OO^n(0,n)$ and its a consequently a weak equivalence. It follows from $1$ that $\psi$ is also a weak equivalence.
  \item The map $v:\SSt_{\Delta^n}(\Delta^k \times \OO^n(0,n))(n) \isom \SSt_*(\Delta^k \times \OO^n(0,n)) \to \Delta^k_\flat \times \OO^n(0,n) $ can also be identified with the component of the natural transformation $\alpha$ and thus it is a weak equivalence. We conclude that $\phi$ is a weak equivalence.
\end{enumerate}
Our final claim is established. Therefore, \autoref{prop:stsimplex} is proved.
\end{proof}

\subsection{The main theorem}
Let $\on{Set}^{\mathbf{sc}}_\Delta$ be the category of scaled simplicial sets and observe that we have a pair of functors,
\[
  F_1,F_2: (\on{Set}^{\mathbf{sc}})^\op_\Delta \xlongrightarrow{} \on{Set}_\Delta^{+}\!\!-\on{Cat}, \enspace F_1(S)=\on{Fun}^{\operatorname{o}}(\mathfrak{C}^{\mathbf{sc}}[S],\on{Set}_\Delta^{\mathbf{ms}}), \enspace F_2(S)=\left(\on{Set}_\Delta^{\mathbf{mb}}\right)^{\operatorname{o}}_{/S}
\]
which values in the category of $\on{Set}_\Delta^+$-enriched categories and where the superscript “o“ denotes the full (enriched) subcategory on fibrant-cofibrant objects. Let us recall the reader that it then follows that for every $S \in \on{Set}_\Delta^{\mathbf{sc}}$ we have fibrant $\on{Set}_\Delta^{+}$-enriched categories $F_i(S)$ for $i=1,2$.

 We claim that  the unstraightening construction $\UN_{(\mathblank)}$ defines a natural transformation. In virtue of \autoref{rem:unbasechange} it will be enough to show that  for every scaled simplicial set $S$ we have that $\UN_{S}$ defines a $\on{Set}_\Delta^+$-enriched functor. Given a fibrant-cofibrant functors $\scr{F},\scr{G}:\mathfrak{C}^{\mathbf{sc}}[S] \to \on{Set}_\Delta^{\mathbf{ms}}$ let $\on{Nat}(\scr{F},\scr{G})$ denote the corresponding mapping marked simplicial set. Then a map $K \to \on{Nat}(\scr{F},\scr{G})$ is precisely the data of an enriched natural transformation $K \tensor \scr{F} \Rightarrow{} \scr{G} $ where 
 \[
   K \tensor \scr{F}(s)=K \times \scr{F}(s).
 \]
 We can consequently define a map $K \to \on{Map}_{S}(\UN_S(\scr{F}),\scr{G})$ as the composite
 \[
   K \times \UN_S(\scr{F}) \xlongrightarrow{} \on{Un}_*(K) \times \UN_S(\scr{F}) \isom \UN_S(K \times \scr{F}) \to \UN_S(\scr{G}).
 \]
 which shows that $\UN_S$ defines a $\on{Set}_\Delta^+$-enriched functor. The main goal of this section is to show that for every scaled simplicial set $S$, it follows that $\UN_S$ is an equivalence of $\on{Set}_\Delta^+$-enriched categories.

 \begin{proposition}\label{prop:quillenvsenriched}
   Let $S$ be a scaled simplicial set. Then the following are equivalent:
   \begin{itemize}
     \item[i)] The functor $\UN_S$ is a right Quillen equivalence.
     \item[ii)] The functor $\UN_S$ defines an equivalence of fibrant $\on{Set}_\Delta^+$-enriched categories after restriction to the full subcategories of fibrant-cofibrant objects. 
   \end{itemize}
 \end{proposition}
 \begin{proof}
  It follows from Proposition 3.1.10 in \cite{HTT} and our previous discussion that $ii) \implies i)$. To show that $i) \implies ii)$ we show that given a marked simplicial set $K$ then  $\UN_S$ induces an isomorphism in the homotopy category of $\on{Set}_\Delta^{+}$ between $[K,\on{Nat}(\scr{F},\scr{G})]\isom [K,\on{Map}_S(\UN_S(\scr{F},\UN_S(\scr{G}))]$. This follows from the chain of isomorphisms below
  \[
    [K,\on{Nat}(\scr{F},\scr{G})] \isom [K \tensor \scr{F},\scr{G}]^{\on{Fun}} \isom [\UN_*(K)\times \UN_S(\scr{F}),\UN_S(\scr{G})]^{\on{Fib}} \isom [K,\on{Map}_S(\UN_S(\scr{F}),\UN_S(\scr{G}))] 
  \]
  where the second isomorphism is a consequence of $i)$.
 \end{proof}

 \begin{remark}\label{rem:phiidsuffices}
   Let $S$ be an scaled simplicial set and $\phi:\mathfrak{C}^{\mathbf{sc}}[S] \to \scr{C}$ an equivalence of $\on{Set}_\Delta^+$-enriched categories. Observe that it follows from \autoref{prop:Stbasechange} that $\UN_\phi$ is a right Quillen equivalence if and only if $\UN_S$ is a right Quillen equivalence. Therefore for the rest of the section we will let $\phi=\on{id}$.
 \end{remark}

 \begin{corollary}\label{cor:maxscaled2simp}
  Let $\Delta^2_\sharp=(\Delta^2,\sharp)$ denote a maximally scaled 2-simplex. Then the straightening-unstraightening adjunction 
  \[
    \SSt_{\Delta^2_\sharp}:\left(\on{Set}_\Delta^{\mathbf{mb}}\right)_{/\Delta^2_\sharp} \llra \on{Fun}(\mathfrak{C}^{\mathbf{sc}}[\Delta^2_\sharp],\on{Set}_\Delta^{\mathbf{ms}}): \UN_{\Delta^2_\sharp}
  \]  
  is a Quillen equivalence.
\end{corollary}
\begin{proof}
Observe that we have a commutative diagram
\[
  \begin{tikzcd}
    \on{Fun}^{\on{o}}(\mathfrak{C}^{\mathbf{sc}}[\Delta^2_\sharp],\on{Set}_\Delta^{\mathbf{ms}}) \arrow[r,"\UN_{\Delta^2_\sharp}"] \arrow[d] & \left(\on{Set}_\Delta^{\mathbf{mb}}\right)^{\on{o}}_{/\Delta^2_\sharp}  \arrow[d] \\
     \on{Fun}^{\on{o}}(\mathfrak{C}^{\mathbf{sc}}[\Delta^2_\flat],\on{Set}_\Delta^{\mathbf{ms}}) \arrow[r,"\UN_{\Delta^2}"] & \left(\on{Set}_\Delta^{\mathbf{mb}}\right)^{\on{o}}_{/\Delta^2_\flat}
  \end{tikzcd}
\]
where the vertical maps are fully faithful functors. We conclude that $\UN_{\Delta^2_\sharp}$ is fully faithful. It follows from \autoref{prop:quillenvsenriched} that it will be enough to show that $\UN_{\Delta^2_\sharp}$ is essentially surjective. 

Let $p:X \to \Delta^2_\sharp$ be a fibrant object object and pick a fibrant-cofibrant functor $\scr{F}:\mathfrak{C}^{\mathbf{sc}}[\Delta^2_\flat] \to \on{Set}_\Delta^{\mathbf{ms}}$ such that $\UN_{\Delta^2}(\scr{F}) \isom X$. To finish the proof we need to show that $\scr{F}$ factors through $\mathfrak{C}^{\mathbf{sc}}[\Delta^2_\sharp]$.  Since $\UN_{\Delta^2}(\scr{F})$ is equivalent to $X$ it follows that the composition of local $(0,1)$-Cartesian edges in this fibration remains $(0,1)$-Cartesian. Direct inspection reveals that our functor must factor through $\mathfrak{C}^{\mathbf{sc}}[\Delta^2_\sharp]$ and thus the claim holds.
\end{proof}

\begin{remark}\label{rem:appendixHTT}
   It follows from Lemma A.3.6.17 and  Corollary A.3.6.18 in \cite{HTT} that $F_1$ sends homotopy colimits of scaled simplicial sets to homotopy limits of $\on{Set}_\Delta^+$-enriched categories. 
\end{remark}

\begin{lemma}\label{lem:f2coftofib}
  Let $f:S_0 \to S$ be a cofibration of scaled simplicial sets, then the functor
  \[
    f^*: \left(\on{Set}_\Delta^{\mathbf{mb}}\right)^{\on{o}}_{S} \xlongrightarrow{} \left(\on{Set}_\Delta^{\mathbf{mb}}\right)^{\on{o}}_{S_0}
  \]
  is a fibration of $\on{Set}_\Delta^{+}$-enriched categories.
\end{lemma}
\begin{proof}
  Since both $\on{Set}_\Delta^{+}$-enriched categories are fibrant it follows from Theorem A.3.2.24 in \cite{HTT} that it will enough to show the following:
  \begin{itemize}
    \item[*)] Given a pair of fibrant objects $X,Y \in \left(\on{Set}_\Delta^{\mathbf{mb}}\right)^{\on{o}}_{S}$ then the induced morphism on mapping $\infty$-categories
    \[
       \on{Map}_S^{\on{th}}(X,Y) \xlongrightarrow{} \on{Map}^{\on{th}}_{S_0}(f^*X,f^*Y)
     \] 
     is a fibration of marked simplicial sets.
  \end{itemize}
  More generally we consider a pair of adjoint lifting problems
  \[\begin{tikzcd}
  A & \on{Map}_{S}(X,Y) & {} & {} & {A \times X \coprod\limits_{A \times f^*X}B \times f^*X} && Y \\
  B & \on{Map}_{S_0}(f^*X,f^*Y) &&& {B \times X} && S
  \arrow[from=1-1, to=2-1]
  \arrow[from=1-1, to=1-2]
  \arrow[from=2-1, to=2-2]
  \arrow[from=1-2, to=2-2]
  \arrow[from=1-5, to=1-7]
  \arrow[from=1-5, to=2-5]
  \arrow[from=2-5, to=2-7]
  \arrow[from=1-7, to=2-7]
  \arrow[dashed, from=2-5, to=1-7]
  \arrow[dotted, from=2-1, to=1-2]
  \arrow[squiggly, tail reversed, from=1-3, to=1-4]
\end{tikzcd}\]
where $A \to B$ is \bS-anodyne. Since $f:S_0 \to S$ is a cofibration it follows that the canonical map $f^* X \to X$ is a cofibration so we can use \autoref{prop:pushoutproduct} to conclude that we can produce the desired solution to the lifting problem.
\end{proof}

\begin{theorem}\label{thm:mainun}
  Let $S$ be a scaled simplicial set. Then the functor $\UN_S$ induces an equivalence of $\on{Set}_\Delta^+$-enriched categories
  \[
    \UN_S: \on{Fun}^{\operatorname{o}}(\mathfrak{C}^{\mathbf{sc}}[S],\on{Set}_\Delta^{\mathbf{ms}}) \xlongrightarrow{} \left(\on{Set}_\Delta^{\mathbf{mb}}\right)^{\on{o}}_{/S}
  \]
\end{theorem}
\begin{proof}
  We say that a scaled simplicial set $S$ is good if the conclusion of the theorem holds. In virtue of \autoref{prop:quillenvsenriched} we know that \autoref{prop:stsimplex} shows that the scaled simplicial sets $(\Delta^n,\flat)$ are good for $n \geq 0$. Moreover, it follows from \autoref{cor:maxscaled2simp} that $(\Delta^2,\sharp)$ is also good.

  Recall that every scaled simplicial set $S$ can be expressed as a filtered colimit over the natural numbers
  \[
    S_0 \to S_1 \to  \cdots S_k \to \cdots
  \]
  such that each map $S_i \to S_{i+1}$ is a cofibration and such that $S_0$ is a disjoint union of points. Moreover, a simplex by simplex argument shows that each step in this filtration can be obtain via pushouts along cofibrations:
  \begin{itemize}
    \item $(\partial \Delta^n, \flat ) \to (\Delta^n,\flat)$ for $n\geq 0$.
    \item $(\Delta^2,\flat) \to (\Delta^2,\sharp)$. 
  \end{itemize}
  We saw in \autoref{rem:appendixHTT} that $F_1$ maps homotopy colimits to homotopy limits. We see that in order to finish the proof it will be enough to show that $F_2$ maps the colimits appearing in our filtration to homotopy limits. Using \autoref{lem:f2coftofib} we reduce the problem to verifying that $F_2$ maps those colimits to ordinary limits which follows from direct inspection. This finishes the proof.
\end{proof}

\begin{corollary}
  Let $S$ be a scaled simplicial set and let $\phi:\mathfrak{C}^{\mathbf{sc}}[S] \to \scr{C}$ be an equivalence of $\on{Set}_\Delta^{+}$-enriched categories. Then the straightening-unstraightening adjunction,
   \[
    \SSt_{\phi}:\left(\on{Set}_\Delta^{\mathbf{mb}}\right)_{/S} \llra \on{Fun}(\scr{C},\on{Set}_\Delta^{\mathbf{ms}}): \UN_{\phi}
  \]  
is a Quillen equivalence.
\end{corollary}

\begin{remark}\label{rem:rehash}
  Let $\on{N}^{\mathbf{sc}}$ be the right adjoint to $\mathfrak{C}^{\mathbf{sc}}$. Given a scaled simplicial set $S$  we define $\bcat{F}\!\on{ib}^{01}(S)=\on{N}^{\mathbf{sc}}(F_2(S))$. We also define $\bcat{B}\!\on{icat}_\infty=\on{N}^{\mathbf{sc}}((\on{Set}_\Delta^{\mathbf{ms}})^{\on{o}})$ and observe that an analogous discussion to that of \cite{AGS_CartII} shows that \autoref{thm:mainun} shows that we have an equivalence of $\infty$-bicategories
  \[
    \UN_S: \on{Fun}(S,\bcat{B}\!\on{icat}_\infty) \xlongrightarrow{} \bcat{F}\!\on{ib}^{01}(S).
  \]
\end{remark}

\begin{definition}
  Let $\bcat{S}=(S,T_S)$ be an $\infty$-bicategory and let $S_{M_S}=(S,M_S)$  (see \autoref{def:mstri}). We denote by $\bcat{L}\!\on{Fib}(\bcat{S})=\bcat{F}\!\on{ib}^{01}(S_{M_S})$, the $\infty$-bicategory of local $(0,1)$-fibrations over $\bcat{S}$. Similarly, given another $\infty$-bicategory $\bcat{D}$ we define $\on{Fun}^{\on{oplax}}(\bcat{S},\bcat{D})=\on{Fun}(S_{M_S},\bcat{D})$.
\end{definition}

\begin{remark}
  In \cite{GHL_Gray}, the definition $\on{Fun}^{\on{oplax}}(\bcat{S},\bcat{D})$ is proposed as a model for oplax unital functors. It is expected that this definition will agree to that of \cite{GaitsgoryRozenblyum} but this claim remains to this day, unproven.
\end{remark}

\begin{corollary}\label{cor:laxunitl}
   Let $\bcat{S}$ be an $\infty$-bicategory. Then the straightening-unstraightening adjuction associated to the scaled simplicial set $(S,M_S)$
   \[
    \SSt^{\on{oplax}}_{\bcat{S}}: \bcat{L}\!\on{Fib}(\bcat{S}) \llra \on{Fun}^{\on{oplax}}(\bcat{S},\bcat{B}\!\on{icat}_\infty): \UN^{\on{oplax}}_{\bcat{S}}
  \]
  yields an equivalence of $\infty$-bicategories between the $\infty$-bicategory of local $(0,1)$-fibrations over $\bcat{S}$ and the $\infty$-bicategory of oplax unital functors with values in $\infty$-bicategories.
\end{corollary}
\begin{proof}
  This follows immediately from \autoref{thm:mainun} and \autoref{rem:rehash}.
\end{proof}

\section{The \texorpdfstring{$\infty$-}-bicategorical Yoneda embedding}
In this section we will give a proof of the Yoneda lemma for $(\infty,2)$-categories. We would like to point out that this result has already appeared in the work of Hinich \cite{Hin} in the context of enriched $\infty$-category theory. Throughout this section we fix an $\infty$-bicategory $\bcat{C}$.

\begin{definition}
  Let $\bcat{C}$ be an $\infty$-bicategory and let $\mathbb{F}(\bcat{C})=\on{Fun}^{\on{gr}}(\Delta^1, \bcat{C})$. We observe that evaluation at $0$ and evaluation at $1$ induce maps $\on{ev}_i: \mathbb{F}(\bcat{C}) \to \bcat{C}$ for $i=0,1$. It follows from Theorem 2.2.6 in \cite{GHL_Cart} that evaluation at $0$ yields a $(1,0)$-fibration. One can similarly show that the map $\on{ev}_1$ defines a $(0,1)$-fibration.
\end{definition}

\begin{remark}
  Let us recall that morphism in $\mathbb{F}(\bcat{C})$ is $(0,1)$-Cartesian  if and only if it is map to an equivalence under $\on{ev}_0$. Similarly, a 2-morphism is Cartesian in $\mathbb{F}(\bcat{C})(x,y)$ if and only if its image in $\bcat{C}$ under $\on{ev}_0$ is an invertible.
\end{remark}

\begin{remark}\label{rem:freefib}
  Let $f:\bcat{A} \to \bcat{C}$ be a functor of $\infty$-bicategories and consider a pulblack diagram
  \[
    \begin{tikzcd}
      \mathbb{F}(\bcat{C})\times_{\bcat{C}}\bcat{A} \arrow[r] \arrow[d] & \mathbb{F}(\bcat{C}) \arrow[d,"\on{ev}_1"] \\
      \bcat{A}\arrow[r,"f"] & \bcat{C}
    \end{tikzcd}
  \]
  it follows that evaluation at $0$ induces a map $\mathbb{F}(\bcat{C})\times_{\bcat{C}}\bcat{A} \to \bcat{C}$ which is a $(1,0)$-fibration by Proposition 3.8 in \cite{AScof}. We observe that we have a commutative diagram
  \[
    \begin{tikzcd}
      \bcat{A} \arrow[rr] \arrow[dr,swap,"f"] &  &\mathbb{F}(\bcat{C})\times_{\bcat{C}}\bcat{A} \arrow[dl,"\on{ev}_0"] \\
      & \bcat{C} &
    \end{tikzcd}
  \]
where the horizontal morphism is induced by the map $\Delta^1 \tensor \bcat{A} \to \Delta^0 \tensor \bcat{A}\isom \bcat{A}$. Moreover, we see as a consequence of Theorem 3.17 in \cite{AScof} that for every $(1,0)$-fibration $\bcat{X} \to \bcat{C}$ we have a trivial fibration of $\infty$-bicategories
  \[
    \on{Map}_{\bcat{C}}(\mathbb{F}(\bcat{C})\times_{\bcat{C}}\bcat{A}, \bcat{X}) \xlongrightarrow{} \on{Map}_{\bcat{C}}(\bcat{A},\bcat{X})
  \]
  where we are commiting a slight abuse of notation by denoting $\on{Map}_{\bcat{C}}(-,-)$ the mapping $\infty$-bicategory with respect to $(1,0)$-fibrations.
\end{remark}

\begin{remark}\label{rem:nonfibrant}
 Let $p:(A,E_A,T_A \subset C_A) \to (\bcat{C},\sharp,T_{\bcat{C}}\subset \sharp)$ be a map of $\bS$ simplicial sets. Then regardless of the fibrancy of $A$ we can always consider $\mathbb{F}(\bcat{C})\times_{\bcat{C}} A$ (which might not be a $(1,0)$-fibration) whose decorations are given as follows:
 \begin{itemize}
    \item An edge is marked if is associated map $\Delta^1 \tensor \Delta^1 \to \bcat{C}$ factors through $\Delta^1 \times \Delta^1$ and the restriction to $\Delta^{1}\times \Delta^{\{1\}}$ is marked in $A$.
    \item A triangle is lean if its restriction $\Delta^{\{1\}} \tensor \Delta^2  $ is lean in $A$.
    \item A triangle is thin if it is lean and its image in $\bcat{C}$ is thin.
  \end{itemize} 
  However, the proofs in \cite[Theorem 3.17, Corollary 3.20]{AScof} are of entirely combinatural nature and under close inspection one sees that we always have an anodyne (with respect to the given decorations) morphism $A \to \mathbb{F}(\bcat{C})\times_{\bcat{C}} A$.
\end{remark}

Let us consider the $(0,1)$-fibration $\on{ev}_1:\mathbb{F}(\bcat{C}) \to \bcat{C}$ and observe that we have a commutative diagram over $\bcat{C}$,
 \[
    \begin{tikzcd}
     \mathbb{F}(\bcat{C}) \arrow[rr,"\on{ev}_0 \times \on{ev}_1"] \arrow[dr,swap,"\on{ev}_1"] &  & \bcat{C}\times \bcat{C} \arrow[dl,"\pi_2"] \\
      & \bcat{C} &
    \end{tikzcd}
  \]
  where $\pi_2$ is the projection onto the second factor. It follows from our definitions that the map $\on{ev}_0 \times \on{ev}_1$ can be seen as a map of $(0,1)$-fibrations where $\bcat{C}\times \bcat{C}$  is classified by the constant functor with value $\bcat{C}$. Let $\mathcal{F}$ be the functor classified by $\mathbb{F}(\bcat{C}) \to \bcat{C}$. We make the following observations:
  \begin{enumerate}
     \item For every $c \in \bcat{C}$ we have a functor $p_c:\mathcal{F}(c) \to \bcat{C}$ and for every morphism $u: c \to c'$ we have a commutative diagram $p_{c'}\circ \mathcal{F}(u)=p_c$.
     \item For every $c \in \bcat{C}$ it follows that we have a morphism
  \[
    \mathbb{F}(\bcat{C})\times_{\bcat{C}}* = \bcat{C}_{\upslash c} \xlongrightarrow{} \bcat{C}
  \]
  which is a $(1,0)$-fibration by \autoref{rem:freefib} and Furthermore, we note that for ever $c \in \bcat{C}$ it follows that we have a morphism
  \[
    \mathbb{F}(\bcat{C})\times_{\bcat{C}}* = \bcat{C}_{\upslash c} \xlongrightarrow{} \bcat{C}
  \]
  which is a $(1,0)$-fibration by \autoref{rem:freefib}. Careful inspection reveals that for every $u:c \to c'$ the induced morphism $u_*: \bcat{C}_{\upslash c} \to \bcat{C}_{\upslash c'}$ preserves $(1,0)$-Cartesian edges whose fibres are $\infty$-categories.
  \item Combining 1 and 2 we see that $p_c:\mathcal{F}(c) \to \bcat{C}$ is $(1,0)$ fibration and that $\mathcal{F}(u)$ is a functor of $(1,0)$-fibrations.
   \end{enumerate}  
   We conclude that $\mathcal{F}$ can be expressed as a composite
   \[
     \mathcal{F}: \bcat{C} \xlongrightarrow{} \bcat{F}\!\on{ib}^{10}(\bcat{C})^{\leq 1} \xlongrightarrow{} \bcat{B}\!\on{icat}_\infty
   \]
   where the second functor is the obvious projection and the superscript “$\leq 1$“, denotes the $\infty$-bicategory spanned by $(1,0)$-fibrations whose fibres are $\infty$-categories. Using a dual version of our main result or equivalently Corollary 3.90 in \cite{AGS_CartII}  we obtain a functor
  \[
    \mathcal{Y}: \bcat{C} \xlongrightarrow{} \on{Fun}(\bcat{C}^{\op},\bcat{C}\!\on{at}_\infty)
  \]
  which we call the bicategorical Yoneda embedding.

  \begin{proposition}\label{prop:I}
   There exists afunctor $\bcat{I}:  \bcat{F}\!\on{ib}^{10}(\bcat{C})^{\leq 1}  \to  \bcat{F}\!\on{ib}^{10}(\bcat{C})^{\leq 1}$ which sends a $(1,0)$-fibration $p:\scr{G} \to \bcat{C}$ (with $\infty$-categorical fibres) to the $(1,0)$-fibration $\bcat{I}(\scr{G})$ which is defined as the $\bS$ simplicial set characterised uniquely by the universal property
   \[
     \on{Map}_{\bcat{C}}(A,\bcat{I}(\scr{G})) \isom \on{Map}_{\bcat{C}}(\mathbb{F}(\bcat{C})\times_{\bcat{C}}A, \scr{G})
   \]
   where $A$ is an $\bS$ simplicial set and $\mathbb{F}(\bcat{C})\times_{\bcat{C}}A$ is defined as in \autoref{rem:nonfibrant}.
  \end{proposition}
  \begin{proof}
    There are two main things to prove: We need to show that $\bcat{I}(\scr{G}) \to \bcat{C}$ is a $(1,0)$-fibration whose fibres are $\infty$-categories and that the construction $\bcat{I}(-)$ is functorial. We will start by first proving the second assertion.

    Let $K \in \on{Set}_\Delta^{+}$ (which we view as having the maximal scaling) and consider a map $K \to \on{Map}_{\bcat{C}}(\scr{G},\scr{H})$. We will construct a morphism $K \to \on{Map}_{\bcat{C}}(\bcat{I}(\scr{G}),\bcat{I}(\scr{H}))$ as follows:

    Given a simplex $\Delta^n \xlongrightarrow{\sigma \times \overline{\varphi}}  \bcat{I}(\scr{G}) \times K$ we consider a morphism 
    \[
      \mathbb{F}(\bcat{C})\times_{\bcat{C}}\Delta^n \xrightarrow{f_{\overline{\sigma}}\times \varphi} \scr{G}\times K \xrightarrow{} \scr{H}
    \]
    where $\varphi$ is given by the composite $\mathbb{F}(\bcat{C})\times_{\bcat{C}}\Delta^n \to \Delta^{n} \xrightarrow{\overline{\varphi}} K$. This map is clearly compatible with the projection and functorial a thus yields a map $\bcat{I}(\scr{G})\times K \to \bcat{I}(\scr{H})$.
     
     To finish the proof we will show that $\bcat{I}(\scr{G}) \in \bcat{F}\!\on{ib}^{10}(\bcat{C})^{\leq 1}$. First let us observe that given $c \in \bcat{C}$ we have a canonical isomorphism,
     \[
       \bcat{I}(\scr{G})\times_{\bcat{C}}\{c\} \isom \on{Map}_{\bcat{C}}(\bcat{C}_{\upslash c}, \scr{G})
     \]
     which shows that the fibres of $\bcat{I}(\scr{G})$ are in fact $\infty$-categories. Let $i:A \to B$ be an anodyne morphism in the model structure for $(1,0)$-fibrations developed in \cite{AGS_CartI}. It follows from \autoref{rem:nonfibrant} that we have a commutative diagram
     \[
       \begin{tikzcd}
         A \arrow[r,"\simeq"] \arrow[d,"\simeq"] & \mathbb{F}(\bcat{C})\times_{\bcat{C}}A \arrow[d] \\
         B \arrow[r,"\simeq"] & \mathbb{F}(\bcat{C})\times_{\bcat{C}}B
       \end{tikzcd}
     \]
     where the horizontal morphisms are weak equivalences. It then follows by 2-out-of-3 that the right-most vertical morphism is also a weak equivalence. We conclude that $\bcat{I}(\scr{G})$ is a $(1,0)$-fibration.
  \end{proof}
  
  \begin{remark}
    By the previous proposition the fibration $\bcat{I}(\scr{G})$ corresponds under the straightening equivalence to a functor $\bcat{C}^\op \to \bcat{C}\!\on{at}_\infty$ mapping an object $c$ to
    \[
      \on{Map}_{\bcat{C}}(\bcat{C}_{\upslash c},\scr{G})\isom \on{Nat}_{\bcat{C}^\op}(\bcat{C}(\mathblank,c),\SSt_{\bcat{C}}(\scr{G}))
    \]
    where $\on{Nat}_{\bcat{C}^\op}(\mathblank,\mathblank)$ is the mapping $\infty$-category in  $\on{Fun}(\bcat{C}^{\op},\bcat{C}\!\on{at}_\infty)$. We will omit the explicit verification of the fact that there exists an equivalence of contravariant functors 
    \[
        \SSt_{\bcat{C}}(\bcat{I}(\scr{G})) \isom \on{Nat}_{\bcat{C}^\op}(\mathcal{Y}(\mathblank),\SSt_{\bcat{C}}(\scr{G})).
      \]  

  \end{remark}

  \begin{theorem}\label{thm:yoneda}
    For every $\infty$-bicategory $\bcat{C}$ the Yoneda embedding
    \[
      \mathcal{Y}:\bcat{C} \xlongrightarrow{} \on{Fun}(\bcat{C}^{\op},\bcat{C}\!\on{at}_\infty), \enspace c \mapsto \bcat{C}(\mathblank,c)
    \]
   is fully-faithful. Moreveover, given a functor $\scr{F}: \bcat{C}^\op \to \bcat{C}\!\on{at}_\infty$ there is a equivalence of $\bcat{C}\!\on{at}_\infty$-valued functors
    \[
      \on{Nat}_{\bcat{C}^{\op}}(\mathcal{Y}(\mathblank),\scr{F}) \xRightarrow{\simeq} \scr{F}
    \]
    which is natural in $\scr{F}$.
  \end{theorem}
  \begin{proof}
    For the proof it will be more convenient to work with (note the slight abuse of notation) the equivalent functor,
    \[
      \mathcal{Y}: \bcat{C} \to \bcat{F}\!\on{ib}^{10}(\bcat{C})^{\leq 1}. 
    \]
    By construction we have $\mathcal{Y}(c): \bcat{C}_{\upslash c} \to \bcat{C}$. It follows from  \cite[Theorem 3.17]{AScof}  applied to the functor $c:\Delta^0 \to \bcat{C}$ and Corollary 3.90  in \cite{AGS_CartII}  that $\mathcal{Y}(c)$ is classified under the Grothendieck construction by the representable functor on $c$.

    We observe that we have functors of $\infty$-categories
    \[
      \bcat{C}(c,c') \xlongrightarrow{} \on{Map}_{\bcat{C}}(\bcat{C}_{\upslash c},\bcat{C}_{\upslash c'}) \xlongrightarrow{\simeq}\on{Map}_{\bcat{C}}(c,\bcat{C}_{\upslash c'})\isom \bcat{C}(c,c')
    \]
    where the first map is induced by $\mathcal{Y}$ and the second is given by restriction along the map the selects the identity morphism in $\bcat{C}_{\upslash c}$. Note the later map is a weak equivalence as noted in \autoref{rem:freefib}.

    A dual argument to that \cite[Theorem 3.17]{AScof} shows that  the map $\bcat{C} \to \mathbb{F}(\bcat{C})$ is an equivalence in the model structure for $(0,1)$-fibrations. We can therefore, model the functor $\mathcal{Y}$ using $\SSt_{\bcat{C}}(\bcat{C})$.

    Let $\varphi_c:\Delta^0  \to \SSt_{\bcat{C}}(\bcat{C})(c)$ be the functor that selects the object $(0,c) \to (1,c)$ which is degenerate on the second component and consider the composite
    \[
     \Psi: \mathfrak{C}^{\mathbf{sc}}[\bcat{C}](c,c')\times \Delta^{0} \xlongrightarrow{\on{id}\times \varphi_c}  \mathfrak{C}^{\mathbf{sc}}[\bcat{C}](c,c') \times \SSt_{\bcat{C}}(\bcat{C})(c) \xlongrightarrow{} \SSt_{\bcat{C}}(\bcat{C})(c').
    \]
   Inspection reveals that $\Psi$ sends a morphism $f:c \to c'$ to the object of $\SSt_{\bcat{C}}(\bcat{C})(c')$ represented by $(0,c) \to (1,c) \to (1,c')$ and thus corresponds under the pertinent identifications to the object $f \in \bcat{C}_{\upslash c'}$.  Careful, analysis of $\Psi$ reveals that the map 
   \[
     \Phi:\bcat{C}(c,c') \xlongrightarrow{} \on{Map}_{\bcat{C}}(c,\bcat{C}_{\upslash c'})\isom \bcat{C}(c,c')
   \]
   is equivalent to the identity and thus a weak equivalence. We conclude that $\mathcal{Y}$ is fully faithful. 

   To prove the final claim we show that the functor $\bcat{I}$ in \autoref{prop:I} is naturally equivalent to the identity. We construct a natural transformation $\bcat{I} \xRightarrow{} \mathbb{1}$ to the identity functor which is induced by the canonical map
   $A \to \mathbb{F}(\bcat{C}) \times_{\bcat{C}} A$. We observe that the induced map of fibres
   \[
     \on{Map}_{\bcat{C}}(\bcat{C}_{\upslash c}, \scr{F}) \xlongrightarrow{\simeq} \scr{F}(c)
   \]
is given by restriction along the anodyne map $\Delta^0 \to \bcat{C}_{\upslash c}$ selecting the identity morphism and thus is a weak equivalence. 
  \end{proof}
\newpage
{}


\begin{thebibliography}{7}{}
\bibitem[AGH23]{AGH} F. Abell\'an, A. Gagna \& R. Haugseng \emph{Mates for $(\infty,2)$-categories} In preparation.
\bibitem[AG21]{AG_cof} F. Abell\'an Garc\'ia. \emph{Marked colimits and higher cofinality.} J. Homotopy Relat. Struct. (2021) DOI: \href{https://doi.org/10.1007/s40062-021-00296-2}{s40062-021-00296-2}
\bibitem[AGSI22]{AGS_CartI} F. Abell\'an Garc\'ia \& W.H. Stern. \emph{2-Cartesian fibrations I: A model for $\infty$-bicategories fibred in $\infty$-bicategories.} Appl Categor Struct, vol 30, pp. 1341–1392 (2022) \textsc{doi}:
\bibitem[ASII22]{AGS_CartII} F. Abell\'an \& W.H. Stern. \emph{2-Cartesian fibrations II: A Grothendieck construction for $\infty$-bicategories} \href{https://arxiv.org/abs/2201.09589}{2201.09589}
\bibitem[ASII22b]{AScof} F. Abell\'an \& W.H. Stern. \emph{On cofinal functors of $\infty$-bicategories} 
 \bibitem[Be21]{Berman} J. D. Berman. \emph{On lax limits in infinity categories}. arXiv: \href{https://arxiv.org/abs/2006.10851}{2006.10851}
\bibitem[Buc14]{Buckley} M. Buckley. \emph{Fibred 2-categories and bicategories.} J. Pure \& Applied Algebra. Vol 218, Issue 6. pp. 1034--1074 (2014)
\bibitem[CDW23]{CDW} M. Christ, T. Dyckerhoff, T. Walde \emph{Complexes of stable $\infty$-categories.} arXiv:\href{https://arxiv.org/abs/2301.02606}{2301.02606}
\bibitem[GHL21]{GHL_Gray}  A. Gagna, Y. Harpaz, E. Lanari. \emph{Gray tensor products and lax functors of $(\infty,2)$-categories.} Advances in Mathematics, vol. 391 (2021)  \textsc{doi}:
\href{https://doi.org/10.1016/j.aim.2021.107986}{j.aim.2021.107986} 
  \bibitem[GHL22]{GHL_Equivalence} A. Gagna, Y. Harpaz, E. Lanari. \emph{On the equivalence of all models for $(\infty,2)$-categories}. Math. Soc, Lond (2022) \textsc{doi}:
  \href{https://doi.org/10.1112/jlms.12614x}{jlms.12614}
  \bibitem[GHL22b]{GHL_LaxLim} A. Gagna, Y. Harpaz, E. Lanari. \emph{Fibrations and lax limits of $(\infty,2)$-categories.} arXiv: \href{https://arxiv.org/abs/2012.04537}{2012.04537}
  \bibitem[GHL22c]{GHL_Cart} A. Gagna, Y. Harpaz, E. Lanari.  \emph{Cartesian fibrations of $(\infty,2)$-categories.} arXiv: \href{https://arxiv.org/pdf/2107.12356.pdf}{2107.12356}
  \bibitem[GR17]{GaitsgoryRozenblyum} D. Gaitsgory and N. Rozenblyum. ``A Study in Derived Algebraic Geometry, Volumes I and II.'' Mathematical Surveys and Monographs
  Vol. 221. 2017. 969 pp. AMS
  \bibitem[Hin20]{Hin} V. Hinich, \emph{Yoneda lemma for enriched $\infty$-categories} Advances in Mathematics, vol. 367 (2020) \textsc{doi}:\href{https://doi.org/10.1016/j.aim.2020.107129}{j.aim.2020.107129}
  \bibitem[LNP22]{Global} S. Linskens, D. Nardin, L. Pol. arXiv:\emph{Global homotopy theory via partially lax limits.} \href{https://arxiv.org/abs/2206.01556}{2206.01556}
 \bibitem[Lur09]{HTT} J. Lurie. \enquote{Higher Topos Theory}. Princeton University Press, 2009. Available on \href{http://people.math.harvard.edu/~lurie/papers/highertopoi.pdf}{the author's webpage}
  \bibitem[Lur15]{HA} J. Lurie. \enquote{Higher Algebra}. 2017 {\em Available on} \href{http://people.math.harvard.edu/~lurie/papers/HA.pdf}{\em the author's webpage.} 
  \bibitem[Lur18]{LurieGoodwillie} J. Lurie. \emph{$(\infty,2)$-categories and the Goodwillie calculus.} arXiv: \href{https://arxiv.org/abs/0905.0462}{0905.0462}  

\end{thebibliography}
\end{document}